\documentclass[11pt]{amsart}

\makeatletter
\renewcommand\l@subsection{\@tocline{2}{0pt}{2pc}{5pc}{}}
\makeatother

\newcommand{\private}[1]{}   

\newcommand{\fn}[1]{\private{\footnote{{\bf Comment:} {#1}}}}
\newcommand{\pl}{{\bf Pascal:}}
\newcommand{\zz}[1]{}
\usepackage{amssymb, amsmath, amscd, amsthm, color, url, epsfig, amsfonts, latexsym, stmaryrd}
\usepackage{esint} 
\usepackage[all]{xy}          
\xyoption{dvips}              

\newcommand{\BR}{{\mathbb R}}
\newcommand{\BC}{{\mathbb C}}
\newcommand{\BZ}{{\mathbb Z}}
\newcommand{\BQ}{{\mathbb Q}}
\newcommand{\BK}{{\mathbb K}}\newcommand{\BN}{{\mathbb N}}

\newcommand{\calD}{{\mathcal{D}}}
\newcommand{\calA}{{\mathcal{A}}}
\newcommand{\calT}{{\mathcal{T}}}
\newcommand{\calH}{{\mathcal{H}}}

\newcommand{\calS}{{\mathcal{S}}}
\newcommand{\calN}{{\mathcal{N}}}
\newcommand{\calC}{{\mathcal{C}}}
\newcommand{\unit}{{\mathbf{1}}}
\newcommand{\id}{\mathrm{id}}
\newcommand{\spt}{\operatorname{spt}}
\newcommand{\ev}{\mathrm{ev}}
\newcommand{\fron}{\operatorname{fr}}
\newcommand{\map}{\operatorname{map}}
\newcommand{\fiber}{\operatorname{fiber}}
\newcommand{\CDGA}{\operatorname{CDGA}}
\newcommand{\Chains}{\operatorname{Ch}}

\newcommand{\SA}{\operatorname{SemiAlg}}
\newcommand{\Top}{\operatorname{Top}}
\newcommand{\sSet}{\operatorname{sSet}}
\newcommand{\mass}{\mathbf{M}}\newcommand{\flatnorm}{{\mathbf{F}}}
\newcommand{\Ch}{\operatorname{C}}\newcommand{\Ho}{\operatorname{H}}
\newcommand{\Coch}{\operatorname{C}}
\newcommand{\Norm}{\operatorname{N}}
\newcommand{\Chs}{\operatorname{C}^{str}}

\newcommand{\Apl}{{\operatorname{A}_{PL}}}

\newcommand{\APA}{{\operatorname{A}_{PA}}}
\newcommand{\Spa}{\operatorname{S}^{PA}}
\newcommand{\Sing}{{\operatorname{S}}^{\textrm{sing}}}

\newcommand{\inter}{\operatorname{int}}
\newcommand{\fib}[1]{\widehat{#1}}
\newcommand{\rank}{\operatorname{rank}}
\newcommand{\im}{\operatorname{im}}
\newcommand{\lbar}{\underline{l}}
\newcommand{\nbar}{\underline{n}}
\newcommand{\sgn}{\operatorname{sgn}}
\newcommand{\proj}{\operatorname{pr}}
\newcommand{\omin}{\Omega_{\min}}
\newcommand{\osmooth}{\Omega_{\calC^{\infty}}}
\newcommand{\ompa}{\Omega_{PA}}

\newcommand{\superset}{{\supset }}

\newcommand{\quism}{\stackrel{\simeq}{\to}}
\newcommand{\iso}{\stackrel{\cong}{\to}}
\newcommand{\toSA}{\stackrel{SA}{\to}}
\newcommand{\bbr}[1]{\llbracket{#1}\rrbracket}
\newcommand{\current}{chain}
\theoremstyle{plain}
\newtheorem{thm}{Theorem}[section]
\newtheorem{prop}[thm]{Proposition}
\newtheorem{lemma}[thm]{Lemma}
\newtheorem{cor}[thm]{Corollary}

\theoremstyle{definition}
\newtheorem{definition}[thm]{Definition}
\newtheorem{example}[thm]{Example}
\newtheorem{rem}[thm]{Remark}

\newcommand{\refS}[1]{Section~\ref{S:#1}}
\newcommand{\refT}[1]{Theorem~\ref{T:#1}}
\newcommand{\refC}[1]{Corollary~\ref{C:#1}}
\newcommand{\refP}[1]{Proposition~\ref{P:#1}}
\newcommand{\refD}[1]{Definition~\ref{D:#1}}
\newcommand{\refL}[1]{Lemma~\ref{L:#1}}
\newcommand{\refE}[1]{equation~$(\ref{E:#1})$}
\newcommand{\refN}[1]{~$(\ref{#1})$}

\begin{document}


\sloppy

\title{Real homotopy theory of semi-algebraic sets}


\author{Robert M. Hardt}
\address{Department of Mathematics, Rice University, Houston, TX 77251}
\email{hardt@math.rice.edu}
\urladdr{http://math.rice.edu/~hardt}

\author{Pascal Lambrechts}
\address{Institut de Recherche en Math\'{e}matique et Physique, 2 Chemin du Cyclotron, B-1348 Louvain-la-Neuve, Belgium}
\email{pascal.lambrechts@uclouvain.be}
\urladdr{http://milnor.math.ucl.ac.be/plwiki}

\author{Victor Turchin}
\address{Department of Mathematics, Kansas State University, Manhattan, KS 66506}
\email{turchin@ksu.edu}
\urladdr{http://www.math.ksu.edu/~turchin}

\author{Ismar Voli\'c}
\address{Wellesley College, Wellesley, MA 02482}
\email{ivolic@wellesley.edu}
\urladdr{http://palmer.wellesley.edu/~ivolic}

\thanks{The first, third, and fourth   authors were supported in part
  by the National Science Foundation grants DMS-0604605, DMS-0504390,
  and DMS-0968046. 
The second author is Ma\^\i tre de Recherche au F.N.R.S.}



\begin{abstract}
We complete the details of a theory outlined by Kontsevich and Soibelman
that associates to a semi-algebraic set a certain graded commutative differential algebra
of ``semi-algebraic differential forms'' in a functorial way.  This algebra encodes the real homotopy type of the semi-algebraic set in the spirit of the DeRham algebra of differential forms on a smooth manifold.  Its development is needed for
 Kontsevich's proof of the formality of the little cubes operad.
\end{abstract}

\maketitle

\zz{
                {\LARGE \textbf{NOT PUBLIC}: }the changes I made on mid january 2011are all commented by
                footnotes like this; it correspond to the comment
                \texttt{backslash zz}. Change the definition of this
                command to get public version.
                } 

{\tableofcontents}

\parskip=5pt
\parindent=0cm


\section{Introduction}


We start with a brief review of rational homotopy theory whose beginnings go back to Sullivan
\cite{Sul:inf} and whose full development can be found in \cite{BoGu:RHT} and \cite{FHT:RHT}.
To any topological space $X$, one can associate in a  functorial way a certain commutative differential graded algebra (or CDGA for short) $\Apl(X)$
with coefficients
in $\BQ$. The main feature of this algebra is that if $X$ is  nilpotent (e.g., simply-connected)
  with finite Betti numbers, then any CDGA $(A,d)$, where $d$ is the differential, that is quasi-isomorphic to $\Apl(X)$ contains all
 the information about the
 rational homotopy type of $X$. For example,
an analog of the DeRham theorem implies that
the rational singular cohomology algebra of $X$ is given by the homology of $(A,d)$, i.e.~$\Ho^*(X;\BQ)\cong \Ho(A,d).$
Another procedure  also recovers $\pi_*(X)\otimes\BQ$, the rational homotopy groups of $X$, from $(A,d)$ \cite{Sul:inf}.

  Sullivan's construction was inspired by DeRham theory which associates to a smooth manifold $M$ a commutative
 differential graded algebra of smooth differential forms $\osmooth^*(M)$ with coefficients in $\BR$.
 In fact, the DeRham and Sullivan constructions are related because $\osmooth^*(M)$ is quasi-isomorphic to
 $\Apl(M;\BR):=\Apl(M)\otimes_\BQ\BR$.
 One loses some information about the rational homotopy type of $M$ after tensoring with $\BR$ but what is left (the ``real homotopy type'') is still valuable.
 For example, one can still recover the real singular cohomology (this is the classical DeRham's theorem) or the dimension of the
 rational homotopy groups. In some applications, the DeRham CDGA of smooth forms is even more manageable than
 $\Apl(M;\BR)$ because it is closer to the geometry of the manifold. For instance, a variation of this approach is used to prove that compact Kahler manifolds are formal in
 \cite{DGMS:rea}.

 The goal of the present paper is to supply the details of a theory sketched by Kontsevich and Soibelman in \cite[Appendix 8]{KoSo:def}
which produces an analog of the DeRham differential algebra for semi-algebraic sets (subsets of $\BR^n$
 defined by finitely many polynomial equations, inequalities, and boolean operations).
 More precisely, following their approach, we will construct a functor
 \begin{align*}
 \ompa^*\colon \SA & \longrightarrow\CDGA \\
                 X & \longmapsto\ompa^*(X)
 \end{align*}
 where $\SA$ is the category of semi-algebraic sets and $\ompa^*(X)$ is the CDGA with real coefficients of 
 so-called \emph{PA forms}
 on the semi-algebraic set $X$.  Here ``PA" stands for ``piecewise semi-algebraic'' which is the terminology coined by Kontsevich and Soibelman for describing a wider class of spaces than we will consider in this paper; nevertheless, we will retain the ``PA" notation for the benefit of ease of translation between this paper and the Kontsevich-Soibelman one.
 Our main result is \refT{main} which we restate here.
 \begin{thm}\label{T:MainTheorem}
 For $X$ a compact semi-algebraic set, $\ompa^*(X)$ and $\Apl(X;\BR)$ are connected by a zigzag of natural
 transformations of CDGAs that induce isomorphisms on homology.
 \end{thm}
 This in particular implies that any CDGA quasi-isomorphic to
 $\ompa^*(X)$ contains all the information about the real homotopy type of a compact semi-algebraic set $X$.

 Our main motivation for supplying the details of the proof of this theorem and various results related to it is that it is an essential ingredient in
 Kontsevich's proof of the formality of the little cubes operad in \cite[\S 3]{Kon:OMDQ}.  Details of this formality proof will be given in \cite{LaVo:for}.

\subsection{Outline of the paper}

We build on \cite[Appendix 8]{KoSo:def}, which gives a good first idea of the theory. The roadmap of the present paper is as follows.

\textbullet\ \  In Section \ref{S:Background}, we review semi-algebraic sets and
 their stratifications; some PL topology and the semi-algebraic Hauptvermutung; and the notion of a current
 which is a convenient alternative to a singular chain in $\BR^n$.

\textbullet\ \ 
 In \refS{SAcur},
 we define semi-algebraic chains in terms of currents (this corresponds to what Kontsevich and Soibelman
 call chains in \cite{KoSo:def}). Briefly, a $k$-chain $\gamma\in\Ch_k(X)$ in a semi-algebraic set $X\subset\BR^n$ is represented by a linear combination of bounded semi-algebraic oriented smooth $k$-dimensional
 submanifolds $V_i$ in $\BR^n$
 whose closure is contained in $X$. By a remarkable property of semi-algebraic sets, the $k$-volume of the  $V_i$ is
 finite. Any two such linear combinations $\sum n_i\bbr{V_i}$ and $\sum n'_j\bbr{V'_j}$ are declared to be equal if,
 for any smooth form with compact support in $\BR^n$,
 $$\sum n_i\int_{V_i}\omega=\sum n'_j\int_{V'_j}\omega.$$
 In other words, two chains are equal if they are equal as
 currents.
 This defines a chain complex $\Ch_*(X)$ whose boundary $\partial$ is the classical boundary of currents.
 In order to prove that $\partial(\Ch_k(X))\subset\Ch_{k-1}(X)$ and that a semi-algebraic map $f\colon X\to Y$
 induces a chain map, we need to establish a generalization of Stokes' Theorem to semi-algebraic manifolds, which is \refT{Stokes}.
In this section, we also construct a natural cross product $\times\colon\Ch_*(X)\otimes\Ch_*(Y)\to\Ch_*(X\times Y)$.

\textbullet\ \ In \refS{conv}, we develop a suitable notion of convergence for sequences of chains. This will be important for showing
that some cochains are determined by their value on a dense subset, i.e.~at generic points.

\textbullet\ \  In \refS{PAForms}, we build the cochain complex of PA forms.  We begin this process in \refS{SACochains} by defining the cochain complex $\Ch^*(X)$ with coefficients in $\BR$
 as the  dual of $\Ch_*(X)$.
 The CDGA of PA forms that we are after will be a subcomplex $\ompa^*(X)\subset\Ch^*(X)$.
 Along the way, we first define a subcomplex of minimal forms $\omin^*(X)\subset \Ch^*(X)$ in \refS{MinimalForms} and show that it
 is equipped with a graded commutative cross product. The key property necessary for the proof of the existence of that product is the fact that a minimal form is smooth at a generic point, which is the content of \refP{minaesm}.

  The problem with $\omin^*(X)$ is that it does not satisfy the DeRham isomorphism. We thus need to extend
  this algebra in order to have the right homotopical properties.
  For this, we first define in \refS{scc} a notion of strongly continuous family of chains which allows us to  construct  the desired cochain complex of PA forms $\ompa^*(X)$ in Section \ref{SS:PAForms}, where we also exhibit an algebra structure on $\ompa^*(X)$ (the proof turns out to be more involved than for minimal forms).

\textbullet\ \  In \refS{equAplompa}, we prove \refT{MainTheorem}. The proof mainly follows the scheme of the analog result for the DeRham algebra of smooth manifolds.  The main difference is that our proof of the Poincar\'e Lemma is more complicated and it relies on the semi-algebraic Hauptvermutung.

\textbullet\ \  In \refS{MonEquiv}, we show that the weak equivalence established in the previous section
are monoidal, and we prove other   monoidal equivalences.  This is useful in applications to operads, which is our main motivation.\fn{\pl
10 july 2008. changed a bit}

\textbullet\ \ In \refS{SAbundle}, we study  integration along the fiber for semi-algebraic bundles
whose fibers are compact oriented manifolds.\fn{\pl new section}

\textbullet\ \  In \refS{diffKS}, we discuss some differences between this paper and
  \cite[Appendix 8]{KoSo:def} and explain why we were unable to prove some of the claims made there.

\subsection{Acknowledgments} We thank Thierry De Pauw for many hours of discussion about currents, for his invaluable
help in the proof of \refT{Stokes}, and for finding many results we needed in \cite{Fed:GMT}.
 We thank Yves F\'elix, Don Stanley, Maxim Kontsevich, and Yan Soibelman for fruitful conversations.
We warmly thank Greg Arone for his patience and encouragements to finish this paper. We also thank Michel Coste for help with semi-algebraic sets and Lou van den Dries for answering questions and
  pointing out the reference \cite{CLR:nat}.

  
\section{Background}\label{S:Background}



\subsection{Review of semi-algebraic sets and stratifications}


A short summary of semi-algebraic sets with definitions and most of the  properties that we will need can be found in
\cite[Section 1]{Har:loc},
and a complete reference is \cite{BCR:GAR}.
Briefly, a subset $X$ of $\BR^m$ is \emph{semi-algebraic} if it is a finite union of finite intersections of solution sets
of polynomial equations and  inequalities.
A \emph{semi-algebraic map} is a continuous map of semi-algebraic sets $X\subset \BR^m$ and
$Y\subset \BR^n$
whose graph is a semi-algebraic set in $\BR^m\times\BR^n$.  By a \emph{semi-algebraic function} we will
mean a semi-algebraic map with values in $\BR$. Note that in this paper we always assume
that  semi-algebraic maps and functions  are continuous. Semi-algebraic sets are clearly stable under taking 
 finite intersections, finite unions, and complements. Also, given a semi-algebraic set $X$, its closure $\overline X$ and its interior $\inter(X)$ are
 semi-algebraic.  The same is true for the image $f(X)$ or the preimage $f^{-1}(X)$ under a semi-algebraic map $f$ 
 \cite[p. 23]{BCR:GAR}.

\begin{definition}
A \emph{semi-algebraic manifold of dimension $k$} is a semi-algebraic set such that each point has
a semi-algebraic neighborhood semi-algebraically
homeomorphic to $\BR^k$ or to $\BR_{+}\times\BR^{k-1}$.
\end{definition}
Clearly a semi-algebraic manifold is a topological manifold with boundary,
that is, a semi-algebraic manifold of dimension $k-1$. We can therefore talk about the orientability and the orientation of a semi-algebraic manifold as well as the induced orientation on the boundary.

We will also consider smooth submanifolds of $\BR^n$, always without boundary except when stated otherwise.
 If $X$ is a semi-algebraic set in $\BR^n$, by a \emph{semi-algebraic smooth
submanifold in $X$} we mean a smooth submanifold of $\BR^n$ that is a semi-algebraic subset of $X$.

\begin{definition}
A semi-algebraic set $X\subset\BR^N$ is called \emph{bounded} if it is contained in a ball of
finite radius.
\end{definition}

A fundamental tool in the theory of semi-algebraic sets we will need is the existence
of certain type of finite partitions
into smooth submanifolds called \emph{stratifications}. By a \emph{smooth stratification} of a subset $X$ of $\BR^n$, we mean a finite
partition $\calS$ of $X$ such that
each $S\in\calS$ is a connected smooth submanifold and its relative closure
$X\cap\overline S$ is a union of elements of $\calS$,
one being $S$ and the others being of dimension $<\dim S$. Here an element
of $\calS$ is called a \emph{stratum}, and  $\calS$ is called a \emph{semi-algebraic
stratification} if each
of its strata is semi-algebraic.
Semi-algebraic sets and mappings admit smooth semi-algebraic stratifications.
More precisely, we have the following \cite[Proposition 9.1.8]{BCR:GAR}:

\begin{prop}
\label{P:strat}
Given a semi-algebraic set $X\subset\BR^{m}$ and a semi algebraic map
$f\colon X\to\BR^{n}$ there exists a
 stratification $\calS$ of $X$ such that each stratum  $S$ is a semi-algebraic
smooth submanifold,  the restriction $f|_S$ is a smooth map of constant rank,
and $\{f(S):S\in\calS\}$ is a stratification of $f(X)$.
\end{prop}
We will call $\calS$ of \refP{strat} a \emph{smooth semi-algebraic stratification of $X$ with respect to
$f$}.
Given a finite family $\calA$ of semi-algebraic subsets of a semi-algebraic set
$X$ (for example $\calA$ may itself be a stratification),
we say that the stratification $\calS$ is a refinement of $\calA$ if each
element of $\calA$ is a union of strata in $\calS$.
The above smooth stratification of $X$ relative to $f$ can be chosen to
be a refinement of any given finite family $\calA$ of semi-algebraic subsets of
$X$.

There is a larger class  of \emph{subanalytical sets} introduced by Hironaka
in \cite{Hir:tri} (see also \cite{Har:str}) which he proved are also
smoothly stratifiable. 

If $T\subset\BR^N$ is smoothly stratifiable, the \emph{local dimension of $T$
at $x\in T$}, denoted by $\dim_x(T)$,
is the highest dimension of a smooth submanifold $M\subset T$ such that
$x\in\overline M$. The \emph{dimension of
$T$} is $\dim(T)=\sup_{x\in T}\dim_{x}T$.  If $T\subset X$ is an inclusion
of smoothly stratifiable sets in $\BR^N$, we say that $T$ is of
\emph{codimension (at least)  $1$} in $X$ if, for each $x\in T$, $\dim_{x}(T)<\dim_{x}(X)$.  For a semi-algebraic set $X$ and a map $f\colon
X\to Y$,
 $\dim(\overline X\setminus X)<\dim X$ \cite[Theorem 2.8.12]{BCR:GAR} and
$\dim f(X)\leq \dim X$ \cite[\S 2]{BCR:GAR}.

For $k\ge 0$, any subset $Y$ of $\BR^{n}$ has a $k$-dimensional Hausdorff
(outer) measure, $\calH^k(Y)\in[0,\infty]$, defined for example in
\cite[2.10.2]{Fed:GMT}.  In particular, the  $k$-dimensional Hausdorff measure
of a $k$-dimensional  smooth submanifold of $\BR^n$ is its usual
$k$-dimensional volume with respect to the Riemannian metric induced from
$\BR^{n}$.
For a smoothly stratifiable set $X\subset \BR^N$,  $\dim(X)<k$ if and only if $\calH^k(X)=0$.
 For us, the key property will be that bounded semi-algebraic sets have finite volume.
\begin{thm}
\label{T:finiteHausdorff}
If $X$ is a bounded semi-algebraic set of dimension $\leq k$ then its
 Hausdorff $k$-measure $\calH^k(X)$ is finite.
\end{thm}
This fact was discovered in the complex case
by Lelong \cite{Lel:int} and is proved in the real case in \cite[3.4.8 (13)]{Fed:GMT}.
See also  \cite[Proposition 4.1, page 178]{vdD:lim} for a more modern reference.

Compact semi-algebraic sets admit semi-algebraic triangulations. More precisely,
if $X$ is a compact semi-algebraic set,
there exists a finite simplicial subcomplex $K$ of $\BR^{q}$ for some $q$ and a semi-algebraic homeomorphism
$\phi\colon K\iso X$. Moreover, if $\calS$ is a given stratification of $X$, we can choose $K$ and $\phi$ such that
 each stratum  of $\calS$ is the union of the images  of some open simplices of $K$ under $\phi$ \cite[Th\'eor\`eme 9.2.1]{BCR:GAR}.
The image under $\phi$ of the standard combinatorial stratification of $K$ into open simplices gives a new stratification
that refines $\calS$; this is called a \emph{triangulation of $X$ compatible with $\calS$}. Further, the polyhedron $K$
is unique up to PL homeomorphism (see \refT{Hauptvermutung}). Any $n$-simplex can be naturally decomposed as a union of
$n+1$ $n$-dimensional cubes, i.e.~subspaces homeomorphic to $[0,1]^n$, whose vertices are the barycenters of the various
faces of the simplex.  These cubes meet along codimension $1$ faces. Therefore any
compact  semi-algebraic set is semi-algebraically homeomorphic
 to a finite union of cubes; we call this a semi-algebraic \emph{cubification} of the semi-algebraic set.

There is also a weaker notion of triangulation in the noncompact case. Indeed, since $\BR^n$ is semi-algebraically homeomorphic
to $(0,1)^n$, any semi-algebraic set is homeomorphic to a bounded one, so there is no loss of generality in assuming
that a  semi-algebraic set  $X$ is  bounded.
Then its closure $\overline X$ is compact and  admits  a semi-algebraic
triangulation that refines the stratification $\{X,\overline X\setminus X\}$.
 This means that $X$ is semi-algebraically homeomorphic to
 a finite union of open simplices and the collection of these simplices is a stratification.
 The closure in $X$ of each stratum consists of that
 open simplex and of all of its faces that belong to $X$.
We call such an open simplex with some of its faces \emph{semi-open} and this stratification a \emph{semi-open triangulation} of $X$, or by abuse of terminology, a triangulation of $X$.  Any finite stratification of $X$ can be refined by such a semi-open triangulation.

By the main result of \cite{Har:loc} (see also \cite[Proposition 9.3.1.]{BCR:GAR}),
given a semi-algebraic map $f\colon X\to Y$ and a stratification of $X$
there exists a refinement of that stratification such that, for  each stratum $S$ of this refinement,
the restriction $f|_S\colon S\to f(S)$ is a trivial bundle.
In other words,
there exists a semi-algebraic homeomorphism  $S\cong f(S)\times F$ in which $f|_S$ corresponds to the
projection on the first factor. We say that the refined stratification \emph{trivializes} $f$.

A stratification  particularly useful to us is the following
\begin{lemma}
\label{L:smoothstrat}
Let $X\subset \BR^{N}$ be a semi-algebraic set of dimension $k$ and let $f\colon X\to\BR^{n}$ be a semi-algebraic map.
Any stratification of $X$ can be refined to a stratification $\calS$ with each of its strata $S$ a smooth submanifold
(without boundary) of $\BR^{N}$, and with 
the restriction $f|_S$ a smooth map that satisfies:
\begin{itemize}
\item $f|_S$ is of rank $k$,
 $f(S)$ is a submanifold of dimension $k$ in $\BR^{n}$, and $f|_S$ is a diffeomorphism
between $S$ and $f(S)$, or
\item  $f|_S$ is of rank $<k$ and $\dim(f(S))<k$.
\end{itemize}
This dichotomy is preserved by passing to a smooth  refinement of $\calS$.
Stratification $\calS$ can furthermore be refined so that for each pair of strata  $S$ and $S'$ on which the restriction of $f$
is of rank $k$, either $f(S)=f(S')$ or $\overline{f(S)}\cap f(S')=\emptyset$.
 \end{lemma}
\begin{proof}
Start with a stratification of $f(X)$ into smooth submanifolds
and refine its preimage under $f$ into a smooth semi-algebraic stratification  of  $X$ with respect to $f$ as in \refP{strat}.
Take a refinement that trivializes
$f$ and such that each stratum is connected. We still have that on each
stratum $S$ of this refinement, $f|_S$ is either of constant rank $k$ or everywhere of rank $<k$.
Let $S$ be a stratum such that $f|_S$ is of constant rank $k$.
Since $f|_S\colon S\to f(S)$ is a trivial bundle  and its domain and codomain are both of dimension
$k$, its fiber is discrete. Since $S$ is connected, we deduce that $f|_S$ is a bijection.
Moreover $f(S)$ is contained in a submanifold of dimension $k$, i.e~$f|_S$ is a smooth injective immersion
into a submanifold of the same dimension.  Thus $f(S)$ is a submanifold of dimension $k$ and
 $f|_S$ is a diffeomorphism onto its image. On the other hand, if $S$ is a stratum such
that $\rank(f|_S)<k$ then $\dim f(S)<k$.

It is clear that this dichotomy is preserved by passing to a refinement.

For the last part, let $S_{1},\cdots,S_{l}$ be the strata on which $f$ is of rank $k$.
Let $\calA$ be the smallest partition of $\cup_{i=1}^{l}S_{i}$ such that each set $S_{i}\cap\cap_{j\in J}f^{-1}(f(S_{j}))$ is a union
of elements of $\calA$
for $1\leq i\leq l$ and $J\subset\{1,\cdots,l\}$. Consider the family $\calA'=\{\inter_{X} (A):A\in\calA\}$, where
$\inter_{X} (A)$ is the interior of $A$ in $X$.
Then for $A_{1},A_{2}\in\calA'$ we have that $f(A_{1})=f(A_{2})$ or $\overline{f(A_{1})}\cap f(A_{2})=\emptyset$.
Moreover $\cup\calA'$ is dense in $\cup_{i=1}^{l}S_{i}$. Now let $X'=X\setminus\cup\calA'$ (this is also semi-algebraic) and take a smooth semi-algebraic
stratification $\calS'$ of $X'$ with respect to $f|_{X'}$ and that refines $\{\overline{A'}\cap X':A'\in\calA'\}$
It is straighforward to check that $\calA'\cup\calS'$ is a stratification of $X$ with
the desired properties.
\end{proof}

A \emph{maximal stratum} is a stratum that is not a subset of the closure of any other stratum. The closure of a maximal stratum is disjoint from
any other maximal stratum. The union of the maximal strata of a smooth stratification 
is a dense smooth submanifold (not necessarily of the same dimension on each connected component).

By partitioning each stratum into its connected components, we can refine any stratification into a stratification whose strata are connected.
 All the properties of stratifications considered above (smoothness, trivializations, properties from \refL{smoothstrat})
are preserved by taking such connected refinements.


\subsection{Review of PL topology}
\label{S:PLtop}


We will at several points throughout the paper  need various results from PL topology. In particular we will need the \emph{Hauptvermutung} (uniqueness of triangulations)
for semi-algebraic sets and the notion of \emph{collapsing polyhedra} (strong form of contractibility).
The basic reference from which we extract the following review of these concepts is \cite[Chapters I-II]{Hud:PLT}  (most of the results are also in \cite{RoSa:PLT}).

A \emph{closed  simplex}
 in $\BR^n$ is the convex hull of at most $n+1$ points in general position. An \emph{open 
simplex} is the interior of a closed  simplex in the affine subspace of the same dimension and containing the closed simplex
(a point
is both a closed and an open $0$-simplex). A closed  simplex can be partitioned into a finite family of open simplices called its \emph{faces}.
In this paper we define a \emph{polyhedron} as a subspace of $\BR^n$ that is the union of \emph{finitely} many closed
 simplices (therefore in this paper a polyhedron is always compact).
A \emph{PL map} between two polyhedra is a continuous map whose graph is a polyhedron
and it is  a  \emph{PL homeomorphism} if it is bijective.

A \emph{simplicial complex} is a finite partition of a polyhedron into open simplices, also called a \emph{triangulation} of the
polyhedron.
A \emph{subdivided triangulation} of a given triangulation is another triangulation which is a refinement (as partition) of the former.
Any two triangulations of the same polyhedron admit a common subdivided triangulation. We will often abuse notation by identifying a simplicial complex (which is a partition) with its geometric realization  
(which is a polyhedron obtained by taking the union of the elements of the partition.)

A \emph{PL ball} is a polyhedron PL homeomorphic to a closed  simplex.
A subspace $F\subset B$ of a PL ball $B$ is called a \emph{face} if there is
a PL homeomorphism from $B$ to some simplex that sends $F$ onto a face of that simplex.
If $P_0\subset P_1$ is an inclusion of polyhedra one says that $P_1$ \emph{elementary collapses} to $P_0$ if
$\overline{P_1\setminus P_0}$ is a ball and if $\overline{P_1\setminus P_0}\cap P_0$ is a face of that ball. More generally, a polyhedron $P$ \emph{collapses} to a subpolyhedron $P_0$ if there is a decreasing sequence of polyhedra,
$P=P_r\superset P_{r-1}\superset\dots\superset P_0$, such that $P_{i}$ elementary collapses to $P_{i-1}$, $1\leq i\leq r$.
In this situation we write $P\searrow P_0$. One says that a polyhedron $P$ is \emph{collapsible} if $P\searrow *$, where $*$ is a one-point space.
A collapsible polyhedron is clearly contractible but the converse is not true (a counterexample is given by the ``dunce hat'' of Zeeman \cite{Zeeman}).

We say that a simplicial complex $K$
\emph{simplicially collapses} to a subcomplex $K_0$, and we write $K\searrow^sK_0$,
if there exists a sequence of
simplicial subcomplexes $K=K_r\superset K_{r-1}\superset\cdots\superset K_0$ such that
 $K_i=K_{i-1}\cup\tau_i$ where $\tau_i$ is a simplex
with an open codimension $1$ face which is disjoint form $K_{i-1}$
\cite[Definition in II.1 and remark following it, page 43]{Hud:PLT}.

It is clear that if a simplicial complex simplicially collapses to a subcomplex, then
there is a collapse of the underlying polyhedra. Conversely, suppose that a polyhedron $P$ collapses to a subpolyhedron $P_0$ and that $K$
is a triangulation of $P$ such that some subcomplex $K_0$ of $K$ is a triangulation of $P_0$.  Then
 there exists a subdivision $K'$ of $K$ that simplicially collapses to $K_0'$ \cite[Theorem 2.4]{Hud:PLT}.
In particular, any triangulation of a PL ball admits a subdivided triangulation that simplicially collapses to a point.

We have already reviewed  the fact that a compact semi-algebraic set admits a semi-algebraic triangulation, i.e.~is semi-algebraically
homeomorphic to a compact polyhedron. A very useful fact is the following result about the uniqueness of this triangulation, proved by  Shiota and Yokoi \cite[Corollary 4.3]{ShYo:tri}.
\begin{thm}[Hauptvermutung for compact semi-algebraic sets]
\label{T:Hauptvermutung}
If a compact semi-algebraic set is semi-algebraically homeomorphic to two polyhedra,
then those polyhedra are PL-homeomorphic.
\end{thm}

 The following is an immediate consequence.
\begin{cor}\label{C:collapsibility}
If a compact semi-algebraic set is semi-algebraically
homeomorphic to a collapsible polyhedron, then any semi-algebraic triangulation of it admits a subdivided
triangulation that simplicialy collapses to a vertex.
\end{cor}


\subsection{Review of currents}
\label{S:currents}


To define semi-algebraic chains, we will find the language of currents very convenient, so we review this notion here. A basic reference is \cite{Fed:GMT}.

Denote by
$\calD^k(\BR^n)$  the vector space of smooth differential $k$-forms on
$\BR^n$ with compact support.  This space can be endowed with the
with the usual Frechet topology (see \cite[4.1.1]{Fed:GMT})
but we will not use it\zz{changed following Bob'suggestion}.
The exterior differential $d\colon \calD^k(\BR^n)\to\calD^{k+1}(\BR^n)$ is linear and continuous.

A \emph{$k$-current} on $\BR^n$ is a linear continuous form $T\colon \calD^k(\BR^n)\to\BR$. We denote
by $\calD_k(\BR^n)$ the vector space of $k$-currents. In other words, $\calD_k(\BR^n)$ is the topological
dual of $\calD^k(\BR^n)$.
The adjoint of the exterior differential
$d\colon\calD^k(\BR^n)\to\calD^{k+1}(\BR^n)$ defines a boundary operator
$\partial\colon \calD_{k+1}(\BR^n)\to\calD_{k}(\BR^n)$ by $\partial T(\omega)=T(d\omega)$.

For $U$ an open subset of $\BR^{n}$, we define $\calD^{k}(U)$ as the subspace of smooth forms in $\BR^n$ whose support lies in
$U$.
The \emph{support of a current $T$} in $\BR^{n}$, $\spt(T)$, is the smallest closed subset of $\BR^{n}$
such that
$\omega\in\calD^{k}(\BR^{n}\setminus \spt(T))\Longrightarrow T(\omega)=0$.
Clearly $\spt(\partial T)\subset  \spt(T)$.
For $X\subset\BR^n$ we set $$\calD_k(X)=\{T\in\calD_k(\BR^n):\spt(T)\subset X\}$$ and thus  get a
chain complex $$\calD_*(X):=\left(\bigoplus_{k\geq 0} \calD_k(X),\partial\right).$$

The most important example of a current for us is the following
\begin{example}
\label{X:bbr} Denote by $\calH^{k}$ the $k$-dimensional Hausdorff measure on $\BR^n$.
Let $V$ be a smooth $k$-dimensional oriented submanifold of $\BR^n$ such that $\calH^{k}(V)<\infty$. One
 defines a $k$-current $\bbr{V}\in\calD_k(\BR^n)$ by $$\bbr{V}(\omega)=\int_V\omega, \ \ \ \ \omega\in\calD^k(\BR^n).$$
\end{example}

Currents can be equipped with various (semi-)norms and we review some of them from
 \cite[\S 1.8.1 page 38, \S  4.1.7 page 358, and \S 4.1.12  pages 367-368]{Fed:GMT}
\begin{definition}
\label{D:flatchain}
\ \ 
\begin{itemize}
\item
The \emph{comass}
of a differential form $\omega\in\calD^k(\BR^n)$ is
\[\mass(\omega)=\sup\{\omega(x)(v_1\wedge\cdots\wedge v_k):x\in\BR^{n},v_i\in\BR^n,\|v_i\|\leq1\}\in\BR_+\]
\item
The \emph{mass} of a $k$-current $T\in\calD_k(\BR^n)$  is
$$\mass(T)=\sup\{T(\omega):\omega\in\calD^k(\BR^n),\mass(\omega)\leq1\}\in\BR_+\cup\{\infty\}.$$
\item
A current $T\in\calD_{k}(\BR^{n})$ is called \emph{normal} if $\spt(T)$ is compact and
$\mass(T)$ and $\mass(\partial T)$ are finite.
\item
The \emph{flat seminorm of $T\in\calD_{k}(\BR^{n})$ relative to a compact set $K\subset\BR^{n}$} is defined by
$$\flatnorm_{K}(T)=\inf\{\mass(T-\partial S)+\mass(S): S\in\calD_{k+1}(\BR^{n}),\spt(S)\subset K\}.$$

\item
A current is called a \emph{flat chain} if it is the limit of a sequence of normal currents  for the flat seminorm relative to some compact set.
\end{itemize}
\end{definition}
We will write $\mass_k(T)$ for the mass when we want to emphasize the dimension of the current.  Also, when there is no ambiguity about the compact set $K$ in the definition of the flat seminorm, we will drop it from the notation.

A normal current is clearly a flat chain. Also, the boundary of a flat chain is a flat chain \cite[p. 368]{Fed:GMT}.
There are two other properties of flat chains we will use:
\begin{enumerate}
\item  \emph{Absolute continuity}, which states that if $T$
is a $k$-dimensional flat chain and if $\mathcal H^k (\spt(T )) = 0$, then
$T = 0$ \cite[Theorem 4.1.20.]{Fed:GMT};  and
\item \emph{Constancy theorem}, 
 which states that if $U$ is an open
set in $\BR^n$,  $T$ is a flat $k$-chain in $\BR^n$ with $U\cap\spt(T)$
a connected $k$-dimensional oriented smooth submanifold $V$ and $U\cap \spt( \partial T)=\emptyset$,
 then, in $U$, $T$ is just a constant multiple
of the $k$-current $\bbr{V}$ given by integration along  $V$  \cite[\S 4.1.31 (2)]{Fed:GMT}. In case $T$ is an
integral flat chain (as will be the case for the semi-algebraic {\current}s
 considered later), this multiple is an integer.
\end{enumerate}

The norms of \refD{flatchain} can be used to study convergence of sequences of currents. It is clear that convergence in mass implies
convergence in the flat seminorm.
We can also equip  $\calD_k(\BR^n)$
with its weak topology characterized by the fact that a sequence of currents
$(T_{n})_{n\geq1}$ converges weakly to $T$ in $\calD_k(\BR^n)$, which we denote by $T_{n}\rightharpoonup T$,
 if, for each $\omega\in \calD^k(\BR^n)$, $(T_{n}(\omega))_{n\geq1}$ converges to $T(\omega)$ in $\BR$.
An elementary computation shows that convergence in flat norm implies weak convergence.

If $V$ is a smooth oriented submanifold of dimension $k$ in $\BR^{n}$ with finite Hausdorff $k$-measure,
then the current $\bbr{V}\colon\omega\mapsto\int_{V}\omega$ is of mass
 $\mass(\bbr{V})=\calH^{k}(V)$.


\section{Semi-algebraic {\current}s}
\label{S:SAcur}


The goal of this section is to functorially associate to each semi-algebraic set $X$ a chain subcomplex
$(\Ch_{*}(X),\partial)\subset (\calD_{*}(X),\partial)$ consisting
\zz{
     Following the suggestion of Bob, I changed semi-algebraic
     currents by semi-algebraic chains (or, shortly, chains).
     To do so I introduced a command \texttt{backslash current}
     and used that command at each place where I thought that
     ``chain'' was better than ``current'', essentially each time we
     talk about semi-algebraic currents.
     }
 of ``semi-algebraic \current s''.

Let $M$ be an oriented compact semi-algebraic manifold of dimension $k$
 and let $f\colon M\to\BR^{n}$
be a semi-algebraic map. Consider a stratification of $M$ with connected strata, having all the properties of \refL{smoothstrat}.
Let $S_{1},\cdots,S_{l}$ be the strata such that $f|_{S_{i}}$ is of rank $k$. Each of the submanifolds $S_{i}$ is
oriented by the orientation of $M$.
Consider the family $\calN=\{f(S_{i}):1\leq i\leq l\}$ of image submanifolds,
which could be of cardinality $<l$ because $f(S_{i})=f(S_{j})$ for $i\not=j$ is a possibility. Fix an arbitrary orientation on each $N\in\calN$ and
set $n_{N}=\sum_{i,f(S_i)=N}\epsilon_{i}$ where $\epsilon_{i}=\pm 1$, the sign depending on whether the diffeomorphism
$f|_{S_{i}}\colon S_{i}\iso N$ preserves or reverses the orientation. Since $M$ is compact, the semi-algebraic set $N\subset f(M)$
is bounded and by \refT{finiteHausdorff} it is  of finite $k$-volume. Therefore we can define a current
 $f_{*}\bbr{M}\in\calD_{k}(\BR^{n})$
by the formula
\begin{equation}\label{E:define-fM}
f_{*}\bbr{M}(\omega)=\sum_{N\in\calN}n_{N}\int_{N}\omega
\end{equation}
for $\omega\in\calD^{k}(\BR^{n})$.  This current is of finite mass, i.e.~$\mass(f_{*}\bbr{M})=\sum_{N\in\calN}|n_{N}|\cdot\calH^{k}(N)<\infty$.
Notice also that
$$f_{*}\bbr{M}(\omega)=\sum_{i=1}^{l}\int_{{S_{i}}}f^{*}\omega.$$
It is easy to check that the right side of the last equation does not depend on the choice of the stratification, because
it is clearly unchanged under refinement, and any two stratifications
satisfying the properties of \refL{smoothstrat} have a common refinement with the same properties.
It is also clear that if $h\colon M'\iso M$ is a semi-algebraic homeomorphism preserving orientation, then currents
$f_*(\bbr{M})$ and $(f\circ h)_*(\bbr{M'})$ are equal.
Note that if $f$ is also smooth or even Lipschitz, then $f_*(\bbr{M})$ coincides with the usual
pushforward operation for currents \cite[\S 4]{Fed:GMT}. However, a semi-algebraic map in general is not Lipschitz
even though it is piecewise smooth.

\begin{definition}
\label{D:C}
A \emph{semi-algebraic $k$-{\current} in $\BR^{n}$} is a current of the form $f_{*}\bbr{M}\in\calD_{k}(\BR^{n})$
 as constructed above
for some oriented compact semi-algebraic $k$-dimensional manifold $M$ and some semi-algebraic map $f\colon M\to\BR^{n}$.
The set of semi-algebraic $k$-{\current}s in $\BR^n$ is denoted by $\Ch_{k}(\BR^{n})$.
If $X$ is a semi-algebraic set in $\BR^{n}$, we set
$$\Ch_{k}(X)=\{\gamma\in \Ch_{k}(\BR^{n}):\spt(\gamma)\subset X\}.$$
The elements of  $\Ch_{k}(X)$ are called\zz{see the previous footnote
  on semi algebraic currents vs. chain; changed again vs. version of
  18 january} the \emph{semi-algebraic
  $k$-\current s in $X$}, or simply \emph{$k$-chains}.
\end{definition}

\begin{lemma} $\Ch_k(X)$ is a subgroup of $\calD_k(X)$.
\end{lemma}
\begin{proof} The chain $f_{1*}\bbr{M_{1}}+f_{2*}\bbr{M_{2}}$ can be represented by
$(f_1\amalg f_2)_{*}\bbr{M_1\amalg M_2}$ where, without loss of generality, we have supposed that the manifolds $M_{1}$ and $M_{2}$ are disjoint subsets of the same $\BR^{m}$. Also 
$-f_{*}\bbr{M}=f_{*}\bbr{-M}$ where $-M$ is $M$ with the opposite orientation.
\end{proof}

The proof of the following is straightforward.
\begin{prop}
\label{P:dimC}
 If $i>\dim(X)$, then $\Ch_i(X)=0$.
\end{prop}

It is clear that any semi-algebraic \current{} in $\Ch_k(X)$ can be represented by a semi-algebraic map $g\colon M\to X$
where $M$ is an oriented compact semi-algebraic manifold of dimension $k$. However, here is  another useful representation of
semi-algebraic \current s.

\begin{prop}[{\bf Alternative description of semi-algebraic {\current}s}]
\label{P:altSAcur}
Let $V_1,\cdots,V_r$ be disjoint smooth semi-algebraic oriented $k$-dimensional submanifolds of $\BR^N$
 such that each $\overline{V_i}$ is a compact subset
of the semi-algebraic set $X\subset \BR^N$. Let $n_1,\cdots,n_r$ be integers. Then
$\sum_{i=1}^r n_i\bbr{V_i}$ represents an element of $\Ch_k(X)$.
Conversely, any element of  $\Ch_k(X)$ admits such a representation.
\end{prop}

\begin{proof}
Let $X$ be a semi-algebraic set and let  $V\subset X$ be
an oriented $k$-dimensional semi-algebraic manifold, not necessarily compact but
with a compact closure $\overline V$ in  $X$. We do not suppose that  $\overline V$ is a manifold.
We can stratify $V$ into a finite number
of disjoint smooth semi-algebraic submanifolds $V_{i}$ of dimension $k$ and a subset $Z$ of dimension $<k$.
By \refT{finiteHausdorff}, the $V_{i}$ have finite $k$-volume  and
we can define the \emph{pushforward current} $\bbr{V}\in \calD_{k}(X)$ by
$$\bbr{V}(\omega)=\sum_{i}\int_{V_{i}}\omega.$$
This formula is independent of the choice of stratification.
We want to show that $\bbr{V}\in \Ch_{k}(X)$.
Notice first that by \cite[Proposition 2.8.12]{BCR:GAR}, $\dim(\overline V\setminus V)<k$
and $\dim(\overline V_{i}\setminus V_{i})<k$.
The compact space $\overline V$ admits a semi-algebraic triangulation $\phi\colon K\iso \overline V$
where $K$ is a finite simplicial complex and we can assume that the triangulation  refines the family
 $\{V_{i}\}\cup\{Z\}$. Each closed $k$-dimensional simplex $\sigma$ of $K$
is mapped into the closure of some $V_{i}$ and the orientation of $V_{i}$ induces an
 orientation of $\sigma$. We thus get a semi-algebraic $k$-{\current} $\phi_{*}\bbr{\sigma}$ and it
is clear that $\bbr{V}=\sum_{\sigma}\phi_{*}\bbr{\sigma}\in \Ch_{k}(X)$, where the sum
is taken over all $k$-simplices of $K$.

The converse is an immediate consequence of formula \refN{E:define-fM}.
\end{proof}

We now turn  to functorial properties of semi-algebraic {\current}s. Let $g\colon X\to Y$ be a semi-algebraic map.
If $f_{*}\bbr{M}$ is a semi-algebraic {\current} in $X$ then $(g\circ f)_{*}\bbr{M}=g_{*}(f_{*}\bbr{M})$ is a semi-algebraic {\current} in $Y$. It is easy to check that this gives a well-defined homomorphism
$$g_{*}\colon \Ch_{k}(X)\longrightarrow \Ch_{k}(Y)$$
and that this association is functorial.  In other words, $(h\circ g)_{*}=h_{*}\circ g_{*}$ and $\id_{*}=\id$.

In order to prove that $g_*$ defines a chain map in \refC{chaincomplex}, we need the
following.
\begin{thm}
\label{T:Stokes}
Let  $M$ be  a compact semi-algebraic
oriented manifold and let $f\colon M\to \BR^n$ be  a semi-algebraic map.
Then
$\partial(f_{*}(\bbr{M}))=f_{*}(\bbr{\partial M})$.
\end{thm}

\begin{rem}
When $f$ is the identity map, this theorem can be thought of as a generalization of Stokes' Theorem which states that
 $\partial \bbr{M}=\bbr{\partial M}$, and in the general case it is a combination of  Stokes' Theorem and naturality
of the boundary operator with respect to $f_*$.
\end{rem}

\begin{proof}
Suppose that $M\subset\BR^{m}$ and set $k=\dim M$.
Take a triangulation of $M$
such that each stratum $S$ is a smooth submanifold of $\BR^{m}$ and $f|_S$ is smooth.
Since both sides of the equation we want to establish are clearly additive,  it is enough to prove that the equation holds on each closed
simplex of maximal dimension in that triangulation. Thus we can assume that  $M$ is stratified as the standard
combinatorial stratification
of the $k$-simplex $\Delta^{k}$. In particular, $M$ has the unique
maximal stratum $V=\operatorname{int(M)}=M\setminus\partial M$ of dimension $k$, and $k+1$
codimension $1$ connected strata $S_{0},\cdots,S_{k}$ whose union is dense in $\partial M$.
Since $f$ is smooth on the interior of the simplex~\cite[Remark 9.2.3]{BCR:GAR},
 by replacing the original
triangulation by its barycentric subdivision we can furthermore assume that
$f|_{M\setminus \overline{S_{0}}}$ is smooth.

The proof now proceeds by induction on $k=\dim(M)$.
If $k=0$, there is nothing to show. If $k=1$, there exists a semi-algebraic homeomorphism $\phi\colon[0,1]\iso M$ that is smooth on
$(0,1]$. For $0<\epsilon<1$, set $M_{\epsilon}=\phi([\epsilon,1])$. Since $\mass(f_{*}\bbr{M})<\infty$, Lebesgue's Bounded
Convergence Theorem and smooth Stokes' Theorem imply that for each smooth $0$-form
$\omega\in\calD^{0}(\BR^{n})$ we have
\begin{eqnarray*}
(\partial f_{*}\bbr{M})(\omega)&=&f_{*}\bbr{M}(d\omega)=\int_{\operatorname{int}(M)}f^{*}(d\omega)=
\lim_{\epsilon\to0}\int_{M_{\epsilon}}f^{*}(d\omega)
=\lim_{\epsilon\to0}\int_{\partial M_{\epsilon}}f^{*}(\omega)=\\
&=&\lim_{\epsilon\to0}\omega(f(\phi(1)))-\omega(f(\phi(\epsilon)))
=\omega(f(\phi(1)))-\omega(f(\phi(0)))=(f_{*}[\partial M])(\omega).
\end{eqnarray*}
This proves the result for $k=1$.

Suppose that the theorem has been proved in dimension $<k$ for some $k\geq2$.
We first prove it in dimension $k$ with the assumption that
$f$ is injective on $M$ and of maximal rank on each stratum.

To do this, we first observe that $f_*\bbr{S_i}$  and $\partial f_*\bbr{M}$  are flat chains  as 
 in \refD{flatchain}.
 Notice that $f_*\bbr{S_i}=f_*\bbr{\overline{S_i}}$ and, by induction hypothesis,
 $\partial f_*\bbr{\overline{S_i}}= f_*\bbr{\partial\overline{S_i}}$. Since these semi-algebraic
 {\current}s
 are of finite mass we deduce that $ f_*\bbr{{S_i}}$
 is a normal current, hence a flat chain. For $\bbr{M}$,
consider an increasing
sequence $M_{1}\subset M_{2}\subset \cdots$ of $k$-dimensional compact  smooth manifolds
whose union is $V=\operatorname{int}(M)$. By the smooth Stokes' Theorem,
 $\mass(\partial f_{*}[M_{n}])=\mass(f_{*}[\partial  M_{n}])<\infty$,
and $\mass(f_{*}[M_{n}])\leq \mass(f_{*}\bbr{M})<\infty$.
 Thus each current $f_{*}[M_{n}]$ is normal and, by Lebesgue's convergence theorem,
their sequence converges in mass, and hence in flat semi-norm, to $f_{*}\bbr{M}$. This implies that $f_{*}\bbr{M}$ is a flat chain, so the same is true
for $\partial f_{*}\bbr{M}$ by \cite[p. 368]{Fed:GMT}.

Using the fact that $f$ is  smooth on $M\setminus\partial M$, and injective and continuous on $M$,  it is easy to check
with the smooth Stokes' Theorem that
$\spt(\partial f_{*}(\bbr{M}))\subset f(\partial M)$.

To continue, we will use the constancy theorem from \refS{currents}.
 Since $f$ is a homeomorphism on its image and is smooth on each stratum,
it is clear that $\cup_{i=0}^ k f(S_i)$ is a smooth submanifold and is open in the closed subset $f(\partial M)$.
Set $Z=f(\partial M)\setminus\cup_{i=0}^ k f(S_i)$.  This is a closed subset of dimension $<k-1$.
Since each $f(S_i)$ is a  connected component of $\cup_{i=0}^ k f(S_i)$, by \cite[\S 4.1.31 (2)]{Fed:GMT} with $C=f(S_i)$ and $r=c_i$, there exist $c_i\in\BR$ such that
\[\spt(\partial f_{*}(\bbr{M})-c_if_*\bbr{S_i})\cap f(S_i)=\emptyset.\]
Now set $T=\partial f_{*}(\bbr{M})-\sum_{i=0}^{k}c_{i}f_{*}[S_{i}]$. We deduce that
$\spt(T)\subset Z$. Since $T$ is a flat $(k-1)$-chain and $\dim(Z)<k-1$,
 Theorem 4.1.20 of
\cite{Fed:GMT} implies that $T=0$.  In other words,
\begin{equation}\label{E:sumci}
\partial f_{*}(\bbr{M})=\sum_{i=0}^{k}c_{i}f_{*}\bbr{S_{i}}.
\end{equation}

Clearly $f_{*}[\partial M]=\sum_{i=0}^{k}f_{*}\bbr{S_{i}}$, so it remains only to prove
that all the $c_{i}$'s are $1$. For $1\leq i\leq k$, this is an immediate consequence of the classical
Stokes' Theorem since $f$ is smooth on $M\setminus \overline{ S_{0}}$.
Applying the boundary operator to both sides of \refE{sumci} and using the induction hypothesis we get that
$$c_{0}f_{*}[\partial   \overline{S_{0}}]+
\sum_{i=1}^{k}f_{*}[\partial \overline{ S_{i}}]=0.$$
For $1\leq i\leq k$, let $B_{i}$ be the $(k-2)$-dimensional stratum of $M$ whose closure is
 $\partial  \overline {S_{0}}\cap\partial  \overline {S_{i}}$. After canceling terms, the last equation becomes
$\sum_{i=1}^{k}(c_{0}-1)f_{*}[B_{i}]=0$.
Clearly  $f(B_{1})\subset\spt(f_{*}[B_{1}])$ and, for $2\leq i\leq k$,
 $f(B_{1})\cap\spt(f_{*}[B_{i}])=\emptyset$. Since $k\geq2$,  $f(B_{1})\not=\emptyset$ and we deduce that $c_{0}=1$.
 This proves the theorem in dimension $k$ under the extra assumption that $f$ is injective and of maximal
 rank on each stratum.


For the general case, consider the semi-algebraic inclusion
$i\colon M\hookrightarrow \BR^m$
and, for $\epsilon\in[0,1]$, define the map
\begin{align*}
f_{\epsilon}\colon M & \longrightarrow\BR^{n}\times\BR^{m} \\
                   x & \longmapsto(f(x), \epsilon\cdot i(x)).
\end{align*}
When $\epsilon>0$, $f_{\epsilon}$  is injective and smooth of maximal
 rank on each stratum, and therefore by the beginning of the proof we have that
 $\partial f_{\epsilon*}(\bbr{M})=f_{\epsilon*}(\bbr{\partial M})$.
Since the boundary operator $\partial$ is continuous with respect to weak convergence,
in order to finish the proof it is enough to show that  $f_{\epsilon*}(\bbr{M})$ and
 $f_{\epsilon*}(\bbr{\partial M})$ converge weakly to $f_{0*}(\bbr{M})$ and
$f_{0*}(\bbr{\partial M})$ as $\epsilon\to 0$, respectively.

We prove
 that  $f_{\epsilon*}(\bbr{M})\rightharpoonup f_{0*}(\bbr{M})$, the case
of $\partial M$ being completely analogous.
Let $\delta>0$. Consider a 
differential form $\alpha\in\calD^k(\BR^n\times\BR^m)$.
 %
Since $f$ and $i$ are smooth on the complement of a codimension $1$ semi-algebraic subset of $M$,
there exists a  compact semi-algebraic  codimension $0$ smooth submanifold with boundary $M'\subset M$ on which $f$ and $i$ are smooth and such that
$\mass(f_{1*}(\bbr{M\setminus M'}))\leq{\delta}/({4\cdot\mass(\alpha)})$.
Moreover, notice that $f_{\epsilon}=q_{\epsilon}\circ f_1$, where $q_{\epsilon}$
is the map
\begin{align*}
q_{\epsilon}\colon \BR^n\times\BR^m & \longrightarrow \BR^n\times\BR^m \\
(u,v) & \longmapsto(u,\epsilon\cdot v)
\end{align*}
and which admits Lipschitz constant $1$. Therefore for all $\epsilon\in[0,1]$ we also have
\[\mass(f_{\epsilon*}(\bbr{M\setminus M'}))\leq\frac{\delta}{4\cdot\mass(\alpha)}.\]
Hence
\begin{align}
&\notag\left|
\langle f_{\epsilon*}(\bbr{M})- f_{0*}(\bbr{M})\ ,\ \alpha\rangle
\right|\leq\\\quad\leq
&\notag\left|\langle f_{\epsilon*}(\bbr{M\setminus M'})- f_{0*}(\bbr{M\setminus M'})\ ,\ \alpha\rangle
\right|\ +\ %
\left|\langle f_{\epsilon*}(\bbr{ M'})- f_{0*}(\bbr{ M'})\ ,\ \alpha\rangle
\right|\leq\\
\label{E:fe0-1}&
\leq\delta/2\ +\left|\langle \bbr{ M'}\ ,\ \left((f_{\epsilon}|_{M'})^* -\ (f_{0}|_{M'})^*\right)(\alpha)\rangle
\right|\ .
\end{align}
It is clear that $f_\epsilon|_{M'}$ converges in the $\calC^1$-norm to $f_0|_{M'}$,
and therefore for $\epsilon$ small enough the second summand in  \refN{E:fe0-1} is also less than $\delta/2$.
Therefore 
$f_{\epsilon*}(\bbr{M})$ converges weakly to $f_{0*}(\bbr{M})$.

The proof that $f_{\epsilon*}(\bbr{\partial M})\rightharpoonup f_{0*}(\bbr{\partial M})$ is exactly the same and we will omit it.
\end{proof}


The following is now immediate.

\begin{cor}
\label{C:chaincomplex}
\ \ 
\begin{itemize}
\item If  $X$ is a semi-algebraic set, then
 $\partial(\Ch_k(X))\subset \Ch_{k-1}(X)$ and $\Ch_{*}(X):=(\oplus_{{k\geq0}}\Ch_{k}(X),\partial)$ is a chain complex. 
\item If $g\colon X\to Y$ is a semi-algebraic map
then $g_{*}\colon \Ch_{*}(X)\to \Ch_{*}(Y)$ is a chain map, that is $g_*\partial=\partial g_*$.
\end{itemize}
\end{cor}
We end this section by introducing a cross product on chains by
the following easy result.
\begin{prop}
\label{P:chaincross}
Let $X_1$ and $X_2$ be  semi-algebraic sets.
There exists a degree-preserving linear map
\[\times\colon\Ch_*(X_1)\otimes \Ch_*(X_2)\longrightarrow\Ch_*(X_1\times X_2)\]
characterized by the formula
\[f_{1*}(\bbr{M_1})\times f_{2*}(\bbr{M_2})=(f_1\times f_2)_*(\bbr{M_1\times M_2})\]
for compact oriented semi-algebraic manifolds $M_i$ and semi-algebraic maps $f_i\colon M_i\to X_i$, $i=1,2$.\\
This product satisfies the Leibniz formula
\[\partial(\gamma_1\times\gamma_2)=\partial(\gamma_1)\times\gamma_2+
(-1)^{\deg(\gamma_1)}\gamma_1\times \partial(\gamma_2).\]
Let $T\colon X_1\times X_2\to X_2\times X_1$ given by $(x_1,x_2)\mapsto (x_2,x_1)$ be the twisting map. Then
\[T_*(\gamma_1\times\gamma_2)=(-1)^{\deg(\gamma_1)\deg(\gamma_2)}\gamma_2\times\gamma_1.\]
\end{prop}


\section{Convergence in $\Ch_{k}(X)$}
\label{S:conv}


The goal of this section is to introduce a suitable notion of convergence of sequences of semi-algebraic chains.
There are three classical notions of convergence for currents, which are, from  the strongest to the weakest (see \refD{flatchain}):
\begin{itemize}
\item Convergence in mass (i.e.~for the mass norm $\mass$);
\item Convergence in flat norm (i.e.~for the norm $\flatnorm$ relative to some compact subspace); and
\item Weak convergence, i.e.~$T_n\rightharpoonup T$ in $\calD_k(\BR^N)$ if, for each $\omega\in \calD^k(\BR^N)$, $T_n(\omega)\to T(\omega)$ in $\BR$.
\end{itemize}

However, these definitions of convergence cannot be adapted well to semi-algebraic chains because they are not
preserved by  semi-algebraic maps (because those are not locally Lipschitzian).
As a simple example,
consider the semi-algebraic map $f\colon \BR\to\BR,x\mapsto \sqrt{|x|}$. For $n\geq 1 $
set $\gamma_{n}=n{\cdot}\bbr{[0,1/n^{2}]}$, which is a sequence in $\Ch_{1}(\BR)$. Then $\mass(\gamma_{n})=1/n$ which converges to zero but
$f_{*}(\gamma_{n})=n{\cdot}\bbr{[0,1/n]}$ does not converge  to $0$ even weakly.

We thus introduce the following more suitable notion of convergence in $\Ch_k(X)$.
\begin{definition}
\label{D:SAconv}
A  sequence $\left(\gamma_{n}\right)_{n\geq1}$ in $\Ch_{k}(X)$
\emph{ converges semi-algebraically to  $\gamma\in \Ch_{k}(X)$},
denoted by $\gamma_{n}\toSA\gamma$,
if there exists a semi-algebraic map $h\colon M\times [0,1]\to X$, where $M$ is a compact semi-algebraic
 oriented manifold, and a sequence $(\epsilon_{n})_{n\geq 1}$ in $[0,1]$ converging to zero, such that
 $\gamma_{n}=h_{*}(\bbr{M\times \{\epsilon_{n}\}})$ and
  $\gamma=h_{*}(\bbr{M\times \{0\}})$.
\end{definition}
The following is immediate from the definition.
\begin{prop}
\label{P:preserveconvSA}
Let $f\colon X\to Y$ be a semi-algebraic map. If   $\gamma_{n}\toSA\gamma$ in $\Ch_k(X)$
then $f_*(\gamma_{n})\toSA f_*(\gamma)$ in $\Ch_k(Y)$.
\end{prop}
The following says that if a semi-algebraic subset $X_0$ is dense in $X$, then $\Ch_k(X_0)$ is ``dense'' in $\Ch_k(X)$
 in terms of  SA convergence.
\begin{prop}
\label{P:densitySA}
Let $X$ be a semi-algebraic set and let $X_0\subset X$ be a dense semi-algebraic subset.
Then for each $\gamma\in \Ch_k(X)$, there exists a sequence $(\gamma_n)_{n\geq1}$ in $\Ch_k(X_0)$
such that $\gamma_{n}\toSA\gamma$.
\end{prop}
\begin{proof}
Choose a triangulation of $X$ such that $X_0$ and $X\setminus X_0$ are unions of simplices and
such that $\gamma=\sum_\sigma n_\sigma\cdot\bbr{\overline\sigma}$ where the sum runs over
some $k$-simplices of the triangulation and the orientations of those simplices are chosen so that the
integers  $n_\sigma$ are nonnegative. For each such simplex, since $\sigma\subset\overline{X_0}$ and since
$X_0$ is a union of simplices, there exists an open simplex $\tau$ (which could be $\sigma$) such that $\tau\subset X_0$
and $\sigma$ is a face of $\tau$. It is easy to build a semi-algebraic map
\[h_\sigma\colon\Delta^k\times[0,1]\longrightarrow\overline\tau\]
such that  $h_\sigma(\Delta^k\times(0,1])\subset\tau$
and $h_\sigma(-,0)\colon\Delta^k\iso\overline\sigma$ is a homeomorphism preserving the orientation.
Set $M=\amalg_\sigma\amalg_{i=1}^{n_\sigma}\Delta^k$ and consider the map
\[h=\amalg_\sigma\amalg_{i=1}^{n_\sigma} h_\sigma
\,\colon\,M\times[0,1]\longrightarrow X.\]
Then $\gamma_n=h_*(\bbr{M\times\{1/n\}})$ is the desired sequence.
\end{proof}

Our last important result is that SA (semi-algebraic) convergence implies weak convergence as currents.
\begin{prop}
\label{P:SAcont}
 If $\gamma_{n}\toSA\gamma$ in $\Ch_k(X)$ then $\gamma_{n}\rightharpoonup \gamma$.
  \end{prop}
\begin{proof}
Take $h$, $M$, and $\epsilon_n$ as in \refD{SAconv}.
Then\zz{
               referee's typo (11). 
               Referee suggest to put $\pm$ everywhere but it  is quite useless. The
               point of $\pm$ in the previous version was that I did not want to be bothered with the
               Koszul sign $(-1)^{\dim M}$ in the Leibniz formula. So
               the previous version was indeed correct
               To please the referee, the following then is maybe more
               meaningful.
               }
\begin{eqnarray*}
\gamma_n-\gamma&=&
 h_*(\bbr{M\times\partial[0,\epsilon_n]})\\
&=&\pm(h_*(\bbr{\partial(M\times[0,\epsilon_n])})-h_*(\bbr{(\partial M)\times[0,\epsilon_n]}))\\
&=& \pm(\partial h_*(\bbr{M\times[0,\epsilon_n]})-h_*(\bbr{(\partial M)\times[0,\epsilon_n]}))
\end{eqnarray*}
which implies that
\[
\flatnorm (\gamma_n-\gamma)\leq\mass_{k+1}(h_*(\bbr{M\times[0,\epsilon_n]}))+
\mass_k(h_*(\bbr{(\partial M)\times[0,\epsilon_n]}))\]
where the flat norm is taken with respect to the compact space $h(M\times[0,1])$.

Define the map
\begin{align*}
\tilde h\colon M\times[0,1] & \longrightarrow X\times M  \\
                      (u,t) & \longmapsto(h(u,t),u).
\end{align*}
Hence $\tilde h$ is a homeomorphism on its image and
$h=\proj_1 \tilde h$ where $\proj_1\colon X\times M\to X$ is the projection
which is $1$-Lipschitzian. Therefore
\begin{eqnarray*}
\mass_{k+1}(h_*(\bbr{M\times[0,\epsilon_n]}))&\leq&
\mass_{k+1}(\tilde h_*(\bbr{M\times[0,\epsilon_n]}))\\
&=&\mass_{k+1}(\bbr{\tilde h(M\times[0,\epsilon_n])})\\
&=&\calH^{k+1}(\tilde h(M\times[0,\epsilon_n])).
\end{eqnarray*}
By \refT{finiteHausdorff}, $\calH^{k+1}(\tilde h(M\times[0,1]))<\infty$.  The Lebesgue Bounded Convergence Theorem implies that
\[\calH^{k+1}(\tilde h(M\times[0,\epsilon_n])\,\longrightarrow\,\calH^{k+1}(\tilde h(M\times\{0\}))\,\textrm{ as }n\to\infty\]
and $\calH^{k+1}(\tilde h(M\times\{0\}))=0$ because $\dim \tilde h(M\times\{0\})\leq k$.
This implies that
\[\lim_{n\to\infty}\mass_{k+1}(h_*(\bbr{M\times[0,\epsilon_n]}))\,=\,0.\]
A completely analoguous argument shows that
\[\lim_{n\to\infty}\mass_k(h_*(\bbr{(\partial M)\times[0,\epsilon_n]}))\,=\,0.\]
Thus $\lim_{n\to\infty}\flatnorm(\gamma_n-\gamma)=0$ which  implies the weak convergence.
\end{proof}

\begin{rem} Actually, 
                   \zz{
                     see referee math comment (2). Please check
                     english. I am not sure we should state in the
                     proposition that indeed we get flat convergence,
                     since when we refer to the statement we are only
                     interesteed by weak convergence. I suggest to
                     only this remark. OK?
                   }
the above proof shows that SA-convergence
implies convergence in the flat norm, which is stronger than weak
convergence, but we will not use this fact.
\end{rem}


\section{PA forms}\label{S:PAForms}


The
aim of this section is to construct the contravariant functor of PA forms,
\[\ompa^*\colon \SA\to\CDGA,
\]
(see \refT{omPAfunctor}), which we will show in \refS{Equivalence}
to be weakly equivalent to $\Apl(-;\BR)$. The construction is in two stages. We will first build a functor $\omin^*$ of \emph{minimal forms} in Section \ref{S:MinimalForms}.
Unfortunately this functor is not weakly equivalent to  $\Apl(-;\BR)$ because it does not satisfy the Poincar\'e Lemma. Building
on $\omin^*$, we will then define $\ompa^*$ in Section \ref{SS:PAForms} which will resolve this issue. These PA forms are defined by integration  of minimal forms along the fiber,
but to define this integration correctly,  we will need the notion of strongly continuous family of chains which is given in \refS{scc}.


\subsection{Semi-algebraic cochains and smooth forms}\label{S:SACochains}


Let $X$ be a semi-algebraic set.
We consider the vector space of semi-algebraic cochains with values in $\BR$ as
the linear dual of the chains, i.e.~we let
\[\Coch^k(X)=\hom(\Ch_k(X),\BR).\]
This defines a cochain complex of real vector spaces $\Coch^*(X):=\oplus_{k\geq0}\Coch^k(X)$
whose coboundary $\delta\colon\Coch^k(X)\to\Coch^{k+1}(X)$ is defined as the adjoint
of the boundary $\partial\colon\Ch_{k+1}(X)\to\Ch_k(X)$.

The real number giving the value of a cochain $\lambda\in\Coch^k(X)$ on a chain $\gamma\in\Ch_k(X)$ is denoted
by $\langle\lambda ,\gamma\rangle$. By convention this value is $0$ if  $\lambda$ and $\gamma$
have different degree.

One has a contravariant functor $\Coch^*\colon\SA\to\Chains^*(\BR)$ with values in cochain
complexes over $\BR$. In particular a semi-algebraic map $f\colon X\to Y$ induces a map of cochains
$f^*\colon\Coch^*(Y)\to\Coch^*(X)$ defined by $\langle f^*(\lambda),\gamma\rangle=
\langle \lambda,f_*(\gamma)\rangle$. When $f\colon X\hookrightarrow Y$ is an inclusion
we write $\lambda|_X:=f^*(\lambda)$.

\begin{lemma}
\label{L:smcoch}
Let $W\subset \BR^N$ be a semi-algebraic smooth submanifold
 of dimension $m$  and let $\omega\in\osmooth^k(W)$ be a smooth differential form.
There is a well defined linear map
\begin{align*}
\langle\omega,-\rangle\colon\Ch_k(W) & \longrightarrow \BR \\
                                \gamma & \longmapsto \langle \omega,\gamma\rangle=\sum n_i\int_{V_i}\omega
 \end{align*}
 when $V_i$ are bounded semi-algebraic
 smooth oriented submanifolds in $W$ and $n_i$ are integers
such that $\gamma=\sum n_i\bbr{V_i}$.
Moreover, the map
\begin{align}
\langle-,-\rangle\colon \osmooth^*(W) & \hookrightarrow\Coch^*(W) \label{E:osmoothinclcoch}  \\
\omega & \mapsto\langle\omega,-\rangle \notag
\end{align}
 is an inclusion.
\end{lemma}
\begin{proof}
By \refP{altSAcur},
a chain $\gamma\in\Ch_k(W)$ can be represented by a linear combination of smooth
submanifolds $V_i\subset W$ of finite volume. Further, the smooth form $\omega$
is bounded on the compact support of $\gamma$. Therefore each integral $\int_{V_i}\omega$
converges and it is clear that the value of the linear combination depends only on the chain $\gamma$.

For the injectivity of the map \refN{E:osmoothinclcoch}
 we need to show that the value of $\omega\in\osmooth^k(W)$ at any point $x\in W$ is completely
characterized by the values of $\langle\omega,-\rangle$ on semi-algebraic chains in $W$.
Indeed consider the $m$-dimensional affine subspace
$T_xW\subset\BR^N$ tangent to $W$ at $x$. The orthogonal projection $\pi$
 of a neighborhod of $x$ in $W$  onto $T_xW$ is a semi-algebraic diffeomorphism onto its image.
For an orthonormal $k$-multivector ${\xi}$
in $T_xW$ and a small $\epsilon>0$, let  $\xi[\epsilon]$ be a $k$-dimensional cube in $T_xW$ based at $x$
with edges of length $\epsilon$ in the directions of $\xi$. Then
$\bbr{{\pi^{-1}(\xi[\epsilon])}}\in\Ch_k(W)$ and the value of $\omega(\xi)$ is given by
\[\lim_{\epsilon\to 0}\epsilon^{-k}
\langle \omega\,,\,\bbr{{\pi^{-1}(\xi[\epsilon])}}\rangle.\]
\end{proof}

\begin{definition}
\label{D:smcoch}
Let $W$ be a semi-algebraic smooth submanifold of $\BR^N$. A  cochain in $\Coch^*(W)$
which is in the image of $\osmooth^*(W)$ under the map  \refN{E:osmoothinclcoch} is called \emph{smooth}.
\end{definition}

\begin{lemma}[Strong K\"unneth formula for smooth  cochains.]
\label{L:strKu-sm}
Let $W_1$ and $W_2$ be two semi-algebraic smooth submanifolds in $\BR^{n_1}$ and  $\BR^{n_2}$
 and let $\omega\in\osmooth^*{(W_1\times W_2)}$.
 If, for all $\gamma_1\in\Ch_*(W_1)$ and $\gamma_2\in\Ch_*(W_2)$,
 $\langle\omega,\gamma_1\times\gamma_2\rangle=0$, then $\omega=0$.
\end{lemma}
\begin{proof}
The proof is analogous to the proof of  the injectivity of the map \refN{E:osmoothinclcoch}
in \refL{smcoch}.
\end{proof}


\subsection{Minimal forms}\label{S:MinimalForms}


In this section we give the first version of ``semi-algebraic differential forms'' following  \cite[Section 8.3]{KoSo:def}.

Let $X\subset\BR^N$ be a semi-algebraic set and let $f_0,f_1,\cdots,f_k\colon X\to\BR$ be
semi-algebraic functions.
We will define a cochain
\[\lambda(f_0;f_1,\cdots,f_k)\in\Coch^k(X)\]
which, in the case of $X$ and the $f_i$'s smooth, is just the smooth cochain
$f_0\,df_1\wedge\cdots\wedge df_k\in\osmooth^k(X)$.
To define  $\lambda(f_0;f_1,\cdots,f_k)$ in the general case, set 
$$f=(f_0,f_1,\cdots,f_k)\colon X\longrightarrow\BR^{k+1}$$ 
which is a semi-algebraic map.
Recall that an element of $\Ch_{k}(\BR^{k+1})$ is a $k$-current in $\BR^{k+1}$, and hence can be evaluated on
smooth $k$-forms with compact support in $\BR^{k+1}$.
For a semi-algebraic $k$-{\current} $\gamma\in \Ch_k(X)$, define
\[\langle \lambda(f_0;f_1,\cdots,f_k)\,,\,\gamma\rangle=
f_*(\gamma)(\rho \cdot x_0\,dx_1\wedge\cdots\wedge dx_k),
\]
where $x_{0},\cdots,x_{k}$ are the coordinates in $\BR^{k+1}$ and $\rho\colon\BR^{k+1}\to\BR$ is a smooth bump function
with compact support
 that takes the value $1$ on $\spt(f_*(\gamma))$. Clearly the result is independent of the choice of $\rho$, and  abusing notation we will simply
write $f_*(\gamma)(x_0\,dx_1\wedge\cdots\wedge dx_k)$.

\begin{definition}
We denote by $\omin^k(X)$ the subgroup of $\Coch^k(X)$ generated by the cochains\\
$\lambda(f_0;f_1,\cdots,f_k)$. Its elements are called the \emph{minimal forms}.
\end{definition}

The pullback of a minimal form along a semi-algebraic map is again a minimal form.
More precisely, let $g\colon X\to Y$ be a semi-algebraic map and let $f_i\colon Y\to\BR$ be semi-algebraic functions
for $0\leq i\leq k$. It is immediate to check that
\[g^*(\lambda(f_0;f_1,\cdots,f_k))=\lambda(f_0\circ g;f_1\circ g,\cdots,f_k\circ g).\]

The following formula implies that the coboundary of a minimal form is also a minimal form.

\begin{prop}
$\delta(\lambda(f_0;f_1,\cdots,f_k))\,=\,\lambda(1;f_0,f_1,\cdots,f_k)$.
\end{prop}
\begin{proof}
Let $\gamma\in \Ch_{k+1}(X)$. Recall that the boundary $\partial$ on $\Ch_{k+1}(X)\subset\calD_{k+1}(X)$ is defined
as the adjoint of $d$ on $\calD^{k}(\BR^N)$. Thus using the definitions and \refC{chaincomplex}, we have
\begin{eqnarray*}
\langle \delta(\lambda(f_0;f_1,\cdots,f_k)\,,\,\gamma\rangle&=&
\langle \lambda(f_0;f_1,\cdots,f_k\,,\,\partial\gamma\rangle\\&=&
\left((f_0,f_1,\cdots,f_k)_*(\partial\gamma)\right)(x_0dx_1\cdots dx_k)\\&=&
\left( \partial\left((f_0,f_1,\cdots,f_k)_*(\gamma)\right)\right)(x_0dx_1\cdots dx_k)\\&=&
\left( (f_0,f_1,\cdots,f_k)_*(\gamma)\right)(dx_0dx_1\cdots dx_k)\\&=&
\left( (1,f_0,f_1,\cdots,f_k)_*(\gamma)\right)(tdx_0dx_1\cdots dx_k)\\&=&
 \langle\lambda(1;f_0,f_1,\cdots,f_k)\,,\,\gamma\rangle.
\end{eqnarray*}
\end{proof}

The above implies that $\omin^*(X)$ is a cochain complex. We will later define an algebra structure on it using 
 a certain cross product.  However, in order to prove that this product is well defined, we first need to show in the three following propositions that
minimal forms are well approximated by smooth forms.

First we have continuity of $\langle\mu,-\rangle$.
\begin{prop}
\label{P:contSAmin}
Let $X$ be a semi-algebraic set and let $\mu\in\omin^k(X)$.
 If $\gamma_n\toSA\gamma$ in $\Ch_k(X)$, then $\lim_{n\to\infty}\langle\mu,\gamma_n\rangle=\langle\mu,\gamma\rangle$.
\end{prop}
\begin{proof}
By linearity, it is enough to prove this when $\mu=\lambda(f_0;f_1,\cdots,f_k)$ for some semi-algebraic functions $f_j\colon X\to\BR$.
Set $f=(f_0,f_1,\cdots,f_k)\colon X\to\BR^{k+1}$.
 By Propositions \ref{P:preserveconvSA} and \ref{P:SAcont},  $f_*(\gamma_n)\rightharpoonup    f_*(\gamma)$, and so
\begin{eqnarray*}
\lim_{n\to\infty}\langle\mu,\gamma_n\rangle&=&
\lim_{n\to\infty} \left(f_*(\gamma_n)(x_0dx_1\dots dx_k)\right)\\
&=&\left(f_*(\gamma)(x_0dx_1\dots dx_k)\right)\\
&=&\langle\mu,\gamma\rangle
\end{eqnarray*}
Notice that we have implicitly  used the fact that the bump function $\rho$ in front of $x_0dx_1\dots dx_k$ can be chosen to be the same
for all $\gamma_n$ because, by the SA convergence, $\cup_{n=1}^\infty\spt(\gamma_n)$ is included in some compact set.
 \end{proof}
We next show   that a minimal form is smooth at a ``generic point'' and that any minimal form
is determined by its values at such points.
\begin{prop}
\label{P:minaesm}
Let $X$ be a semi-algebraic set and let $\mu\in\omin^k(X)$.
 There exists a semi-algebraic smooth  submanifold $W$ that is open and dense in $X$ and
such that $\mu|_W$ is smooth.
\end{prop}
\begin{proof}
Suppose that $\mu=\sum_{i=1}^p\lambda(f_0^{(i)};f_1^{(i)},\cdots,f_k^{(i)})$ for some semi-algebraic functions $f_j^{(i)}\colon X\to\BR$.
Take a stratification of $X$ into smooth submanifolds on which the $f_j^{(i)}$ are smooth.
Let $W$ be the union of the maximal strata.  This is open and dense in $X$ and
 $\mu|_W=\sum_i f_0^{(i)}\,df_1^{(i)}\dots df_k^{(i)}$ is a smooth form.
 \end{proof}
\begin{prop}
\label{P:mindense}
Let $X$ be a semi-algebraic set and let $\mu\in\omin^k(X)$. 
 If $X_0\subset X$ is a dense semi-algebraic subset and if
$\mu|_{X_0}=0$ then $\mu=0$.
\end{prop}
\begin{proof}
 Let $\gamma\in \Ch_k(X)$.
Since $X_0$ is dense in $X$, by \refP{densitySA} there exists a sequence $(\gamma_n)_{n\geq1}$ in
$\Ch_k(X_0)$ converging semi-algebraically to $\gamma$. By hypothesis, $\langle\mu,\gamma_n\rangle=0$,
and we conclude by \refP{contSAmin} that $\mu=0$ since it is 0 when evaluated at an arbitrary chain $\gamma$.
\end{proof}

\begin{prop}
\label{P:strKu-min}
Let $X_1$ and $X_2$ be  semi-algebraic sets and let $\mu\in\omin^k(X_1\times X_2)$.
If $\langle\mu,\gamma_1\times\gamma_2\rangle=0$ for all $\gamma_1\in\Ch_*(X_1)$ and
$\gamma_2\in\Ch_*(X_2)$  then $\mu=0.$
\end{prop}
\begin{proof}
By \refP{minaesm} there exists a semi-algebraic smooth submanifold $W$, open and dense in $X_1\times X_2$,
such that $\omega:=\mu|_W$ is smooth. We show that $\omega=0$. Let $(x_1,x_2)\in W$. There exist
smooth neighborhhoods $U_i\subset X_i$ of $x_i$ for $i=1,2$  such that $U_1\times U_2\subset W$.
By \refL{strKu-sm}, $\omega|_{U_1\times U_2}=0$. This implies that $\omega=0$.
Since $W$ is dense in $X$,  we deduce the desired result by \refP{mindense}.
\end{proof}
\begin{prop}
\label{P:omincross}
Let $X_1$ and $X_2$ be  semi-algebraic sets.
There is a degree-preserving linear map
\[\times\colon\omin^*(X_1)\otimes \omin^*(X_2)\longrightarrow\omin^*(X_1\times X_2)\]
given by 
\fn{Victor noticed that the sign in the next formula is probably not needed, even if it does not
hurt. I will check this in LLN{\bf Pascal:} I checked that in Bredon ``Geometry and topology'' Proposition VI.3.1. page 323 and page 315 he uses
the same Koszul sign as mine. So I prefer to keep it. \pl 10july 2008. Actually Victor was right (as always!). The Koszul sign in \refN{E:omincross}
was WRONG (look to the proof). Indeed formula \refN{E:defomincross} implies that there is no sign. Similarly the sign in \refN{E:omPAcross} is removed.}
\begin{equation}
\label{E:omincross}
\langle\mu_1\times\mu_2,\gamma_1\times\gamma_2\rangle=
 \langle\mu_1,\gamma_1 \rangle\cdot \langle\mu_2,\gamma_2 \rangle
\end{equation}
for $\mu_i\in\omin^*(X_i)$ and $\gamma_i\in\Ch_*(X_i)$.  This formula satisfies the Leibniz rule
\[\delta(\mu_1\times\mu_2)=\delta(\mu_1)\times\mu_2+(-1)^{\deg(\mu_1)}\mu_1\times \delta(\mu_2).\]
Further, let $T\colon X_1\times X_2\to X_2\times X_1$ be the twisting map, given by $(x_1,x_2)\mapsto (x_2,x_1)$. Then
\[T^*(\mu_2\times\mu_1)=(-1)^{\deg(\mu_1)\deg(\mu_2)}\mu_1\times\mu_2.\]
\end{prop}
\begin{proof}
Let $k_i$  be nonnegative integers, let $f_0^{(i)},\cdots,f_{k_i}^{(i)}\colon X_i\to\BR$
be semi-algebraic functions for $i=1,2$,
and set $\mu_i=\lambda(f_0^{(i)};f_1^{(i)},\dots,f_{k_i}^{(i)})\in\omin^{k_i}(X_i)$.
Consider also the projections $\proj_i \colon X_1\times X_2\to X_i$.
Set
\begin{equation}
\label{E:defomincross}
\mu_1\times\mu_2=\lambda\left((f_0^{(1)}\proj_1)\cdot(f_0^{(2)}\proj_2)\,;\,f_1^{(1)}\proj_1,\cdots,f_{k_1}^{(1)}\proj_1,
f_1^{(2)}\proj_2,\cdots,f_{k_2}^{(2)}\proj_2\right)
\end{equation}
and extend bilinearly.  It is straightforward to check that this definition satisfies  equation \refN{E:omincross} of the proposition.

The fact that  the minimal form $\mu_1\times \mu_2$
is characterized by equation \refN{E:omincross}, and in particular that the right side
of \refN{E:defomincross} is independent of
the choice of the representatives of the $\mu_i$'s, is a consequence of \refP{strKu-min}.

The Leibniz and twisting  formulas are  consequences of the corresponding formulas
for chains in \refP{chaincross}.
\end{proof}

For $\mu_1,\mu_2\in\omin^*(X)$, we define a multiplication on $\omin^*(X)$ by
\[\mu_1\cdot\mu_2=\Delta^*(\mu_1\times\mu_2)\]
where $\Delta\colon X\to X\times X $ is the diagonal map.
It is immediate from the previous proposition that this multiplication satisfies the Leibniz formula and
is graded commutative.

In conclusion we have
\begin{thm}
\label{T:ominfunctor}
The above construction of minimal forms defines a contravariant functor $\omin^*\colon\SA\to\CDGA$.
\end{thm}
\begin{rem}
\label{rmk:ominnotdeRham}
Note that we do not have the analog of DeRham theorem for the functor $\omin^*$;
in general $H(\omin^*(X))$ is not isomorphic to $H^*(X;\BR)$.
For example, consider the contractible semi-algebraic set
$X=[1,2]$ and the minimal $1$-form $\mu=\lambda(f_0;f_1)$ with $f_0,f_1\colon X\to\BR$
defined by $f_0(t)=1/t$ and $f_1(t)=t$ for $t\in [1,2]$. In other words, $\mu = dt/t$, and this is a smooth form which is a $\delta$-cocycle but it is not a $\delta$-coboundary in $\omin^*(X)$ because the
map $t\mapsto \log(t)$ is not semi-algebraic.  This is an issue we will get around by  enlarging the cochain complex $\omin^*$ in section \ref{SS:PAForms}.
\end{rem}


\subsection{Strongly continuous families of chains}\label{S:scc}


\begin{definition}
\label{D:scc}
Let $f\colon Y\to X$ be a semi-algebraic map.
A \emph{strongly continuous family of chains} or, shortly, a \emph{(strongly) continuous chain of dimension $l$ over $X$ along $f$} is a map
\[\Phi\colon X\longrightarrow\Ch_l(Y)\]
such that there exist
\begin{enumerate}
\item a finite semi-algebraic stratification $\{S_\alpha\}_{\alpha\in I}$ of $X$, and,  for each $\alpha\in I$,
\item an oriented compact  semi-algebraic manifold $F_\alpha$ of dimension $l$ and
\item a semi algebraic map $g_\alpha \colon\overline{S_\alpha}\times F_\alpha\to Y$, where $\overline{S_\alpha}$
is the closure of $S_\alpha$ in $X$, satisfying
\begin{enumerate}
\item 
the diagram

\[\xymatrix{\overline{S_\alpha}\times F_\alpha\ar[r]^-{g_\alpha}\ar[d]_{\proj}&Y\ar[d]^f\\
\overline{S_\alpha}\ar@{^(->}[r]&X}
\]
commutes, and
\item for each $\alpha\in I$ and $x\in\overline{S_\alpha}$,
$\Phi(x)=g_{\alpha*}(\bbr{\{x\}\times F_\alpha})$.
\end{enumerate}

\end{enumerate}

We say that the family $\{(S_\alpha,F_\alpha,g_\alpha)\}_{\alpha\in I}$ \emph{trivializes} or \emph{represents}
the  continuous chain $\Phi$ and we denote by $\Chs_l(Y\stackrel{f}\to X)$ the set of  strongly continuous $l$-chains.
\end{definition}

\begin{rem}
Note that if  we had asked that the maps $g_\alpha$ be defined only on $S_\alpha\times F_\alpha$ instead of on the closure
$\overline{S_\alpha}\times F_\alpha$ in the above definition, we would obtain a sort of semi-algebraic parametrized chain in $Y$ over $X$,
lacking a continuity condition.  We would therefore not have the Leibniz formula from \refP{ltimes}.  It appears that Kontsevich and Soibelman had something  weaker in mind in their definition of continuous chains
\cite[Definition 22]{KoSo:def} and  we have thus added \emph{strongly} and the adornment \emph{str} in order
to distinguish our definition from theirs. See discussion in \refS{diffKS}.
\end{rem}

It is clear that if $\{S'_\beta\}_{\beta\in J}$ is a stratification of $X$ refining the stratification  $\{S_\alpha\}_{\alpha\in I}$ of the above definition,
that is, if for each $\beta\in J$ there exists $\alpha=\alpha(\beta)\in I$ such that $S'_\beta\subset S_\alpha$, then there is an induced trivialization
$\{(S'_\beta,F_ \beta,g_ \beta)\}_{\beta\in J}$ with $F_ \beta=F_{\alpha(\beta)}$ and
 $g_\beta=g_{\alpha(\beta)}|_{\overline{S'_\beta}\times F_\beta}$.

Consider the set
\[\map(X,\Ch_l(Y)):=\{\Phi\colon X\to\Ch_l(Y)\}\]
of all maps of sets from $X$ to the $l$-chains on $Y$. This set has an abelian group structure induced by that on  $\Ch_l(Y)$.
Moreover, if $\Phi\in\map(X,\Ch_l(Y))$ we  define its boundary $\partial\Phi\in\map(X,\Ch_{l-1}(Y))$ by the formula
\[(\partial\Phi)(x)=\partial(\Phi(x))\,,\,x\in X.\]
Let  $\map(X,\Ch_*(Y))=\oplus_{l\geq 0}\map(X,\Ch_l(Y))$ and let $\Chs_*(Y\to X)=\oplus_{l\geq 0}\Chs_l(Y\to X)$.  It is clear that the former is a chain complex of abelian groups and that the latter is a subset.

\begin{lemma}
\label{L:ChsSubcplx}
$\Chs_*(Y\to X)$ is a chain subcomplex of $\map(X,\Ch_*(Y))$.
\end{lemma}
\begin{proof}
We prove first that each $\Chs_l(Y\to X)$ is a subgroup.

The zero element of $\map(X,\Ch_l(Y))$ is an element of $\Chs_l(Y\to X)$ since it is represented
by $\{(X,\emptyset,X\times\emptyset=\emptyset\hookrightarrow Y)\}$.

Let $\Phi,\Phi'\in \Chs_l(Y\to X)$. By taking a common refinement
 we can suppose that these
 continuous chains are
 represented by $\{(S_\alpha,F_\alpha,g_\alpha)\}_{\alpha\in I}$
and $\{(S_\alpha,F'_\alpha,g'_\alpha)\}_{\alpha\in I}$ respectively. Letting
$g''_\alpha=(g_\alpha,g'_\alpha)\colon
 \overline{S_\alpha}\times(F_\alpha\amalg F'_\alpha)\to Y$, it is clear that
$\Phi+\Phi'$ is represented by $\{(S_\alpha,F_\alpha\amalg F'_\alpha,g''_\alpha)\}_{\alpha\in I}$, so it is also
a continuous chain.

The inverse of a continuous chain can be represented by reversing the orientations of the manifolds $F_\alpha$.

It remains to prove that $\partial(\Chs_l(Y\to X))\subset \Chs_{l-1}(Y\to X)$. Indeed suppose that $\Phi\in \Chs_l(Y\to X)$
is represented by
$\{(S_\alpha,F_\alpha,g_\alpha)\}_{\alpha\in I}$. Then  by \refT{Stokes}, $\partial\Phi$ is represented by
$\{(S_\alpha,\partial F_\alpha,\partial g_\alpha)\}_{\alpha\in I}$ where $\partial g_\alpha$ is the restriction
of $g_\alpha$ to $\overline{S_\alpha}\times\partial F_\alpha$.
\end{proof}

Let $f\colon Y\to X$ be a semi-algebraic map, let $\gamma\in\Ch_k(X)$ and let
$\Phi\in\Chs_l(Y\to X)$. We construct a chain $\gamma\ltimes\Phi\in \Ch_{k+l}(Y)$ as follows.
 Take a trivialization $\{(S_\alpha,F_\alpha,g_\alpha)\}$ of $\Phi$. We can suppose that
  the stratification is fine enough to be adapted to
$\gamma$ in the sense that there exist integers $n_\alpha$ such that
\[\gamma=\sum_\alpha{n_\alpha}\cdot\bbr{\overline{S_\alpha}}\]
where ${\overline{S_\alpha}}$ are compact oriented semi-algebraic manifolds (take for example a 
 stratification whose restriction to $\spt(\gamma)$ is a suitable triangulation.)
Set
\begin{equation}
\label{E:gammaPhi}
\gamma\ltimes\Phi=\sum_\alpha n_\alpha\cdot g_{\alpha*}(\bbr{\overline {S_\alpha}\times F_\alpha}).
\end{equation}

We will prove in \refP{ltimes} that this operation is well-defined and satisfies the Leibniz formula.
In order to do so we need  the following version of Fubini's Theorem.
 \begin{lemma}
 \label{L:Fubini}
 Let $g\colon S\times F\to Y$ and $g'\colon S\times F'\to Y$ be two semi-algebraic maps such that
 $S$, $F$, and $F'$ are compact oriented semi-algebraic  manifolds with $\dim F=\dim F'$. Suppose that for all $x\in S$ we have $g_*(\bbr{\{x\}\times F})=g'_*(\bbr{\{x\}\times F'})$ in $\Ch_*(Y)$.
 Then $g_*(\bbr{S\times F})=g'_*(\bbr{S\times F'})$.
 \end{lemma}
 \begin{proof}
Set $T=F\amalg-F'$ where $-F'$ is $F'$ with the opposite orientation and consider the map
$f=g\amalg g'\colon S\times T\to Y$.
By linearity, for each $x\in S$ we have $f_*(\bbr{\{x\}\times T})=0$ and what we now have to show is that $f_*(\bbr{S\times T})=0$.

Set $n=\dim(S)+\dim(T)$. Without loss of generality we can assume that $Y=f(S\times T)$, and hence $\dim(Y)\leq n$.
If $\dim(Y)<n$ then the conclusion of the lemma is immediate, so we assume $\dim(Y)=n$.

Let $W$ be a semi-algebraic smooth submanifold that is open and dense in $S\times T$ and let  $Z\subset S\times T$ be the closure of
$$
(S\times T\setminus W)\cup\left\{(s,t)\in W \left|
\quad
\parbox{6cm}{$f$ is not smooth at $(s,t)$ or  it is smooth but $df(s,t)$ is of rank $<n$}
\right.\right\}
$$
Then $\dim f(Z)<n$ and there exists  a non-empty smooth $n$-dimensional semi-algebraic submanifold $Y_0\subset Y\setminus f(Z)$
such that $\spt(f_*(\bbr{S\times T}))\subset \overline{Y_0}$.
Set $X_0=f^{-1}(Y_0)$ and let $f_0$ be the restriction of $f$ to $X_0$. Then
$f_0$ is locally a diffeomorphism onto its image and for $y_0\in Y$
we have $f^{-1}(y_0)=f^{-1}_0(y_0)$.  This set is discrete, and hence finite by compactness of $S\times T$.

The multiplicity in $f_*(\bbr{S\times T})$ of the neighborhood of  a point  $y_0\in Y_0$
is given by the formula
\[\operatorname{mult}(y_0)=\sum_{x\in f^{-1}(y_0)}\sgn(\det(df(x))).\]
We have
\[f^{-1}(y_0)=\{(s_1,t_1^1),\cdots,(s_1,t_1^{r_1}),\cdots,(s_p,t_p^1),\cdots,(s_p,t_p^{r_p})\}\]
with the $s_i$ are all distinct in $S$, and 
so the multiplicity can be rewritten as
\[\sum_{i=1}^p\left\{\sum_{j=1}^{r_i}\sgn(\det(df{(s_i,t_i^j)}))\right\}.\]
If this expression is non-zero then one of the terms in the brackets  has to be non-zero, but this
 contradicts the fact that $f_*(\bbr{\{s_i\}\times T})=0$. Therefore multiplicity at $y_0$ is zero.
 \end{proof}

 \begin{prop}
 \label{P:ltimes}
 The formula  \refN{E:gammaPhi} above defines a natural linear map
 \begin{align*}
 \ltimes\colon \Ch_k(X)\otimes\Chs_l(Y\to X) & \longrightarrow\Ch_{k+l}(Y) \\
 \gamma\otimes\Phi & \longmapsto\gamma\ltimes\Phi
 \end{align*}
 which satisfies the Leibniz formula
 $\partial(\gamma\ltimes\Phi)=(\partial\gamma)\ltimes\Phi+(-1)^{\deg(\gamma)}\gamma\ltimes(\partial\Phi).$
  \end{prop}
  \begin{proof}
We need to prove that the right side of (\ref{E:gammaPhi}) is independent of the choice of trivialization of $\Phi$.
First, it is clear that if we take a refinement of the stratification and consider the induced trivialization, the right side is unchanged.
Therefore it is enough to prove the invariance for two trivializations  $\{(S_\alpha,F_\alpha,g_\alpha)\}$  and
 $\{(S_\alpha,F'_\alpha,g'_\alpha)\}$ with the same underlying stratification.
 But this is an immediate consequence of  \refL{Fubini}.

By linearity of $\ltimes$ and $\partial$, it is enough to check the  Leibniz  formula  when $\gamma=\bbr{S}$, with $S$ a closed subset of $X$ that is
a compact oriented manifold over which $\Phi$ is trivial in the sense that there is a semi-algebraic
map $g\colon S\times F\to Y$ with $F$ a compact oriented manifold and, for each $x\in S$, $\Phi(x)=g_*(\bbr{\{x\}\times F})$.
Using \refT{Stokes} we then get
\begin{eqnarray*}
\partial(\gamma\ltimes\Phi)&=&\partial(g_*(\bbr{S\times F}))\\
&=&g_*(\bbr{(\partial S)\times F})+(-1)^{\dim(S)}g_*(\bbr{S\times (\partial  F)}))\\
&=&(\partial\gamma)\ltimes\Phi+(-1)^{\deg(\gamma)}\gamma\ltimes(\partial\Phi).
\end{eqnarray*}
 \end{proof}

We now define the \emph{pullback} of continuous chains.
\begin{prop}
\label{P:defpb}
Suppose given a pullback of semi-algebraic sets
\[\xymatrix{
Y'=X'\times_XY\ar[r]^-{\hat h}\ar[d]_{f'}\ar@{}[rd]|{pullback}&Y\ar[d]^f\\
X'\ar[r]_h&X
}\]
and $\Phi\in \Chs_l(Y\to X)$. There exists a unique continuous chain
$\Phi'\in \Chs_l(Y'\to X')$ such that, for each $x'\in X'$, $\hat h_*(\Phi'(x'))=\Phi(h(x'))$.
\end{prop}
\begin{proof}
Set $x=h(x')$.
Since $Y'$ is the pullback, the restriction of $\hat h$ to fibers over $x'$ and $x$ induces an isomorphism
$\hat h_{x'}\colon f'^{-1}(x')\iso f'^{-1}(x)$.
Since the support of $\Phi(x)$ is included in the fiber over $x$ we deduce that the value
$\Phi'(x')$ is completely determined as $(h_{x'}^{-1})_*(\Phi(x))$.
This implies the uniqueness.

The fact that $\Phi'$ is indeed a  continuous chain comes from existence of a trivialization
obtained by taking the pullback of a trivialization of $\Phi$ in a straightforward way.
\end{proof}
\begin{definition}
\label{D:pb}
We call the continuous chain $\Phi'$ from the previous proposition the \emph{pullback of $\Phi$ along $h$}
and we denote it by  $h^*(\Phi)$.
In case of an inclusion $h \colon X'\hookrightarrow X$, we also write $\Phi|_{X'}=h^*(\Phi)$.
\end{definition}
It is easy to check that the pullback operation
\[h^*\colon\Chs_*(Y\to X)\longrightarrow\Chs_*(Y'\to X')\]
is a morphism of chain complexes.
Also, given $\Phi_1\in \Chs_{l_1}(Y_1\to X_1)$ and $\Phi_2\in \Chs_{l_2}(Y_2\to X_2)$, we can construct in the obvious
way a continuous chain $\Phi_1\times\Phi_2\in  \Chs_{l_1+l_2}(Y_1\times Y_2 \to X_1\times X_2)$
characterized by
\[(\Phi_1\times\Phi_2)(x_1,x_2)=\Phi_1(x_1)\times\Phi_2(x_2).\]
The Leibniz formula
\[\partial(\Phi_1\times\Phi_2)=\partial(\Phi_1)\times\Phi_2+(-1)^{\deg(\Phi_1)}\Phi_1\times \partial(\Phi_2)
\]
clearly holds.

Now let $F$ be a compact oriented semi-algebraic manifold of dimension $l$ and let $X$ be a semi-algebraic set.
We define the \emph{constant  continuous chain}
 $\fib{F}\in\Chs(X\times F\stackrel{\proj_1}\to X)$ by $\fib{F}(x)=\bbr{\{x\}\times F}$.
\fn{\pl 10july 2008. The end of this section has been moved to a new section below on SA bundle and integration along the fiber}


\subsection{The cochain complex of PA forms}\label{SS:PAForms}


As we explained in Remark \ref{rmk:ominnotdeRham}, $\omin^*(\Delta^k)$ is not acyclic. In this section we will enlarge the
cochain complex $\omin^*(X)$ into a cochain complex $\ompa^*(X)$ which will satisfy  the Poincar\'e lemma.

Let $f\colon Y\to X$ be a semi-algebraic map, let $\Phi\in\Chs_l(Y\stackrel{f}\to X)$ be
 a continuous $l$-chain and let $\mu\in\omin^{k+l}(Y)$.
  We  define a cochain $\fint_\Phi\mu\in\Coch^k(X)$
 by
 \begin{equation}
 \label{E:Lambda}
\langle\fint_\Phi\mu,\gamma\rangle=\langle\mu,\gamma\ltimes\Phi\rangle,\quad\textrm{ for } \gamma\in \Ch_k(X).
  \end{equation}

\begin{definition}
\label{D:ompa}
A \emph{PA form of degree $k$ in $X$} is a cochain in $X$  of the form $\fint_\Phi\mu$ as defined in equation \refN{E:Lambda}
for some semi-algebraic map $f\colon Y\to X$, some continuous chain
 $\Phi\in\Chs_l(Y\stackrel{f}\to X)$, and some minimal form
 $\mu\in\omin^{k+l}(Y)$.
We denote by $\ompa^k(X)\subset\Coch^k(X)$ the subset of all PA forms in $X$.
\end{definition}
\begin{definition}
\label{D:tfi}
Given a closed semi-algebraic subset $A\subset X$,
we say that a PA form $\alpha\in\ompa^k(X)$
is \emph{a trivial fiber integral over $A$} if there exists a continuous chain
$\Phi\in\Chs_l(Y\to X)$ and a minimal form $\mu\in\omin^{k+l}$ such that
$\alpha= \fint_\Phi\mu$ and
$A$ is  the relative closure of a stratum of some trivialization of  $\Phi$.
\end{definition}

Notice that by definition of a PA form there always exists a (semi-open) triangulation of the space such that the PA form
is a trivial fiber integral over the relative closure of each simplex.
 When $\alpha\in\ompa^k(X)$ is a trivial fiber integral over $A$
  then by definition $\alpha|_A=\fint_{\fib{F}}\mu$ for some compact oriented semi-algebraic
 manifold $F$ and some minimal  form $\mu\in\omin^{k+\dim(F)}(A\times F)$, where
 $\fib{F}\in\Chs_{\dim(F)}(A\times F\stackrel{\proj}\to A)$ is the constant continuous chain introduced before
 \refD{PAfib}.

 \begin{prop}
 \label{P:ompasbgrp}
 $\ompa^k(X)$ is a subgroup of $\Coch^k(X)$ .
 \end{prop}
 \begin{proof}
Let $\alpha_i=\fint_{\Phi_i}\mu_i$ be two PA forms in $\ompa^k(X)$ with
$\Phi_i\in\Chs_{l_i}(Y_i\stackrel{f_i}\to X)$ and $\mu_i\in\omin^{k+l_i}(Y_i)$, for $i=1,2$.

We first show that we can assume that $l_1=l_2$. For concreteness,   
suppose that $l_1\leq l_2$ and set $r=l_2-l_1\geq0$.
There is a natural isomorphism $\Chs_r([0,1]^r\to*)\cong\Ch_r([0,1]^r)$ given by $\Phi\mapsto\Phi(*)$, and hence
we can realize $\bbr{[0,1]^r}$ as an element of  $\Chs_r([0,1]^r\to*)$.
Consider also  the smooth
minimal form $dt_1\cdots dt_r=\lambda(1;t_1,\dots,t_n)\in\omin^r([0,1]^r)$ where $t_1,\dots,t_r$ are the coordinates in $[0,1]^r$.
Then
\[\alpha_1=\fint_{\Phi_1\times\bbr{[0,1]^r}}\mu_1\times dt_1\cdots dt_r ,\]
where $\Phi_1\times\bbr{[0,1]^r}\in \Chs_{l_2}(Y_1\times[0,1]^r\to X\times *)$
so we  can assume that $l_1=l_2$.

Now consider the obvious continuous chain $\Phi_1+\Phi_2\in\Chs_{l_1}(Y_1\amalg Y_2\to X)$
and the minimal form $\mu_1+\mu_2\in\omin^{k+l_1}(Y_1\amalg Y_2)$. Then
\[\fint_{\Phi_1}\mu_1+\fint_{\Phi_2}\mu_2=\int_{\Phi_1+\Phi_2}\mu_1+\mu_2,
\]
and so the sum of two PA forms is still a PA form. It is obvious that $0$ is a PA form as well as is the inverse of a PA form.
\end{proof}
The coboundary of a PA form is again a PA form because of the following.\fn{\pl 10 july 2008.
There was a wrong sign in the previous formula of this lemma. This is lucky because otherwise all my signs
in the formality of disks paper would have been wrong!}
\begin{lemma}
\label{L:deltaompa}
$\delta(\fint_\Phi\mu)=\fint_\Phi \delta \mu+(-1)^{\deg(\mu)-deg(\Phi)}\fint_{\partial\Phi} \mu $.
\end{lemma}
\begin{proof}
Evaluate both side on a chain $\gamma$ using the Leibniz formula for the boundary of $\gamma\ltimes\Phi$
and unravelling definitions.
\end{proof}
It is also clear that $\omin^k(X)\subset\ompa^k(X)$ because a minimal form $\mu$ on $X$ can be
written as $\fint_{\fib{*}} \mu$ where ${\fib{*}}\in\Chs_0(X\stackrel{=}\to X)$ is the
constant continuous $0$-chain defined by ${\fib{*}}(x)=\bbr{\{x\}}$.

If $g\colon X'\to X$ is a semi-algebraic map then $g^*(\ompa^*(X))\subset\ompa^*(X')$ because of the formula
$g^*(\fint_\Phi\mu)=\int_{g^*\Phi}g^*\mu$, which is straightforward to check.

The following\fn{\pl 11july2008 Added} is a direct consequence of \refP{dimC}.
\begin{prop}\label{P:PAdegdim0}
Let $\alpha\in\ompa(X)$. If $\deg(\alpha)>\dim(X)$ then $\alpha=0$.
\end{prop}

In order to finally prove that $\ompa^*(X)$ is an algebra we need the following extension
of \refP{strKu-min}
to PA forms.

\begin{prop}
\label{P:strKu-PA}
Let $X_1$ and $X_2$ be  semi-algebraic sets and let $\alpha\in\ompa^k(X_1\times X_2)$.
If $\langle \alpha,\gamma_1\times\gamma_2\rangle=0$ for all $\gamma_1\in\Ch_*(X_1)$ and
$\gamma_2\in\Ch_*(X_2)$  then $\alpha =0.$
\end{prop}

The proof of this proposition is more complicated than for the case of minimal forms,
in particular because it is unclear whether,
for a PA form $\alpha$ on $X$, there exists a dense \emph{semi-algebraic} smooth manifold on
which $\alpha$ is smooth.

To prove \refP{strKu-PA}, first we establish two lemmas.
The first is an integral representation of PA forms. To state it, we introduce the following notation.
Suppose that $X=(0,1)^n$ is an open $n$-cube and let $x_1,\dots,x_n$ be the standard coordinates.
For a subset $K=\{i_1,\dots,i_k\}\subset \nbar=\{1,\dots,n\}$ with $1\leq i_1<\dots<i_k\leq n$
we set
\[dx_K=dx_{i_1}\wedge\dots\wedge dx_{i_k}.\]
\begin{lemma}
\label{L:repintPA}
Let $X=(0,1)^n$  and $\alpha\in\ompa^k(X)$. Assume that $\alpha$ is a trivial fiber integral over $X$.
There exists a subset $X_0\subset X$ such that the complement $T= X\setminus X_0$ is a closed smoothly
stratifiable set
of codimension $1$ and there exists, for each subset
$K\subset \nbar$ of cardinality $k$, a smooth function
$G_K\colon X_0\to \BR$ satisfying the following:  If $V$ is a  semi-algebraic smooth oriented submanifold of dimension $k$ in $\BR^n$ such that
$\overline V\subset X_0$, then
\begin{equation}
\label{E:repintPA}
\langle\alpha,\bbr{V}\rangle=\sum_K\int_V G_K\,dx_K,
\end{equation}
where the sum runs over all subsets $K\subset\nbar$ of cardinality $k$.
\end{lemma}
\begin{rem}
Here by $\int_V G_K\,dx_K$ we mean $\int_Vi^*(G_K\,dx_K)$, where $i\colon V\hookrightarrow X$ is the inclusion
and $i^*$ is the pullback of differential forms.
\end{rem}
\begin{rem} Since $X_0$ may not be semi-algebraic, we cannot quite speak
of the restriction $\alpha|_{X_0}$,  but we will nevertheless abuse notation and say that $\alpha|_{X_0}$ is given
by the integral represention \refN{E:repintPA}.
If we just ask for the integral representation  without smoothness assumption on the
 $G_K$'s then  we can assume that  $X_0$ is semi-algebraic (see formula \refN{E:repintproof} below which gives an integral representation
 on the semi-algebraic dense subset $X'$).
\end{rem}
\begin{proof}
Since $\alpha$ is a trivial fiber integral over $X$
there exists a compact oriented semi-algebraic manifold $Y$ of
dimension $l$ and a minimal form $\mu\in\omin^{k+l}(X\times Y)$ such that
$\alpha=\fint_{Y}\mu$. By taking a finite cubification of $Y$ it is easy
to see that we can suppose that $Y=[0,1]^l$ (by replacing $\mu$ by the sum of the restrictions
of $\mu$  over a finite partition ``up to codimension 1 faces'' of 
$Y$ into cubes $[0,1]^l$).

Now, $\mu=\sum_{j=1}^s\lambda(f^{j}_0;f^{j}_1,\dots,f^{j}_{k+l})$ for some
semi-algebraic functions $f^{j}_i\colon X\times Y\to \BR$.
There exists a codimension $1$ semi-algebraic subset $Z\subset X\times Y$ on the complement of
which the  $f^{j}_i$ are smooth, and hence
\[\mu|_{(X\times Y)\setminus Z}= \sum_{j=1}^s f^{j}_0\,df^{j}_1\dots df^{j}_{k+l}.\]
Denote by $x_1,\cdots, x_n$ and $y_1,\cdots,y_l$ the coordinates on $X$ and $Y$.
Then  on $(X\times Y)\setminus Z$,
\begin{equation}
\label{E:gKL}
\sum_{j=1}^s f^{j}_0\,df^{j}_1\dots df^{j}_{k+l}=\sum_{K,L}g_{K,L}\,dy_L\,dx_K
\end{equation}
where the sum runs over subsets $K\subset\nbar$ and $L\subset \lbar$ with $|K|+|L|=k+l$
and  $g_{K,L}$ are semi-algebraic smooth functions on $(X\times Y)\setminus Z$
 (
$g_{K,L}$ is a signed sum of $(k+l)!$ products of $f^{j}_0$ with first order partial derivatives of $f^{j}_i$ for $1\leq i\leq k+l$).
When $K\subset\nbar$ is of cardinality $k$ then $L=\lbar$ and we set $g_K=g_{K,\lbar}$.

For $x\in X$, consider the slice
\begin{equation}
\label{E:slice}
Z(x)=\{y\in Y:(x,y)\in Z\}.
\end{equation}
Now set
$T'=\{x\in X:\dim Z(x)=l\}$.  This is a semi-algebraic subset of $X$ as can be seen from the semi-algebraic local triviality theorem \cite[Proposition 9.3.1]{BCR:GAR} applied
to $Z\stackrel{\proj}\to X$.  Further, $\dim T'<n$ because
$\dim Z<n+l$.

Fix $K\subset\nbar$ of cardinality $k$.
 For $x\in X\setminus T'$, the function
\begin{align*}
g_K(x,-)\colon Y & \longrightarrow\BR \\ 
y & \longmapsto g_K(x,y)=g_{K,\lbar}(x,y)
\end{align*}
is almost everywhere defined and, by \cite[Proposition 2.9.1 and remark that follows it]{BCR:GAR}, it is semi-algebraic on its domain.
Set $g_K^+=\max(g_K,0)$ and $g^-_K=g^+_K-g_K$.
Consider the hypographs of $g_K^\pm$,
\[U^\pm_K:=\{(x,y,t)\in ((X\times Y)\setminus Z)\times\BR: 0\leq g^\pm_K(x,y)\leq t\}.\]
With analogous notation for the slice as in \refE{slice}, we thus have, for $x\in X\setminus T'$,
\[\calH^{l+1}(U^\pm_K(x))=\int_{Y} g^\pm_K(x,y)\,dy_{\lbar}\,\,,\]
and $g_K(x,-)$ is integrable if and only if $U^+_K(x)$ and $U^-_K(x)$ are
of finite $(l+1)$-volume.
By \cite[Theorem 3]{CLR:nat}, the set
\[T''_K:=\{x\in X\setminus T':g_K(x,-)\textrm{ is not integrable over }Y\}\]
is semi-algebraic.
\vspace{3mm}

We now claim that $\dim(T''_K)<n$.

Let $\{i_{k+1},\dots,i_n\}$ be the complement of the set $K$ in $\nbar$ and
define the top degree minimal form
\[\mu_K=\sum_{j=1}^s\lambda(f^{j}_0;f^{j}_1,\dots,f^{j}_{k+l},x_{i_{k+1}},\dots,x_{i_{n}})
\in\omin^{n+l}(X\times Y)\]
which is also smooth on $(X\times Y)\setminus Z$. Using (\ref{E:gKL}), we have that
\[\pm\sum_{j=1}^s f^{j}_0\,df^{j}_1\dots df^{j}_{k+l}\,dx_{i_{k+1}}\cdots dx_{i_{n}}\,=\,
g_{K}\,dy_{\lbar}\,dx_{\nbar}
\]
where the sign comes from the signature of the permutation of the shuffle $(K,\nbar\setminus K)$.
For concreteness, we will assume that this signature is $+1$.
Set $$A^\pm_K=\{(x,y)\in X\times Y\setminus Z: g^\pm_K(x,y)>0\}.$$
To avoid issues arising from the noncompactness of the open cube $X$,
fix $\epsilon>0$ and let $X_{\epsilon}=[\epsilon,1-\epsilon]^n$ and
$A^\pm_{K,\epsilon}=A^\pm_K\cap (X_{\epsilon}\times Y)$. Consider the chains $\bbr{A^\pm_{K,\epsilon}} \in\Ch_{n+l}(X\times Y)$
 obtained by taking the $(n+l)$-dimensional strata of $A^\pm_{K,\epsilon}$ equipped with the orientation of $X\times Y$.
Then
\[\int_{X_{\epsilon}\times Y}|g_K(x,y)|\,dy_{\lbar}\, dx_{\nbar}=
\langle\mu_K,\bbr{A^+_{K,\epsilon}}\rangle-\langle\mu_K,\bbr{A^-_{K,\epsilon}}\rangle<\infty,\]
and hence $g_K$ is integrable over $X_{\epsilon}\times Y$.
By Fubini's Theorem (and since $\epsilon>0$ is arbitrary), we deduce that
 for almost every $x\in X\setminus T'$, $g_K(x,-)$ is integrable,
which establishes the claim.

Now set $X'=X\setminus\overline{T'\cup\cup_{K\subset\nbar,|K|=k}T''_K}$.  This is an open
dense semi-algebraic subset of $X$. For each $x\in X'$, $g^\pm_K(x,-)$ is integrable on $Y$ and
we set
\begin{align*}
G^\pm_K\colon X' & \longrightarrow\BR  \\
x & \longmapsto\int_Yg^\pm_K(x,y)\,dy_{\lbar}
\end{align*}
and $G_K=G^+_K-G^-_K$. If $V$ is a semi-algebraic $k$-dimensional
oriented smooth submanifold with closure in $X'$, then since $V\cap T'=\emptyset$ we have that $Z\cap(V\times Y)$
is negligible in $V\times Y$ and
using \eqref{E:gKL} we have
\[\langle\alpha,\bbr{V}\rangle=\int_{(V\times Y)\setminus Z}\sum_{K,L}g_{K,L}\,dy_L\,dx_K.\]
If $|L|<l$ then $|K|>k$ and since $V$ is $k$-dimensional in the direction of the $x_i$'s and $F$ is in the direction
of the $y_i$'s, we have that the restriction of the differential form $g_{K,L}\,dy_L\,dx_K$ to the tangent space
of $(V\times Y)\setminus Z$ is zero. Therefore the only contributions come from the $g_K=g_{K,\lbar}$, and
by Fubini's Theorem we have
\begin{equation}
\label{E:repintproof}
\langle\alpha,\bbr{V}\rangle=\sum_K\int_V G_K(x)\,dx_K.
\end{equation}

It remains to prove that $G_K$ is smooth on the complement of a codimension $1$ smoothly stratified
closed subspace $T$ of $X$. The value of the integral $G_K(x)$, for $x\in X'$, can be obtained as
the difference of volumes of the slices at $x$ of the hypographs of $g_K^+$ and $g_K^-$.
These hypographs  are semi-algebraic sets, hence they are global subanalytic sets
(see \cite{CLR:nat} for a quick definition and a list of references on global subanalytic sets and functions).
By Theorem 1' of \cite{CLR:nat}  there exist (global) subanalytic functions
$A_1,\dots,A_s\colon X'\to\BR$ (not necessarily continuous) and a polynomial
$P\in\BR[a_1,\dots,a_s,u_1,\dots,u_s]$
such that
\[G_K=P(A_1,\dots,A_s,\log(A_1),\dots,\log(A_s)).\]
By \cite{Hir:tri} (see also \cite{Har:str}),
graphs of subanalytic functions are stratifiable, and therefore there exists a  finite stratification of $X'$ into  subanalytic smooth  submanifolds
 such that $G_K$ is smooth on each stratum. Let $X_0$ be the union of the maximal strata.
Then $G_K$ is smooth on $X_0$ and   its complement $T=X'\setminus X_0$
is a closed smoothly stratifiable subset of codimension $1$.
\end{proof}

We next establish  a continuity principle for PA forms: A PA form is determined by its values
 on the complement of any closed codimension $1$ smoothly stratifiable subset.
 We have an analogous statement for minimal forms
in \refP{mindense} but the proof there was easier because we assumed that the codimension $1$
subset was semi-algebraic.
\begin{lemma}
\label{L:contcodimPA}
Let $X$ be a semi-algebraic set, let $T\subset X$ be a closed codimension 1  smoothly stratifiable
 subset,  and
let $\alpha\in\ompa^k(X)$ be a PA form. If
$
\forall\gamma\in\Ch_k(X)$,
\[ \spt(\gamma)\cap T=\emptyset\,\Longrightarrow
\,\langle\alpha,\gamma\rangle=0,\]
then $\alpha=0$.
\end{lemma}
\begin{rem}This could be summarized by ``$\alpha|_{X\setminus T}=0\Longrightarrow\alpha=0$'' but
this does not quite make sense because $X\setminus T$ may be not semi-algebraic.
\end{rem}
\begin{proof}
Let $\gamma\in \Ch_k(X)$. We will prove that  $\langle\alpha,\gamma\rangle=0$.

Consider a cubification
of $X$ such that $\alpha$ is a trivial fiber integral over each closed cube and such that $\gamma$ is
a linear combination of cubes. By linearity it is enough to prove the statement when
$\gamma=\bbr{\overline\kappa}$, where $\kappa$ is a single cube. Let $\tau$ be an open
cube of maximal dimension such that $\kappa$ is a face of it.
There exists some semi-algebraic homeomorphism which lets us assume that $\overline\tau=[0,1]^n$
and that $\overline\kappa=[0,1]^k\times\{0\}^{n-k}$. The image of $T$ under this semi-algebraic homeomorphism is still of codimension $1$.

Since $\alpha$ is a trivial fiber integral over $[0,1]^n$, there exists a compact oriented manifold $F$
of dimension $l$ and a minimal form
\[\mu=\sum_{i=1}^r\lambda(f^i_0;f^i_1,\dots,f^i_{k+l})
\in\omin^{k+l}([0,1]^n\times F)\]
such that $\alpha|_{\overline\tau}=\fint_{\fib{F}}\mu$.
There exists $Z\subset [0,1]^n\times F$ of codimension $1$ such that the $f^i_j$ are smooth
on the complement of $Z$. As seen in the proof of \refL{repintPA} (the part following equation (\ref{E:slice})),
the set $T':=\{x\in[0,1]^n:\dim(Z\cap\{x\}\times F)=l\}$ is semi-algebraic
of codimension $1$ and the same is true for its closure. Since we can always enlarge $T$, we can assume
that $\overline {T'}\subset T$. This implies that if $V$ is a smooth submanifold in $X$
 then the $f^i_j$'s are smooth almost everywhere on $(V\setminus T)\times F$.

We now claim that there exists a sequence $\{\gamma_n\}$ in $C^k(\overline\tau)$ such that
$\gamma_n\toSA\gamma$ and $\calH^k(\spt(\gamma_n)\cap T)=0$.

First recall that $\gamma=\bbr{[0,1]^k\times\{0\}^{n-k}}$.
To prove the claim, consider the positive orthant of the $(n-k-1)$-dimensional  sphere
\[S^{n-k-1}_+=\{v\in [0,1]^{n-k}:\|v\|=1\}\]
and of the $(n-k)$-disk
\[D^{n-k}_+=\{v\in [0,1]^{n-k}:\|v\|\leq 1\}.\]
Since $T$ is of codimension $1$ in $[0,1]^n$, we have that
$\calH^n(T\cap ([0,1]^k\times D^{n-k}_+))=0$.
Passing to polar coordinates we have
\[
\calH^n(T\cap ([0,1]^k\times D^{n-k}_+))=C\cdot\int_0^1
\int_{S^{n-k-1}_+}\calH^k(T\cap [0,1]^k\times\{rv\})\,r^{n-k-1}dv\,dr
\]
where $C$ is a positive constant, $dv$ is the standard $(n-k-1)$-area measure on
$S^{n-k-1}_+$ and $dr$ is the Lebesgue measure on the interval $[0,1]$.
Hence the double integral is zero and since the integrand is non-negative,
there exists $v_0\in {S^{n-k-1}_+}$ such
that $\calH^k(T\cap [0,1]^k\times\{rv_0\})=0$  for almost every $r\in[0,1]$.
Therefore there exists a sequence $\{r_n\}$ in $[0,1]$, decreasing to $0$, such that
if we set
\[\gamma_n\, =\,\bbr{[0,1]^k\times\{r_n\cdot v_0\}}\,\in\,\Ch_k([0,1]^n)\]
then $\calH^k(T\cap\spt(\gamma_n))=0$ and $\gamma_n\toSA \gamma$. This proves the claim.

Clearly $\gamma_n\times\bbr{F}\toSA \gamma\times\bbr{F}$ and from \refP{contSAmin}
we have
\begin{equation}
\label{E:limPASA}
\langle\alpha,\gamma\rangle=
\langle\mu,\gamma\times\bbr{F}\rangle=
\lim_{n\to\infty}
\langle\mu,\gamma_n\times\bbr{F}\rangle=
\lim_{n\to\infty}
\langle\alpha,\gamma_n\rangle.
\end{equation}
Therefore it is now enough to prove that $\langle\alpha,\gamma\rangle=0$ when
$\calH^k(T\cap\spt(\gamma))=0$, and from now on we suppose that $\spt(\gamma)\cap T$
has zero $k$-volume. We can moreover suppose that $\gamma=\bbr{V}$, where $V$ is a
semi-algebraic oriented smooth submanifold of dimension $k$.

Let $m\geq1$. Since $T$ is closed,  $\spt(\gamma)\cap T$ is compact and it is covered by
 finitely many open balls $B(x_i^{(m)},1/m)$, centered at $x_i^{(m)}\in T$ and of radius $1/m$, for $1\leq i\leq N_m$.
Set $V_m=V\setminus \cup_{i=1}^{N_m}B[x_i^{(m)},2/m]$.  This is a
 semi-algebraic oriented smooth
submanifold. Since $\overline{V_m}\cap T=\emptyset$, the hypothesis of the lemma implies that
$\langle\alpha,\bbr{V_m}\rangle=0$. Using that the $f^i_j$ are smooth almost everywhere on $(V\setminus T)\times F$, that
 $\calH^k(T\cap V)=0$,  and that $\cup_{m\geq1}V_m= V\setminus T$, as well as the 
  Lebesgue Bounded Convergence Theorem, we get
\begin{eqnarray*}
\langle\alpha,\bbr{V}\rangle
&=&
\langle\mu,\bbr{V\times{F}}\rangle\\
&=&
\sum_{i=1}^r\int_{(V\setminus T)\times F}f^i_0\,df^i_1\dots df^i_{k+l}\\
&=&
\lim_{m\to\infty}\sum_{i=1}^r\int_{V_m\times F}f^i_0\,df^i_1\dots df^i_{k+l}\\
&=&
\lim_{m\to\infty}\langle\alpha,\bbr{V_m}\rangle \\
&=&
0.
\end{eqnarray*}
\end{proof}

We are now finally ready for
\begin{proof}[Proof of \refP{strKu-PA}]
Let $\gamma\in\Ch_k(X_1\times X_2)$. We will prove that $\langle\alpha,\gamma\rangle=0$.

Since $\spt(\gamma)$ is compact, we can restrict to the product of compact sets
$\proj_1(\spt(\gamma))\times\proj_2(\spt(\gamma))\subset X_1\times X_2$, so we can suppose that
$X_1$ and $X_2$ are compact. Take cubifications $\{\kappa^{1}_i\}$ and $\{\kappa^{2}_j\}$
of $X_1$ and $X_2$. By linearity it is enough to check that
$\alpha|_{\overline\kappa^{1}_i\times \overline\kappa^{2}_j}=0$
for all pairs of closed cubes in $X_1$ and $X_2$. Hence we can suppose that $X_1=[0,1]^{n_1}$
and  $X_2=[0,1]^{n_2}$.

Now take a triangulation of $X_1\times X_2$ such that $\alpha$ is a trivial fiber integral over the closure of each simplex
and such that $\gamma$ is a linear combination of simplices.
By linearity, without loss of generality we can assume that $\gamma=\bbr{\overline\sigma}$
for a single simplex $\sigma$. Let $\tau$ be a maximal dimensional open simplex such that $\sigma$ is a face
of $\overline\tau$. By the same argument as in the proof of \refP{densitySA}, we can construct
a sequence $\{\gamma_n\}\subset\Ch_k(\tau)$ such that $\gamma_n\toSA\gamma$.
By the same argument as in equation (\ref{E:limPASA}), it is enough to show
that $\langle\alpha,\gamma_n\rangle=0$. In other words, there is no loss of generality
in supposing that $\spt(\gamma)$ is contained in the simplex $\tau$ which is an
open set in $X_1\times X_2$
(because it is of maximal dimension).

Since $\spt(\gamma)$ is compact, it is covered by finitely many open sets of the form $O_1\times O_2$
contained in $\tau$ and such that $O_i$ is an open set in $X_i$ semi-algebraically homeomorphic to
$(0,1)^{n_i}$ for $i=1,2$. To conclude the proof, we will show that $\alpha|_{O_1\times O_2}=0$.

Set $X=O_1\times O_2\cong(0,1)^n$ with $n=n_1+n_2$ and notice  that $\alpha$ is a trivial fiber integral over $X$
since $X\subset\tau$. By \refL{repintPA}, there exists a closed codimension $1$ smoothly stratified $T\subset X$
and  smooth functions
$G_K\colon X\setminus T\to\BR$ for $K\subset\nbar$ of cardinality $k$,
giving the integral representation \refN{E:repintPA} for $\alpha|_{X\setminus T}$.
We prove that $G_K=0$. For concreteness, suppose that $K=\{1,\dots,k_1,n_1+1,\dots,n_1+k_2\}$
with $k_1+k_2=k$ and $0\leq k_i\leq n_i$. If $G_K$ was not zero  there would exist an $x\in X\setminus T$ such that, say,
$G_K(x)>0$. By continuity, there exists $\epsilon>0$ and non empty subsets $U_1\subset O_1$
and $U_2\subset O_2$ such that $G_K|_{U_1\times U_2}>\epsilon$.
Take  $V_1\subset U_1$ a non empty $k_1$-dimensional semi-algebraic submanifold parallel
 to the linear space $[0,1]^ {k_1}\times\{0\}^{n-k_1}$ (i.e.~a codimension $0$ manifold contained in some translate of $[0,1]^ {k_1}\times\{0\}^{n-k_1}$)
 and similarly a non-empty $k_2$-dimensional manifold  $V_2\subset U_2$ parallel to $\{0\}^{n_1}\times [0,1]^{k_2}\times \{0\}^{n_2-k_2}$.
By Fubini's Theorem,
\[\langle\alpha,\bbr{V_1}\times\bbr{V_2}\rangle=
\int_{V_1\times V_2}G_K\,dx_K\geq\epsilon\cdot\calH^k(V_1\times V_2)>0\]
which contradicts our main hypothesis. Therefore $G_K=0$.
This implies that $\alpha|_{(O_1\times O_2)\setminus T}=0$. By \refL{contcodimPA} we
deduce that $\alpha|_{O_1\times O_2}=0$ and the proposition is proved.
\end{proof}
\begin{prop}
\label{P:omPAcross}
Let $X_1$ and $X_2$ be  semi-algebraic sets.
There is a degree-preserving linear map
\[\times\colon\ompa^*(X_1)\otimes \ompa^*(X_2)\longrightarrow\ompa^*(X_1\times X_2)\]
characterized \zz{
                             in his math comment (3) the referee says
                             that he does not believe that the below
                             formula characterizes the cross product
                             on PA-forms. Actually in the statement
                             there was a typo that I corrected below:
                             $\alpha_i$ should belong to $\ompa$ and
                             not $\omin$. With this correction I do
                             believe that this formula characterizes
                             that cross product, because of  by
                             \refP{strKu-PA}. Actually the
                             \refP{strKu-PA} is not trivial (see its
                             proof) and is needed exactly for
                             \refP{omPAcross}, which then ensure
                               naturality of the product; otherwise to
                               prove naturality would be
                               complicated. I add a phrase at end of
                               proof to explain why it is cgharacterizing. Or do I oversee something
                               that that the referee means?
                               }
by the formula\fn{\pl 10 july 2008. I removed the Koszul sign in \refN{E:omPAcross} as for \refN{E:omincross}}
\begin{equation}
\label{E:omPAcross}
\langle\alpha_1\times\alpha_2,\gamma_1\times\gamma_2\rangle=
\langle\alpha_1,\gamma_1 \rangle\cdot \langle\alpha_2,\gamma_2 \rangle
\end{equation}
for $\alpha_i\in\ompa^*(X_i)$ and $\gamma_i\in\Ch_*(X_i)$, $i=1,2$.  The Leibniz formula
\[\delta(\alpha_1\times\alpha_2)=\delta(\alpha_1)\times\alpha_2+(-1)^{\deg(\alpha_1)}\alpha_1\times \delta(\alpha_2)\] also holds.
Further, let $T\colon X_2\times X_1\to X_1\times X_2$ be the twisting map given by $(x_2,x_1)\mapsto (x_1,x_2)$. Then
\[\alpha_2\times\alpha_1=(-1)^{\deg(\alpha_1)\deg(\alpha_2)}T^*(\alpha_1\times\alpha_2).\]
\end{prop}
\begin{proof}
We can write $\alpha_i=\fint_{\Phi_i}\mu_i$ for some
$\mu_i\in\omin^*(Y_i)$ and \zz{typo corrected} $\Phi_i\in\Chs(Y_i\to X_i)$,   $i=1,2$.
Set
\begin{equation}\label{E:defomPAcross}
\alpha_1\times\alpha_2=\fint_{\Phi_1\times\Phi_2}\mu_1\times\mu_2.
\end{equation}
It is straightfoward to check that this PA form has the desired characterizing property.
The Leibniz formula is a consequence of \refL{deltaompa}  and of the Leibniz formulas for minimal forms and continuous chains
(Propositions \ref{P:omincross} and \ref{P:ltimes}.) Similarly for the  twisting formula.  The right side of \refE{defomPAcross} is independent of
the choice of the representatives of $\alpha_i$ by \refP{strKu-PA},
which\zz{added to answer referee math comment (3)} %
 also implies that
\refE{omPAcross} characterizes this cross product.
\end{proof}

We define a multiplication on $\ompa^*(X)$ by
\[\alpha_1\cdot\alpha_2=\Delta^*(\alpha_1\times\alpha_2)\]
where $\Delta\colon X\to X\times X $ is the diagonal map.
It is immediate from the previous proposition that this multiplication satisfies the Leibniz formula and
is graded commutative.

In conclusion, we have
\begin{thm}
\label{T:omPAfunctor}
The above construction of PA forms defines a contravariant functor $\ompa^*\colon\SA\to\CDGA$.
\end{thm}


\section{Equivalence between $\Apl$ and $\ompa^*$}\label{S:Equivalence}


\label{S:equAplompa}
Recall that $\SA$ is the category of semi-algebraic sets. There is an obvious forgetful functor
$u\colon\SA\to\Top$. Consider the contravariant functors
\begin{align*}
\ompa^*&\colon\SA\longrightarrow\CDGA\\
\Apl(u(-);\BR)&\colon\SA\longrightarrow\CDGA
\end{align*}
where $\CDGA$ the category of commutative differential graded algebras and $\Apl$ is the Sullivan functor of piecewise polynomial forms as defined in \cite{BoGu:RHT} (see also \cite{FHT:RHT}).
Abusing notation we will simply write $X$ for $u(X)$.
The aim of this section is to prove our main theorem:
\begin{thm}
\label{T:main}
There is a zigzag of natural transformations
\[\ompa^*(X)\longrightarrow\cdots \longleftarrow\Apl(X;\BR)\]
which is a weak equivalence when $X$ is a compact semi-algebraic set.
\end{thm}
The explicit chain of quasi-isomorphisms will be given in section \ref{S:weOmApl}, \refE{mainzz} before \refP{critqi}.
The  proof of this theorem will mimic the proof of \cite[\S 11(d)]{FHT:RHT} which gives a weak equivalence between $\Apl(-;\BR)$ and
the DeRham functor $\osmooth^*$ of smooth differential forms on a smooth manifold.
From now on we will also drop the field of coefficients $\BR$ from the notation of $\Apl$.

\begin{rem} The theorem is likely to be true without the compactness hypothesis.  As it stands, the statement is sufficient  for the application we have in mind
(formality of the little cubes operad). 
\zz{changed a bit this remark to answer math comment (4) of the
  referee}
Our proof would work also in the noncompact case if Poincar\'e
\refL{Poincare} below could be proved for \emph{noncompact}
collapsible polyhedra.
 In fact, Kontsevich and Soibelman claim that the theorem is true for an even more general class of spaces they call \emph{PA spaces} \cite[Definition 20]{KoSo:def}.
\end{rem}

In order to prove the weak equivalence $\ompa^*\simeq\Apl$, we first need to establish some
homotopy properties of $\ompa^*$, namely
\begin{enumerate}
\item A Poincar\'e lemma stating that $\ompa^*(\Delta^k)$ is acyclic;
\item A sheaf property, i.e.~that $\ompa^*$ satisfies
 a Mayer-Vietoris  sequence; and
\item The fact that the simplicial sets $\ompa^k(\Delta^\bullet)$ are ``extendable''; this is used in establishing the Mayer-Vietoris property for the functor  $\APA$ which is built from the PA forms on the simplex $\Delta^k$ in
the same way as $\Apl$ is built from the polynomial forms on $\Delta^k$.
\end{enumerate}
We establish these results in the next three sections.


\subsection{Poincar\'e Lemma for $\ompa^*$}


Recall the notion of a collapsible polyhedron from \refS{PLtop}. The version of the ``Poincar\'e lemma'' that we
need is the following.
\begin{lemma}
\label{L:Poincare}
Let $X$ be a compact  semi-algebraic set that is semi-algebraically homeomorphic to
a collapsible polyhedron. Then
$\Ho^0(\ompa^*(X))\cong \BR$ and $\Ho^i(\ompa^*(X))=0$ for $i>0$.
\end{lemma}

As in the classical proof of the Poincar\'e lemma, the strategy is to convert a geometric homotopy into a cochain homotopy
and to deduce that the cochain complex is acyclic.  But this approach is not as straightforward in our case
because it is unclear whether the cochain homotopy operator maps PA
forms to PA forms \zz{the following added in view of referee's math
                                     comment (5). It does not really
                                     answer his request
                                   }
(see \refS{complexityPoincare}).

Set $I=[0,1]$ and let $h\colon X\times I\to X$
be a semi-algebraic homotopy. Define a cochain homotopy operator
\[\Theta_h\colon \Coch^{k}(X)\longrightarrow\Coch^{k-1}(X)\]
by, for $\lambda\in\Coch^{k}(X)$ and $\gamma\in\Coch_{k-1}(X)$,
\[\langle\Theta_h(\lambda),\gamma\rangle=\langle\lambda,h_*(\gamma\times\bbr{I})\rangle.\]
It is easy to check that
\begin{equation}
\label{E:hmtpy}
\delta\Theta_h\pm\Theta_h\delta=(h(-,1))^*-(h(-,0))^*.
\end{equation}

If $\lambda$ is a minimal form then $\Theta_h(\lambda)=\fint_{\fib{I}}h^*(\lambda)$ is a PA form,
but in general it is not a minimal form as can be seen from an example analogous to that of Remark
\ref{rmk:ominnotdeRham}.
We do not know whether $\Theta_h$ maps PA forms to PA forms but we will establish a weaker
result in that direction. To state it we need the following
\begin{definition}
Let $X$ be a semi-algebraic set and let $A\subset X$ be a closed semi-algebraic subset.
We say that a semi-algebraic  homotopy $h\colon X\times I\to X$ is \emph{supported on $A$} if
$h(A\times I)\subset A$  and $h(x,t)=x$ for $x\in\overline{X\setminus A}$ and $t\in I$.\\
\end{definition}

\begin{lemma}
\label{L:Hsimple}
Let $X$ be a semi-algebraic set, let $A\subset X$ be a closed semi-algebraic subset and
let $h\colon X\times I\to X$ be a semi-algebraic homotopy.
If $h$ is supported on $A$ and if $\alpha\in\ompa^*(X)$  is a trivial fiber integral over $A$,
then $\Theta_h(\alpha)\in\ompa^*(X)$.
\end{lemma}
\begin{proof}
The general idea of the proof is that on the one hand
 $\Theta_h(\alpha|_{\overline{X\setminus A}})=0$ because the homotopy
  is trivial over $\overline{X\setminus A}$, and that
on the other hand  $\alpha|_A=\int_{\fib{F}}\mu$ so that $\Theta_h(\alpha|_A)=\int_{\fib{F\times I}}(h\times\id_F)^*\mu$. So
$\Theta_h(\alpha)$ can be obtained by gluing two PA forms. 

In more detail, set $X_0=\overline{X\setminus A}$ and $A_0=A\cap X_0$, so that $X$ is the pushout of $X_0$ and
$A$ over $A_0$.
We have $\alpha=\int_\Phi\mu$ where
 $\Phi\in\Chs_l(Y\stackrel f\to X)$,   $\mu\in\omin^{k+l}(Y)$,
 and $A$ is the closure of a stratum of a trivialization of $\Phi$. Thus
there exists a map $g\colon A\times F\to Y$ such that $F$ is a compact oriented manifold of
dimension $l$, $fg=\proj_1$, and $g_*(\bbr{\{a\}\times F})=\Phi(a)$ for $a\in A$.
Set $Y_0=f^{-1}(X_0)$. The map $g$
 restricts to a map $g_0\colon A_0\times F\to Y_0$.
Consider the pushout diagram
\begin{equation}
\label{E:poYtilde}
\xymatrix{
A_0\times F
\ar[r]^{g_0}
\ar@{^(->}[d]
\ar@{}[dr]|{\textrm{pushout}}
&
Y_0\ar[d]^{\tilde j}\\
A\times F\ar[r]_{\tilde g}&\tilde Y
}
\end{equation}
and the map $q\colon\tilde Y\to Y$ induced by $g\colon A\times F\to Y$
and $Y_0\hookrightarrow Y$.

It is easy to construct a continuous chain
$\tilde\Phi\in\Chs_l(\tilde Y\stackrel{fq}\to X)$
such that $q_*(\tilde\Phi(x))=\Phi(x)$ for $x\in X$, and hence
$\fint_\Phi\mu=\fint_{\tilde\Phi}q^*(\mu)$.

Now define a homotopy $\tilde h\colon\tilde Y\times I\to \tilde Y$
 using maps $\tilde g$ and  $\tilde j$ from diagram \refN{E:poYtilde} by
\[
\left\{\begin{aligned}
\tilde h(\tilde g(a,f),t)&=\tilde g(h(a,t),f)&&\textrm{for $a\in A$, $f\in F$ and $t\in I$}\\
\tilde h(\tilde j(y),t)&=\tilde j(y)&&\textrm{for $y\in Y_0$ and $t\in I$}
\end{aligned}
\right.
\]
That $\tilde h$ is well-defined is a consequence of the support of $h$ on $A$.

We will show that
\begin{equation}
\label{E:thPA}
\Theta_h(\fint_\Phi\mu)=\fint_{\tilde\Phi\times\bbr{I}}\tilde h^*(q^*\mu)
\end{equation}
which implies the statement of the lemma.

When restricted to $X_0$, the right side of equation \eqref{E:thPA}
is
\[\fint_{(\tilde\Phi|_{X_0})\times\bbr{I}}(\tilde h|_{Y_0\times I}^*q^*)\mu=
\fint_{(\tilde\Phi|_{X_0})\times\bbr{I}}\proj_1^*q^*\mu\]
which is $0$ because the projection $\proj_1$ decreases the  dimension of the chains,
and for the left side we have $\Theta_{h}|_{X_0\times I}=0$
because the homotopy is constant over $X_0$.

On the other hand, for $\gamma\in \Ch_{k-1}(A)$, we have
\begin{eqnarray*}
\langle\Theta_h(\fint_\Phi\mu),\gamma\rangle&=&
\langle\fint_{\Phi|_A}\left(\mu|_{f^{-1}(A)}\right)\,,\,h_*(\gamma\times\bbr{I})\rangle\\
&=&
\langle\fint_{\hat F}g^*(\mu)\,,\,h_*(\gamma\times\bbr{I})\rangle\\
&=&
\langle g^*\mu\,,\,h_*(\gamma\times\bbr{I})\times\bbr{F}\rangle\\
&=&
\langle g^*\mu\,,\,(h\times\id_F)_*(\gamma\times\bbr{I\times F})\rangle\\
&=&
\langle (h\times\id_F)^*(g^*\mu)\,,\,\gamma\times\bbr{I\times F}\rangle\\
&=&
\langle\fint_{\tilde\Phi\times\bbr{I}}\tilde h^*(q^*\mu)\,,\,\gamma\rangle
\end{eqnarray*}
and this proves equation \eqref{E:thPA} when restricted to $A$.

Since $\Ch_*(X)=\Ch_*(A)+\Ch_*(X_0)$, equation \eqref{E:thPA} holds on $\Ch_*(X)$ and
 the lemma is proved.
\end{proof}

We come  to the proof of the Poincar\'e Lemma.
\begin{proof}[Proof of \refL{Poincare}]
Since $X$ is collapsible, it is semi-algebraically path-connected and the only degree $0$ $\delta$-cocycles are the
constant functions. Therefore $\ompa^0(X)\cap\ker\delta\cong\BR$.

Let $\alpha\in\ompa^k(X)\cap\ker\delta$ with $k\geq1$. We will show that
$\alpha\in\delta(\ompa^{k-1}(X))$.
We have $\alpha=\fint_\Phi\mu$ for some $\Phi\in\Chs_l(Y\to X)$ and $\mu\in\omin^{k+l}(Y)$.
Take a triangulation of $X$ trivializing $\Phi$. By the semi-algebraic Hauptvermutung (\refT{Hauptvermutung} and \refC{collapsibility}),
the polyhedron associated to that triangulation is also collapsible, and it thus admits a subdivided  triangulation that simplicially collapses to some vertex $*$.
This means that the triangulation consists of closed simplices $\tau_1, \cdots,\tau_N$
such that if we set $T_p=\cup_{i=p}^N\tau_i$ for $1\leq p\leq N$ and $T_{N+1}=*$,
then there exist semi-algebraic homotopies
$h_{p}\colon T_p\times I\to T_p$ supported on $\tau_p$
such that $h_{p}(-,0)=\id_{T_p}$ and $h_{p}(T_p\times\{1\})\subset T_{p+1}$.
Since $\alpha|_{T_p}$ is a trivial fiber integral over $\tau_p$,
\refL{Hsimple} implies that $\Theta_{h_{p}}(\alpha|_{T_p})\in\ompa^*(T_p)$.
Define $g_p\colon T_p\to T_{p+1}$ by $g_p(x)=h_{p}(x,1)$ and equation \eqref{E:hmtpy} then implies that
$\delta\Theta_{h_{p}}(\alpha|_{T_p})=g^*_p(\alpha|_{T_{p+1}})-\alpha|_{T_p}$
for $1\leq p\leq N$.

Set
\[
\beta=\Theta_{h_{1}}(\alpha|_{T_1})+
g_1^*(\Theta_{h_{2}}(\alpha|_{T_2}))+
g_1^*g_2^*(\Theta_{h_{3}}(\alpha|_{T_3}))+
\cdots+
g_1^*g_2^*\cdots g_{N-1}^*(\Theta_{h_{N}}(\alpha|_{T_{N}})).
\]
This is an elements of $\ompa^{k-1}(X)$.
Most of the terms of $\delta\beta$ cancel in pairs and we are left with
\[\delta\beta=-\alpha+g_1^*g_2^*\cdots g_{N-1}^*g_N^*(\alpha|_{T_{N+1}}).\]
The second term in this sum is $0$ because $T_{N+1}=*$ and $\alpha$ is of positive degree.
This proves that $\alpha=\delta(-\beta)$ is a coboundary in $\ompa^*(X)$.
\end{proof}


\subsection{Sheaf propery of $\ompa^*$}


\begin{definition}
\label{D:exc}
An \emph{excisive  semi-algebraic pair} is a pair $\{X_1,X_2\}$ of semi-algebraic sets in
the same $\BR^m$ such that there exist semi-algebraic functions $\rho_i\colon X_1\cup X_2\to[0,1]$
with $X_i\supset\rho_i^{-1}((0,1])$ for $i=1,2$ and $\rho_1+\rho_2=1$.
\end{definition}
\begin{lemma}
\label{L:MVompa}
Let $\{X_1,X_2\}$ be an excisive semi-algebraic pair. There exists an exact sequence
\begin{equation}
\label{E:MVompa}
0\longrightarrow
\ompa^k(X_1\cup X_2)\stackrel{r}\longrightarrow
\ompa^k(X_1)\oplus\ompa^k(X_2)\stackrel{\Delta}\longrightarrow
\ompa^k(X_1\cap X_2)\longrightarrow
0
\end{equation}
with $r(\alpha)=(\alpha|_{X_1},\alpha|_{X_2})$ and $\Delta(\alpha_1,\alpha_2)=
\alpha_1|_{X_1\cap X_2}-\alpha_2|_{X_1\cap X_2}$.
\end{lemma}
\begin{proof}
Set $X=X_1\cup X_2$.
It is clear that $\Delta r=0$.
It is also not difficult to construct, from the given partition of unity,  another one that we will also denote by  $\{\rho_1,\rho_2\}$
with smaller support and such that there exist semi-algebraic functions $\sigma_i\colon X\to[0,1]$
with $\rho_i(x)>0\Longrightarrow\sigma_i(x)=1$ and $\overline{\sigma_i^{-1}((0,1])}\subset X_i$.
This can be done by taking a semialgebraic function $\phi\colon[0,1]\to[0,1]$ such that $\phi([0,1/4])=\{0\}$  and
$\phi(1-t)=1-\phi(t)$, and composing the given
$\rho_i$ with $\phi$ to get the new partition of unity. This leaves enough room in $X_i\cap\rho_i^{-1}(0)$
to build the desired functions $\sigma_i$. For the rest of the proof we will assume that we are dealing with this better partition of unity.

Notice that $X$ is the union of the closed sets $\rho_i^{-1}([1/4,1])$, $i=1,2$.
Since $\rho_i^{-1}([1/4,1])\subset X_i$ we deduce that
$\Ch_k(X)=\Ch_k(X_1)+\Ch_k(X_2)$.  Hence $r$ is injective.

To prove the surjectivity of $\Delta$, let $\alpha\in\ompa^k(X_1\cap X_2)$
and let $j_i\colon X_1\cap X_2\hookrightarrow X_i$ be the inclusion.
Then $\alpha=\Delta(j_1^*(\rho_2)\alpha,-j_2^*(\rho_1)\alpha))$.

It remains to prove that $\ker\Delta\subset\im r$.
Let $\alpha_i\in\ompa^k(X_i)$ for $i=1,2$ be such that $\Delta(\alpha_1,\alpha_2)=0$.
Let $\Phi_i\colon\Chs_l(Y_i\stackrel{p_i}\to X_i)$ and $\mu_i\in\omin^{k+l}(Y_i)$ be such that
$\alpha_i=\fint_{\Phi_i}\mu_i$ for $i=1,2$. There is no loss of generality in assuming
that $l$ is positive and independent of $i$ as seen in the proof of \refP{ompasbgrp}. We can also assume that
$Y_i\subset\BR^N$.

Consider the following semi-algebraic sets and maps
\begin{align*}
&\tilde Y_i=\{(p_i(y),\sigma_i(p_i(y)) y):y\in Y_i\}\subset X\times \BR^N,\\
&\tilde p_i\colon \tilde Y_i\to X\ \ \text{given by}\ \ (x,v)\mapsto x,\\
&q_i\colon Y_i\to \tilde Y_i\ \ \text{given by}\ \ y\mapsto (p_i(y),\sigma_i(p_i(y)) y).
\end{align*}
Notice that the restriction of $q_i$ to $p_i^{-1}(\sigma_i^{-1}((0,1])$ is a
homeomorphism onto its image and that $q_i$ maps the entire fiber of $p_i$ over a point
$x\in\sigma_i^{-1}(0)$ to a single point $\{(x,0)\}$.

Define $\tilde\Phi_i$ by \zz{typo corrected in equation below}
\[
\tilde\Phi_i(x)=\left\{
\begin{aligned}
&q_{i*}(\Phi_i(x))&&\quad\textrm{if $x\in X_i$}\\
&0&&\quad\textrm{if $\sigma_i(x)=0$}.
\end{aligned}\right.
\]
This defines a continuous chain
$\tilde\Phi_i\in \Chs_l(\tilde Y_i\to X)$.
It is also easy to build minimal forms $\tilde\mu_i\in \omin^{k+l}(\tilde Y_i)$
such that $q^*_i(\tilde\mu_i)=p^*_i(\rho_i)\mu_i$ in $\omin^{k+l}(Y_i)$.

Set\zz{typo corrected, see typo referee (26)}%
  $\tilde Y=\tilde Y_1\amalg\tilde Y_2$,
 $\tilde p=(\tilde p_1,\tilde p_2)\colon \tilde Y\to X$,
 $\tilde\Phi=\tilde\Phi_1+\tilde\Phi_2\in\Chs_l(\tilde Y\to X)$,
 and $\tilde\mu=\tilde\mu_1+\tilde\mu_2\in\omin^{k+l}(\tilde Y)$.
 This determines a PA form $\alpha=\fint_{\tilde\Phi}\tilde\mu\in\ompa^k(X)$
 and one can easily check that $\alpha|_{X_i}=\alpha_i$,
 since\zz{added to answer refree typo (23)} the
 left side of this equation can be identified (with abuse of notation)
 with $(\alpha_1\cdot\rho_1|_{X_i}+\alpha_2\cdot\rho_2|_{X_i})$.
\end{proof}


\subsection{Extendability of the simplicial set $\ompa(\Delta^\bullet)$}


\begin{definition}[\cite{FHT:RHT}, page 118]
\label{D:ext}
A simplicial set $A_\bullet$  is \emph{extendable}
if, for each $n\geq 0$ and for each subset $I\subset\{0,\cdots,n\}$,
given a collection $\beta_i\in A_{n-1}$ for $i\in I$
such that $\partial_i\beta_j=\partial_{j-1}\beta_i$ for each $i<j$ in $I$,
there exists $\beta\in A_n$ such that $\beta_i=\partial_i(\beta)$
for all $i\in I$.
\end{definition}

The goal in this section is to prove the extendability of the simplicial vector spaces
$\ompa^k(\Delta^\bullet)$.

Fix an integer $n$. Consider the ``corner''
$$
X=\{x=(x_1,\dots,x_n)\in\BR^n:x_i\geq0\textrm{ and }\sum_{i=1}^n x_i\leq 1\}
$$
whose codimension $1$ faces containing the origin are given by
$$
\partial_iX=\{x\in X:x_i=0\}\quad\textrm{ for } 1\leq i\leq n.
$$
Notice that $\partial_iX$ consists of all the faces of $X$ except the ``diagonal face''
$\{\sum x_i=1\}$.
The space $X$ is homeomorphic to the simplex $\Delta^n$. We next show that a family of compatible PA forms
given on the non-diagonal faces can be extended into a PA form on $X$.
\begin{lemma}
\label{L:extX}
Let $\alpha_i\in\ompa(\partial_i X)$ for $1\leq i\leq n$. If
$\alpha_i|_{\partial_i X\cap \partial_j X}=\alpha_j|_{\partial_i X\cap \partial_j X}$
for all $1\leq i,j\leq n$ then there exists $\alpha\in\ompa(X)$ such that
$\alpha|_{\partial_i X}=\alpha_i$ for all  $1 \leq i\leq n$.
\end{lemma}
\begin{proof}
For $I\subset\nbar$, define $\partial_IX=\cap_{i\in I}\partial_i X$.  In particular $\partial_\emptyset X=X$, and define
$\proj_I\colon X\to\partial_IX$
by $\proj_I(x_1,...,x_n)=(x'_1,...,x'_n)$ with $x'_i=0$ if $i\in I$
and $x'_i=x_i$ otherwise.
For $J\subset I$ consider also the inclusion maps $j_J^I\colon\partial_IX\to\partial_JX$.

Define
\[\alpha=\sum_{\emptyset\not=I\subset\nbar}
(-1)^{|I|-1}\proj^*_Ij^{I*}_{\{i\}}(\alpha_i)\,\,\in\ompa^*(X)
\]
where $i$ is an arbitrary element of $I$. Using the hypothesis,
one can easily check that $\alpha$ is independent of the choices of $i\in I$.
Moreover
\begin{eqnarray*}
\alpha|_{\partial_iX}&=&j_\emptyset^{\{i\}*}\alpha\\
&=&j_\emptyset^{\{i\}*}\proj^{*}_{\{i\}}j_{\{i\}}^{\{i\}*}(\alpha_i)+
\sum_{\emptyset\not=J\subset(\nbar\setminus\{i\})}(-1)^{|J|-1}
\lbrace
j_\emptyset^{\{i\}*}\proj^{*}_{J}j_{\{j\}}^{J*}\alpha_j
\,-\,
j_\emptyset^{\{i\}*}\proj^{*}_{J\cup\{i\}}j_{\{j\}}^{J\cup\{i\}*}\alpha_j
\rbrace
\end{eqnarray*}
where 
$j$ is an arbitrary element in $J$.
Since the  diagram 
\[\xymatrix{
\partial_iX\ar[r]^{j_\emptyset^{\{i\}}}\ar[rd]_{j_\emptyset^{\{i\}}}&
X\ar[r]^{\proj_{J}}&
\partial_JX\ar[r]^{j_{\{i\}}^{J}}&
\partial_j X\\
&
X\ar[r]_-{\proj_{J\cup\{i\}}}&
\partial_{J\cup\{i\}}X\ar[ru]_{j_{\{i\}}^{J\cup\{i\}}}
}
\]
is commutative, it follows that all the terms in the brackets in the above sum vanish.
Therefore $\alpha|_{\partial_iX}=
j_\emptyset^{\{i\}*}\proj^{*}_{\{i\}}j_{\{i\}}^{\{i\}*}(\alpha_i)=\alpha_i$.
\end{proof}

\begin{lemma}
\label{L:extDelta}
The simplicial vector space $\ompa^k(\Delta^\bullet)$ is extendable.
\end{lemma}
\begin{proof}
By an elementary induction, it is enough to check the extendability condition of \refD{ext}
for $I=\{0,\cdots,n\}$.

Consider the standard $n$-simplex 
$\Delta^n=\{t=(t_0, \dots, t_n)\in \BR^{n+1}:t_p\geq0,\sum_{p=0}^n t_p=1\}$.
Its faces are   $\partial_p\Delta^n=\{t\in\Delta^n:t_p=0\}$ for $0\leq p\leq n$.
Suppose given $\beta_p\in\ompa^k(\partial_p\Delta^n)$ for $0\leq p\leq n$
such that
$\beta_p\vert_{\partial_p\Delta^n\cap\partial_q\Delta^n}
\,=\,\beta_q\vert_{\partial_p\Delta^n\cap\partial_q\Delta^n}$
for $0\leq p,q\leq n$.
We want to build
$\beta\in\ompa^k(\Delta^n)$ such that $\beta_p=\beta\vert_{\partial_p\Delta^n}$ for $0\leq p\leq n$.

For $0\leq p\leq n$ we have a semi-algebraic homeomorphism
\begin{align*}
\pi_p\colon\Delta^n & \stackrel{\cong}{\longrightarrow} X  \\
(t_0, \dots, t_n) & \longmapsto(x_1=t_0,\dots,x_p=t_{p-1},x_{p+1}=t_{p+1}\dots,x_n=t_n)
\end{align*}
which corresponds to the orthogonal projection of the simplex onto the hyperplane $t_p=0$.
This homeomorphism sends the faces of the simplex to the faces of $X$ with
\[
\pi_p(\partial_i\Delta^n)=\left\{
\begin{aligned}
\partial_{i+1}X&\quad\textrm{if }i<p\\
\partial_{i}X&\quad\textrm{if }i>p
\end{aligned}
\right.
\]
and $\pi_p(\partial_p\Delta^n)$ is the ``diagonal face'' of $X$.
We define PA forms $\alpha_i^p\in\ompa^l(\partial_iX)$ by
\[
\alpha_i^p=
\left\{
\begin{aligned}
(\pi^{-1}_p)^*(\beta_{i-1})&\quad\textrm{if }1\leq i<p\\
(\pi^{-1}_p)^*(\beta_{i})&\quad\textrm{if }p+1\leq i\leq n.
\end{aligned}
\right.
\]
By \refL{extX}, there exists $\alpha^p\in\ompa^l(X)$
such that $\alpha^p\vert\partial_iX=\alpha^p_i$.
We deduce that $\pi_p^*(\alpha^p)\vert\partial_i\Delta^n=\beta_i$ if $i\not=p$.
Define a map $\tau_p\colon\Delta^n\to\BR$ by $\tau_p(t_0,\cdots,t_n)=t_p$
and set
\[\beta=\sum_{p=0}^n\tau_p\cdot\pi^*_p(\alpha^p).\]
Since $\sum\tau_p=1$ and $\tau_p\vert\partial_p\Delta^n=0$
it follows that $\beta\vert\partial_p\Delta^n=\beta_p$.
\end{proof}


\subsection{The weak equivalence $\ompa^*\simeq\Apl$}
\label{S:weOmApl}


We are ready to prove our main theorem by following the scheme
of the proof of $\osmooth\simeq\Apl$ from \cite[\S 11(d)]{FHT:RHT}.

The geometric simplices
define a semi-algebraic cosimplicial space $\{\Delta^p\}_{p\geq 0}$.
Therefore we can consider the simplicial CDGA
$$\APA_\bullet=\{\ompa^*(\Delta^p)\}_{p\geq 0}.$$
This should be compared with the simplicial CDGA $(\Apl)_\bullet$ from \cite[\S 10(c)]{FHT:RHT} where
$\Apl_p$ consists of polynomial forms on the simplicial set $\Delta^p_\bullet$.
More precisely, this is the CDGA
\[\Apl_p=\bigwedge(t_0,\cdots,t_p,dt_0,\cdots,dt_p)
\left/
\left(1-\sum
\nolimits
_{i=0}^p t_i,\,\sum
\nolimits
_{i=0}^p dt_i\right)\right..\]

Since a polynomial is a semi-algebraic function, there is also an obvious inclusion
\[\Apl_\bullet\hookrightarrow\APA_\bullet\]
sending the polynomial $k$-form $\sum_{I=(i_1,\cdots,i_k)\subset[n]=\{0,\cdots,n\}}P_I(t_0,\cdots,t_n)dt_{i_1}\wedge\cdots\wedge dt_{i_k}$
to the minimal form  $\sum_{I}\lambda(P_I(t_0,\cdots,t_n);t_{i_1},\cdots, t_{i_k})$ for
$P_I\in\BR[t_0,\dots,t_n]$.

Recall that, as explained in \cite[\S 10 (b)]{FHT:RHT},
 to any such simplicial CDGA $A_\bullet$ one associates a contravariant functor
\begin{align*}
A(-)\colon\sSet & \longrightarrow\CDGA \\
S_\bullet & \longmapsto A(K_\bullet)=\hom_{\sSet}(K_\bullet,A_\bullet).
\end{align*}
An element of $A^k(K_\bullet)$ can be regarded as a family
$\{\Phi_\sigma\}_{\sigma\in K_\bullet}$
with $\Phi_\sigma \in A^k_{\deg(\sigma)}$, compatible with boundaries and degeneracies.
In particular, we have the contravariant functors $\Apl$ and $\APA$ from simplicial sets to CDGAs.

For a semi-algebraic set $X$, define the simplicial set of semi-algebraic singular simplices in $X$,
$$\Spa_\bullet(X)=\{\Spa_p(X)\}_{p\geq0}$$
where
$$\Spa_p(X)=\{\sigma\colon\Delta^p\to X\ |\ \sigma\textrm{ is a semi-algebraic map}\}.$$
We also have the classical  simplicial set of singular simplices in $X$,
$\Sing_\bullet(X)$,  as in
\cite[\S 10(a) and \S4 (a)]{FHT:RHT} (where it is denoted by $S_*(X)$.)
Since a semi-algebraic map $\sigma\colon\Delta^p\to X$ is continuous, we have a natural inclusion
\[\Spa_\bullet(X)\hookrightarrow \Sing_\bullet(X).\]

Our equivalence between $\ompa^*$ and $\Apl$ goes through the following natural zig-zag
\begin{equation}
\label{E:mainzz}
\xymatrix{
\ompa^*(X)\ar[r]^-{\alpha_X}&
\APA(\Spa_\bullet(X))&
\ar[l]_-{\beta_X}\Apl(\Spa_\bullet(X))&
\ar[l]_-{\gamma_X}\Apl(\Sing_\bullet(X))=:\Apl(X)
}
\end{equation}
where, for $\omega\in\ompa(X)$,
$\alpha_X(\omega)=\{\sigma^*(\omega)\}_{\{\sigma\in \Spa_\bullet(X)\}}$,
$\beta_X$ is induced by the inclusion $\Apl_\bullet\hookrightarrow\APA_\bullet$, and
$\gamma_X$ by the inclusion
$\Spa_\bullet(X)\hookrightarrow \Sing_\bullet(X)$.

The main idea for proving that these natural maps are weak equivalences is to use the following
``Eilenberg-Steenrod criterion''.
\begin{prop}
\label{P:critqi}
Let $A,B\colon\SA\to\Chains^*(\BR)$ be two contravariant functors with values in cochain complexes
and let $\theta\colon A\to B$ be a natural transformation between them.
Suppose that
\begin{enumerate}
\item If $X$ is empty or if  $X$ is semi-algebraically homeomorphic to a collapsible
polyhedron, then $\theta_X$ is a quasi-isomorphism;
\item If $\{X_1,X_2\}$ is an excisive semi-algebraic pair and if $\theta_{X_1}$,
$\theta_{X_2}$, and $\theta_{X_1\cap X_2}$ are quasi-isomorphisms, then  so is
$\theta_{X_1\cup X_2}$.
\end{enumerate}
Then $\theta_X$ is a quasi-isomorphism for every \emph{compact} semi-algebraic set $X$.
\end{prop}
\begin{proof}
As $X$ is  a compact semi-algebraic set, after triangulating we can assume that it is
a finite simplicial complex (by abuse of notation we will not distinguish
 a simplicial complex from its geometric realization).
Consider its second barycentric subdivision  $X''$. For each closed simplex $\sigma$ of $X$,
we consider its second derived neighborhood, $N(\sigma)$, as defined in \cite[page 50]{Hud:PLT}.
In more detail,  $N(\sigma)$  is constructed as follows: If $v$ is a vertex of $X''$ then its \emph{star} in $X''$ is the
smallest closed simplicial subcomplex of $X''$ that is a topological neighborhood of  $v$ in $X''$. In other words, the
 star of $v$  is obtained as the union of the closure of the simplices of $X''$ whose $v$ is a vertex.
Then $N(\sigma)$ is defined to be  the union of the stars of all the vertices of the
second subdivision of $\sigma$, so it is also the smallest closed subcomplex of $X''$ which is a neighborhood of $\sigma$ in $X$.
It is well known
\cite[Lemma 2.10, page 55]{Hud:PLT}
that $N(\sigma)$ collapses onto $\sigma$ and, since $\sigma$ collapses to a point, the second
derived neighborhood  $N(\sigma)$ is collapsible.

It is easy to see
 that if $\sigma$ and $\tau$ are two closed simplices of $X$ then
\[N(\sigma)\cap N(\tau)=N(\sigma\cap \tau),\]
with the convention that $N(\emptyset)=\emptyset$.

Let $\Sigma$ be a set of simplices of $X$ and set $N(\Sigma)=\cup_{\sigma\in \Sigma}N(\sigma)$.
 We prove by induction on the cardinality of $\Sigma$ that
$\theta_{N(\Sigma)}$ is a quasi-isomorphism.
If $\Sigma$ is empty, then so is $N(\Sigma)$ and by hypothesis $\theta_\emptyset$ is a quasi-isomorphism.
Suppose that  $\theta_{N(\Sigma')}$ is a quasi-isomorphism when $|\Sigma'|<k$ and let $\Sigma$
be of cardinality $k$. We write $\Sigma= \Sigma'\cup\{\tau\}$ with $|\Sigma'|<k$.
Then
\[N(\Sigma')\cap N(\tau)=\cup_{\sigma\in\Sigma'}N(\sigma\cap\tau)=N(\Sigma'')\]
for some $\Sigma''$ of cardinality $<k$. So by induction, $\theta_{N(\Sigma')\cap N(\tau)}$
is a quasi-isomorphism as well as  $\theta_{N(\Sigma')}$.
Also $\theta_{N(\tau)}$ is a quasi-isomorphism because $N(\tau)$ collapses to a point.
Finally it is easy to see that the pair $\{N(\Sigma'),N(\tau)\}$ is excisive.
By the ``Mayer-Vietoris condition'' in the hypotheses we deduce that $\theta_{N(\Sigma)}$
is quasi-isomorphism.

Taking for $\Sigma$ the set of all simplices of $X$, we have $N(\Sigma)=X$, and so $\theta_X$ is a quasi-isomorphism.
\end{proof}

For a simplicial set $S_\bullet$, we consider the cosimplicial vector space
$\hom(S_\bullet,\BR)$ and its normalized cochain complex which we denote by
 $\Norm^*(S_\bullet)$ (in \cite[\S 10 (d)] {FHT:RHT}, this is denoted by $\Coch^*(S_\bullet)$,
but we have already reserved this notation for semi-algebraic cochains).

\begin{lemma}
\label{L:prepqi}
Let $X$ be a semi-algebraic set and let $\{X_1,X_2\}$ be a pair of semi-algebraic subsets of $X$.
\begin{enumerate}
\item [(i)] There is  a short exact sequence
\[
0\to
\APA(\Spa_\bullet(X_1)\cup\Spa_\bullet(X_2))\stackrel{r}\to
\APA(\Spa_\bullet(X_1))\oplus\APA(\Spa_\bullet(X_2))\stackrel{\Delta}\to
\APA(\Spa_\bullet(X_1\cap X_2))\to 0.
\]
\item[(ii)] For any simplicial set $S_\bullet$ there is a zig-zag of natural quasi-isomorphisms  of cochain complexes
 \[\xymatrix{
 \APA(S_\bullet)\ar[r]^-{\simeq}&\cdots&
\ar[l]_-\simeq
\Norm^*(S_\bullet)}\]
\item[(iii)] If $\{X_1,X_2\}$ is semi-algebraically excisive, then
there is a natural quasi-isomorphism  of cochain complexes
\[
\APA(\Spa_\bullet(X_1\cup X_2))\stackrel{\simeq}{\longrightarrow}
\APA(\Spa_\bullet(X_1)\cup\Spa_\bullet(X_2)).
\]
\end{enumerate}
All of these properties remain true if $\APA$ is replaced by $\Apl$
and/or $\Spa_\bullet$ by $\Sing_\bullet$.
\end{lemma}
\begin{proof}
(i) The map $r$ is defined by
\[r\left(\{\Phi_\sigma\}_{\sigma\in \Spa_\bullet(X_1)\cup\Spa_\bullet(X_2)}\right)\,\,=\,\,
\left(\{\Phi_\sigma\}_{\sigma\in \Spa_\bullet(X_1)},
\{\Phi_\sigma\}_{\sigma\in \Spa_\bullet(X_2)}\right)\]
and $\Delta$ by
\[\Delta\left(\{\Phi_\sigma\}_{\sigma\in \Spa_\bullet(X_1)},
\{\Psi_\sigma\}_{\sigma\in \Spa_\bullet(X_2)}
\right)\,\,=\,\,
\{\Phi_\sigma-\Psi_\sigma\}_{\sigma\in \Spa_\bullet(X_1\cap X_2)}.
\]
The surjectivity of $\Delta$ is a consequence of the extendability of $\APA$ and $\Apl$
(\refL{extDelta} and \cite[Lemma 10.7 (iii)]{FHT:RHT}) and of
\cite[~Proposition 10.4 (ii)]{FHT:RHT}.
The other parts of exactness are clear.

(ii)  Theorem 10.9 of \cite{FHT:RHT} connects $\Apl(S_\bullet)$ and $\Norm^*(S_\bullet)$
(denoted there by $C_{PL}(S_\bullet)$)
by  a zigzag of quasi-isomorphisms. Since $\APA_\bullet$ is extendable, exactly the same proof
gives a natural weak equivalence $\APA(S_\bullet)\simeq\Norm^*(S_\bullet)$.

(iii) When the pair $\{X_1,X_2\}$  is excisive,
 the classical barycentric argument \cite[Proposition 2.21]{Hat:AT} implies that there is a quasi-isomorphism
\[
\Norm^*(\Spa_\bullet(X_1)\cup\Spa_\bullet(X_2))=
\Norm^*(\Spa_\bullet(X_1))+\Norm^*(\Spa_\bullet(X_2))\stackrel{\simeq}{\longrightarrow}
\Norm^*(\Spa_\bullet(X_1\cup X_2))
\]
and similarly for $\Sing_\bullet$.
The map in (iii) is clear and the fact that it is a quasi-isomorphism is a consequence of (ii).
\end{proof}

We finally arrive at the
\begin{proof}[Proof of \refT{main}]
We prove that $\alpha_X$, $\beta_X$, and $\gamma_X$ of equation \refN{E:mainzz} are quasi-isomorphisms
for a compact semi-algebraic set $X$.

Since both the simplicial CDGAs $\Apl_\bullet$ and $\APA_\bullet$ satisfy the Poincar\'e lemma
and are extendable (for $\APA$ this is \refL{Poincare}  and \refL{extDelta}, and
for $\Apl$ it is \cite[Lemma 10.7]{FHT:RHT}),
Proposition 10.5 of  \cite{FHT:RHT} implies that $\beta_X$ is a quasi-isomorphism.

By \refL{prepqi} (ii), to prove that $\gamma_X$ is a quasi-isomorphism, it suffices to prove that the
induced map
\[ \Norm^*(\Sing_\bullet(X))\longrightarrow \Norm^*(\Spa_\bullet(X))\]
is a quasi-isomorphism. This natural transformation satisfies
the hypotheses of \refP{critqi}. Indeed, if $X$ is collapsible then we have a
semi-algebraic homotopy between the
identity map on $X$ and the constant map, which induces an algebraic nullhomotopy on the cochain
complexes $\Norm^*(\Spa_\bullet(X))$ and $\Norm^*(\Sing_\bullet(X))$, and the Mayer-Vietoris condition comes from an analogous argument
as in the proof of \refL{prepqi} (iii). Therefore $\gamma_X$ is a weak equivalence for compact semi-algebraic sets.

For $X$ collapsible, by \refL{Poincare} we have that
$\tilde H^*(\ompa^*(X))=0$ and
the same is true for the reduced homology of $\APA(\Spa(X))$ by \refL{prepqi} (ii).  We deduce easily that
when $X$ is  collapsible or empty then $\alpha_X$ is a quasi-isomorphism.
Suppose that $\{X_1,X_2\}$ is excisive.
We have the short exact sequences of  \refL{MVompa} and \refL{prepqi} (i) that yield 
long exact sequences in  homology. Using  \refL{prepqi} (iii), the five lemma
implies that $\alpha$ satisfies the second hypothesis of \refP{critqi}.
Therefore $\alpha_X$ is a quasi-isomorphism for all compact semi-algebraic sets $X$.
\end{proof}


\section{Monoidal equivalences}\label{S:MonEquiv}


One\fn{\pl 11 july 2008: this section has been extended}
 of the motivations for establishing the results in this paper is Kontsevich's  proof of the formality
of the little cubes operad in \cite[\S3]{Kon:OMDQ}. In that proof, functors such as $\Ch_*(-)$ and $\ompa^*(-)$ 
need to be extended from objects to operads, but this can only be done if 
these functors are symmetric monoidal (a basic reference for symmetric monoidal categories is \cite{Hovey}). 
 In this section we thus study the monoidal properties of various
functors such as $\Ch_*(-)$ and $\ompa^*(-)$.

The categories of semi-algebraic sets and of topological spaces, equipped  with
the cartesian product and the one-point space are symmetric monoidal categories. The forgetful functor
$u\colon\SA\to\Top$ is  a strong symmetric
monoidal functor, where ``strong'' means
that the natural map $u(X)\times u(Y)\iso u(X\times Y)$ is an isomorphism \cite[\S 2.2.1]{GNPR:mon}.

The cross product constructed in \refP{chaincross},
\[\times\colon\Ch_*(X)\otimes\Ch_*(Y)\to\Ch_*(X\times Y)\]
makes the functor $\Ch_*$ of semi-algebraic chains from \refD{C}
symmetric monoidal.
The classical functor of singular chains
\[\Sing_*\colon\Top\longrightarrow\Chains_+(\BZ)\]
defined as the normalized chain complex of the singular simplicial set $\Sing_\bullet(X)$
is also  symmetric monoidal by the Eilenberg-Zilber map, and hence the same is true for
 the composite $\Sing_*u$ defined
on semi-algebraic sets.

Let $\calT$ be a symmetric monoidal category.
For us, a contravariant functor
\[F\colon\calT\longrightarrow\CDGA
\]
is  \emph{symmetric monoidal} if it is equipped with a natural map
\begin{equation}\label{eq:kappaF}
\kappa\colon F(X)\otimes F(Y)\longrightarrow F(X\times Y)
\end{equation}
satisfying the usual axioms and  such that
$F(\unit_\calT)=\BK$.

Thus the functor $\ompa^*$ is  symmetric monoidal through the K\"unneth quasi-isomophism
\begin{align*}
\kappa\colon\ompa^*(X)\otimes\ompa^*(Y) & \stackrel{\simeq}\longrightarrow\ompa^*(X\times Y) \\
\alpha\otimes\beta & \longmapsto\proj_X^*(\alpha)\proj^*_Y(\beta).
\end{align*}
and the same is true for $\Apl$.

A  \emph{symmetric  monoidal natural transformation} is a natural transformation between
two (covariant or contravariant) functors that commutes with all the monoidal and symmetry
structure maps. It is a 
\emph{symmetric  monoidal natural equivalence} if it induces weak equivalences
in an obvious sense (for example quasi-isomorphism if the target category
is chain complexes or CDGA). Two symmetric monoidal functors are \emph{weakly equivalent} if they are
connected by a chain of symmetric  monoidal natural equivalences.

\begin{thm}\label{T:monompaapl}
On \zz{this statement was upgraded from proposition to theorem
            (because of reference from formailty paper and also because it is a
            main statement
          } 
the category of compact semi-algebraic sets, the symmetric monoidal contravariant functors
$\ompa$ and $\Apl(u(-);\BR)$ are weakly equivalent.
\end{thm}
\begin{proof}
It is easy to see that the weak equivalences from equation \refN{E:mainzz} before \refP{critqi}
are all symmetric monoidal.
\end{proof}

\fn{\pl 10july2008 As Victor suggested (and he was  correct, as always :-) )
there is a simpler proof giving a direct symmetric morphism
from $\Sing_*u$ to $\Ch_*$. I  changed this accordingly}
\begin{prop}
\label{P:monCHSing}
Symmetric monoidal functors $\Ch_*$ and $\Sing_*u$ are weakly equivalent.
\end{prop}
\begin{proof}
We have the chain complex $\Spa_*(X)$ of semi-algebraic singular chains
defined as the normalized chain complex associated to the simplicial set $\Spa_\bullet(X)$,
and we have seen in the proof of \refT{main} that there is a natural weak equivalence
\[\Spa_*(X)\quism\Sing_*(X).
\]
It is clear that this equivalence is symmetric monoidal.

We also have a natural map
\begin{align*}
\Spa_*(X)&\longrightarrow\Ch_*(X)\\
(\sigma\colon\Delta^k\to X)&\longmapsto \sigma_*(\bbr{\Delta^k})
\end{align*}
which is also symmetric monoidal.
It is a quasi-isomorphism by another application of \refP{critqi}.
\end{proof}

For a  real vector space $V$, denote its dual by
\[V^\vee:=\hom(V,\BR).\]
Thus $(\ompa(X))^\vee$ is a chain complex. There is an evaluation chain map
\[\ev\colon\Ch_*(X)\otimes\BR\to(\ompa(X))^\vee
\]
defined by, for $\gamma\in\Ch_k(X)$,
\[\ev(\gamma)\colon\ompa^k(X)\longrightarrow\BR\quad,\quad\alpha\longmapsto\langle\alpha,\gamma\rangle.\]

The proof of the following proposition is straightforward.
\begin{prop}\label{P:evalmostmon}
The evaluation map $\ev$ is a natural quasi-isomorphism and the following diagram commutes:
\[
\xymatrix{
\Ch_*(X)\otimes\Ch_*(Y)\otimes\BR
\ar[r]^-{\ev\otimes\ev}_-{\simeq}
\ar[dd]_-{\times}
&
(\ompa(X))^\vee\otimes(\ompa(Y))^\vee
\ar[d]_-{\simeq}
\\
&
\left(\ompa(X)\otimes\ompa(Y)\right)^\vee
\\
\Ch_*(X\times Y)\otimes\BR
\ar[r]^-{\ev}_-{\simeq}
&
\ar[u]^-{\simeq}_-{(\times)^\vee}
(\ompa(X\times Y))^\vee.
}
\]
\end{prop}
\section{Oriented semi-algebraic bundles and integration along the fiber}\label{S:SAbundle}

A
\fn{\pl 11 july 2008: this section is new}
natural occurence of a strongly continuous chain $\Phi\in\Chs_k(E\stackrel{p}\to B)$ is when $p$
is a semi-algebraic bundle with fiber a compact  oriented $k$-dimensional semi-algebraic manifold.
The operation $\fint_\Phi$ is then a generalization of the classical integration along the fiber.
In this section we define and study this notion. Many properties here will be useful in the proof 
of the formality of the little cubes operad \cite{LaVo:for}.

First we define  a locally trivial bundle in the category of semi-algebraic sets in an obvious way.
\begin{definition}
\label{D:PAfib}\ \ 
\begin{itemize}
\item An \emph{SA bundle} is a semi-algebraic map $p\colon E\to B$ such that there exists a semi-algebraic set $F$,
a covering of $B$
by a finite family  $\{U_\alpha\}_{\alpha\in I}$ of  semi-algebraic open sets, and semi-algebraic homeomorphisms
$h_\alpha\colon U_\alpha\times F\iso p^{-1}(U_\alpha)$ such that $ph_\alpha=\proj_1$.

\item The space $F$ is called a \emph{generic fiber}.
\item An SA bundle is \emph{oriented}  if $F$ is a compact oriented manifold and each fiber $p^{-1}(b)$
is oriented in such a way that the restriction of $h_\alpha$ to the fiber respects the orientation.
\item The \emph{fiberwise boundary} of an oriented SA bundle is the oriented SA bundle
\[p^\partial\colon E^\partial \longrightarrow B\]
where
\[E^\partial=\cup_{b\in B}\partial(p^{-1}(b))\]
and $p^\partial$ is the restriction of $p$ to $E^\partial$.
\end{itemize}
\end{definition}
Note that because we work in the category of semi-algebraic sets, asking that the covering is finite is not a restriction.

It is clear that the fiberwise boundary of an oriented SA bundle of dimension $k$ is an oriented 
SA bundle of dimension $k-1$.
 
\begin{prop}
\label{P:fibch}
If $p\colon E\to B$ is an oriented SA bundle then the formula
$\Phi(b)=\bbr{p^{-1}(b)}$, for $b\in B$, defines a strongly continuous chain
$\Phi\in \Chs_l(E\to B)$ that we call the \emph{continuous chain associated to the oriented SA bundle}.
Further, $\partial \Phi$
 is the continuous chain  associated to the fiberwise boundary $p^\partial\colon E^\partial \to B$.
\end{prop}
\begin{proof}
Take an open semi-algebraic  cover of $X$ that trivializes the bundle and take a stratification of $B$ such that the closure
of each stratum is contained in an open set of the cover. It is straightforward to construct
 a trivialization of $\Phi$
associated to this stratification from the trivialization of the bundle. The statement about the boundary is also straightforward.
\end{proof}

\begin{definition}\label{D:integrfiber}
Let $p\colon E\to B$ be an oriented SA bundle with fiber of dimension $k$. The \emph{pushforward} or \emph{integration along the fiber}
operation is the   degree $-k$ linear map
\[p_*\colon\omin^{*+k}(E)\longrightarrow\ompa^*(B)
\]
defined by
$p_*(\mu)=\fint_\Phi\mu$, where $\Phi$ is the continuous chain associated to the oriented SA bundle $p$.
\end{definition}

This operation is very important and we now develop some of its properties.

\subsection{Properties of SA bundles}
We first have the obvious
\begin{prop}\label{P:PBSAbundle}
The pullback of an (oriented) SA bundle along a semi-algebraic map
is an (oriented) SA bundle with the same fibers.
\end{prop}

Less obvious is 
\begin{prop}\label{P:compositebdl}
The composite of two SA bundles $p\colon E\to B$ and $q\colon B\to X$ is an SA bundle. 
Moreover, if $p$ and $q$ are oriented  then so is $q\circ p$ and
\[\dim(\fiber(q\circ p))=\dim(\fiber(p))+\dim(\fiber(q)).\]
\end{prop}
\begin{rem}Note that it is not true in general that the composite of two bundles is a bundle.
Here is a counterexample for coverings (topological bundles with discrete fibers):
Let $X$ be the \emph{Hawaian rings} which is the union of circles  $\Gamma_n$ in the plane centered at $(0,1/n)$ with radius $1/n$  for $n\geq1$.
For $N\geq 1$, let $q_N\colon E_N\to X\times\{N\}$  be a two-fold covering such that the restriction
of $q_N$ to $\Gamma_N$ is the non trivial covering and its restriction to $\Gamma_k$  is trivial
for $k\not=N$.
Set $q=\amalg_{N\geq1}q_N$, $E=\amalg_{N\geq1}E_N$, and let
\[q\colon E\longrightarrow \amalg_{N\geq1}X\cong X\times\BN_0.\]
This map is clearly a covering. The projection $p\colon X\times\BN_0\to X$ is  another covering.
But the composite $p\circ q$ is not a covering because it is not trivial on any neighborhood of
 the origin.
 \end{rem}

To prove \refP{compositebdl}, we will first need two lemmas.
Let $J=[0,1]^{a}\times[0,1)^b$ for some integers $a,b\geq0$ and let
\[
p\colon E\longrightarrow B\times J\]
be an SA bundle. Our first task will be to prove that the composite of $p$ with the projection on the
second factor $J$ is an SA bundle.

\begin{lemma}\label{L:rHusemoller}
Let \[r\colon B\times J\to B\times J\]
be the map defined by $r(b,u)=(b,0)$. Then there exists a semi-algebraic map
\[\hat r\colon E\to E\]
such that $p\circ\hat{r}=r\circ p$.
\end{lemma}
\begin{proof}
Since the bundle is semi-algebraic, there exists a finite, and hence enumerable, covering of $B\times J$ which
trivializes $p$. The lemma is then proved by an argument completely analogous to that of
\cite[Theorem 9.6 in Chapter 4, page 51]{Hus:FB}.
\end{proof}
\begin{lemma}\label{L:p2pSA}
Let $p\colon E\to B\times J$ be an SA bundle.
Then the composite
\[E\stackrel{p}\longrightarrow B\times J\stackrel{\proj_2}\longrightarrow J\]
is also an SA bundle.
\end{lemma}
\begin{proof}
Consider the inclusion $r_0\colon B\to B\times J$
defined by $r_0(b)=(b,0)$ and the projection on the first factor
$\proj_1\colon B\times J \to J$. Then $r_0\circ\proj_1$ 
is the map $r$ of \refL{rHusemoller}.
Let $E_0=p^{-1}(B\times \{0\})\subset E$.  This defines a bundle $p_0=p|E_0$ over $B\cong B\times\{0\}$.
The pullback of  $p_0$ along $\proj_1\colon B\times J\to B$
is the bundle $(p_0\times\id)\colon E_0\times J\to B\times J$.
By \refL{rHusemoller} and by universal property of the pullback we get a semi-algebraic
 homeomorphism
\[h\colon E\iso E_0\times J\]
such that $p=(p_0\times \id)\circ h$.
Let $\proj_2\colon B\times J\to J$ be the projection on the second factor.
The composite $\proj_2\circ(p_0\times \id)$ is a trivial SA bundle and,
 since $h$ is a homeomorphism, we deduce that
$pr_2\circ(p_0\times \id)\circ h=\proj_2\circ p$ is also an SA bundle.
\end{proof}
We are now ready for
\begin{proof}[Proof of \refP{compositebdl}]
Take a semi-algebraic covering of $X$ trivializing $q$.
By \cite[Theorem 2]{Har:loc} this open covering can be refined into a 
semi-algebraic triangulation. 
Each $d$-simplex $\sigma$ is the union of exactly $d+1$ $d$-dimensional semi-open cubes
whose vertices are the barycenters of the various faces of $\sigma$.
This gives a finite covering $\{J_\alpha\}$ of $X$. By  triviality of $q$ over $\sigma$,
we have semi-algebraic homeomorphisms
\[g_\alpha\colon J_\alpha\times F\iso q^{-1}(J_\alpha)
\]
such that $q\circ g_\alpha=\proj_1$.
To prove that $q\circ p$ is an SA bundle, it is enough to prove that its restriction
over $J_\alpha$ is one. Using the homeomorphism $g_\alpha$ over $J_\alpha$, this is a consequence of \refL{p2pSA}.
\end{proof}

\subsection{Properties of integration along the fiber}\label{S:propintfiber}
We first study naturality properties of integration along the fiber.
\begin{prop}\label{P:pushPB}
Consider a pullback diagram
\[
\xymatrix{
P
\ar[r]^{\hat f}
\ar[d]_{\hat\pi}
\ar@{}[rd]|-{\text{{pullback}}}
&
E
\ar[d]^{\pi}
\\
X
\ar[r]_{f}
&
B
}
\]
of semi-algebraic sets where $\pi$ is an oriented SA bundle.
For $\mu\in\omin(E)$,
\[f^*(\pi_*(\mu))=\hat\pi_*(\hat f^*(\mu)).\]
\end{prop}
\begin{proof}
It is clear that if $\Phi$ is the continuous chain associated to the oriented SA bundle $\pi$,
then $f^*(\Phi)$ is the continuous chain associated to $\widehat\pi$ (see \refD{pb}).
The formula is then straighforward.
\end{proof}

The following is an extension of the naturality. 
It establishes that two SA bundles over the same base with fibers defining the same semi-algebraic
{\current} have the same pushforward operation. This could occur for example when the two fibers are
compact manifolds that differ only by some codimension $1$
subspace.
\begin{prop}\label{P:natpush}
Let $\pi\colon E\to B$ be an oriented bundle of dimension $k$ and $\rho\colon E'\to E$
be a semi-algebraic map such that $\pi':=\pi\circ\rho$ is also an oriented bundle of dimension
$k$, as in the following diagram \zz{added the diagram}
\[
\xymatrix{E'\ar[rr]^-{\rho}\ar[rd]_{\pi'}&&E\ar[ld]^-{\pi}\\&B.}
\]
Suppose that there is a sign $\epsilon=\pm1$ such that for each $b\in B$,
\[\rho_*(\bbr{\pi'^{-1}(b)})=\epsilon\cdot\bbr{\pi^{-1}(b)}.\]
Then for any minimal form $\mu\in\omin^*(E)$,
\begin{equation}\label{E:natpush}
\pi'_*(\rho^*(\mu))=\epsilon\cdot\pi_*(\mu).
\end{equation}
\end{prop}
\begin{proof}
We need to prove that for $\gamma\in \Ch_l(B)$,
evaluation for two sides of \refE{natpush} is the same.
Take a triangulation of $B$ that trivializes $\pi'$ and that is compatible with $\gamma$.
By linearity, we can restrict to a single simplex and can suppose that $\gamma=\bbr{B}$
and $\pi':E'=F'\times B\to B$ is the projection.
Then by the hypothesis, 
\[\rho\colon  F'\times B\longrightarrow E\]
is a representative of the continuous chain $\Phi$
 associated to the oriented SA bundle $\pi$ (up to the sign $\epsilon$).
The proposition now follows by unravelling definitions.
\end{proof}

We have \zz{
                    the statement below on additivity and its proof are new; needed for the
                    formality paper
                    }
 an additivity formula:
\begin{prop}
Let $\pi\colon E\to B$ be an oriented SA bundle with $k$-dimensional
fiber.  Suppose that
\[E=\cup_{\lambda\in\Lambda}E_\lambda\]
where $\Lambda$ is a finite set, each $E_\lambda$ is a semi-algebraic
subset of $E$, and the restrictions
\[\pi_\lambda:=\pi|E_\lambda\colon E_\lambda\longrightarrow B
\]
are oriented SA bundles with $k$-dimensional
fiber, the orientation being induced by that of the fibers of $\pi$.

Suppose moreover that for $\lambda_1,\lambda_2$ distinct in $\Lambda$
and $b\in B$,
\[\dim(\pi_{\lambda_1}^{-1}(b)\cap \pi_{\lambda_2}^{-1}(b))\,<\,k.\]
Then for any $\mu\in\omin(E)$,
\[\pi_*(\mu)=\sum_{\lambda\in \Lambda}\pi_{\lambda*}(\mu|E_\lambda).\]
\end{prop}
\begin{proof}
Let 
\[\iota_\lambda\colon E_\lambda\hookrightarrow E\]
be the inclusions and set
\begin{eqnarray*}
E'&:=&\coprod_{\lambda\in\Lambda}E_\lambda\\
\rho&:=&(\iota_\lambda)_{\lambda\in\Lambda}\colon
E'\longrightarrow E\\
\pi'&:=&(\pi_\lambda)_{\lambda\in \Lambda}\colon E'\longrightarrow B.
\end{eqnarray*}
We then have a diagram as in \refP{natpush}.
The hypotheses imply that $\rho$ induces a degree $1$ map between the
fibres. Proposition \ref{P:natpush} implies that 
\[\pi_*(\mu)=\pi'_*(\rho^*(\mu))=\sum_{\lambda\in
  \Lambda}\pi_{\lambda*}(\iota^*_\lambda(\mu)).\]
\end{proof}

We have also the following \emph{fiberwise Stokes formula}.

\begin{prop}\label{P:dpush}
Let $\pi\colon E\to B$ be an oriented SA bundle with $k$-dimensional fiber and let
$\pi^\partial\colon E^\partial \to B$ be its fiberwise boundary.
For $\mu\in\omin(E)$,
\[d(\pi_*(\mu))=\pi_*(d(\mu))+(-1)^{\deg(\mu)-k}\pi_*^{\partial}(\mu|E^\partial).\]
\end{prop}
\begin{proof}
This is a direct consequence of \refL{deltaompa}.
\end{proof}

The following is a formula for a double pushforward.\fn{\pl We would have like to have
a general formula like for $E\
\stackrel{\pi_1}\to B\stackrel{\pi_2}\to X$ a composite of two oriented bundles, then
$\pi_{2*}\pi_{1*}=(\pi_2\pi_1)_*$ but there is a problem even to state it because
the intermediate form $\pi_{1*}(\mu)$ is not minimal, so we cannot pushforward it further.}

\begin{prop}\label{P:doublepush}
Let $\pi\colon E\rightarrow B$ be an oriented SA bundle. Let $N$ be an oriented compact semi-algebraic manifold,
and consider the projection
\[\proj_1\colon E\times N\longrightarrow E.\]
For $\mu\in\omin(E)$ and $\nu\in\omin(N)$,
\[(\pi\circ\proj_1)_*(\mu\times\nu)=\pi_*(\mu)\cdot\langle\nu,\bbr{N}\rangle.\]
\end{prop}
\begin{proof}
By naturality it is enough to prove that the two forms we are comparing evaluate to the same thing on $\bbr{B}$ when
$B$ is an oriented compact semi-algebraic manifold.
In that case\fn{\pl the sign in equation (5) of Ptop 5.10 in RHT-SAS need to be corrected. DONE! 11july}
\begin{eqnarray*}
\langle(\pi\circ\proj_1)_*(\mu\times\nu)\,,\,\bbr{B}\rangle
&=&
\langle\mu\times\nu\,,\,\bbr{E\times N}\rangle\\
&\stackrel{\textrm{\refN{E:omincross} in \refP{omincross}}}{=}&
\langle\mu\,,\,\bbr{E}\rangle\cdot
\langle\nu\,,\,\bbr{N}\rangle\\
&=&
\langle\pi_*(\mu)\,,\,\bbr{B}\rangle\cdot
\langle\nu\,,\,\bbr{N}\rangle.
\end{eqnarray*}
\end{proof}

The following is useful in showing that certain integrals along the fiber vanish.
\begin{prop}\label{P:pushfactcodim}
Let $\pi\colon E\to B$ be an oriented SA bundle
that factors as $\pi=q\circ \rho$ with
\[E\stackrel{\rho}\longrightarrow Z\stackrel{q}{\longrightarrow} B.
\]
Suppose that for all  $b\in B$,
\[\dim(q^{-1}(b))<\dim(\pi^{-1}(b)).\]
Then for all minimal forms $\mu\in\omin(Z)$, we have $\pi_*(\rho^*(\mu))=0$.
\end{prop}
\begin{proof}
We need to prove that for $\gamma\in\Ch_l(B)$,
\[\langle\pi_*(\rho^*(\mu))\,,\,\gamma\rangle=0.
\]
Passing to a suitable triangulation of $B$, by naturality it is enough
 to prove this when $B$ is a compact oriented semi-algebraic manifold, $\gamma=\bbr{B}$, and
$\pi$ is trivial over $B$ with $E=F\times B$ for a compact oriented manifold $F$.
From the hypothesis we get that
\[\dim(Z)\leq\dim B+\sup_{b\in B}(\dim q^{-1}(b))<\dim (B\times F).\]
For degree reasons, $\rho^*(\mu)=0$ in $\omin(F\times{B})$.
\end{proof}

Our last proposition concerns multiplicative properties of the pushforward.
\begin{prop}\label{P:prodpushPB}
Consider a pullback diagram
\[
\xymatrix{
P\ar[r]^{q_2}\ar[d]_{q_1}
\ar@{}[rd]|-{\textrm{{pullback}}}
&E_2\ar[d]^{\pi_2}\\
E_1\ar[r]_{\pi_1}&B
}
\]
of semi-algebraic sets. Assume that $\pi_1$ and $\pi_2$ are oriented SA bundles
and set $\pi=\pi_i\circ q_i$.
Let $\mu_i\in\omin(E_i)$ be a minimal form. Then $\pi$ is an oriented SA bundle and
\[
\pi_{1*}(\mu_1)\wedge\pi_{2*}(\mu_2)=\pi_*(q_1^*(\mu_1)\wedge q_2^*(\mu_2)).\]
\end{prop}
\begin{proof}
By \refP{compositebdl}, 
$\pi$ is an oriented SA bundle.
We have a pullback
\[
\xymatrix{
P
\ar[r]^{\pi}
\ar[d]_{\hat\Delta}
&
B
\ar[d]^{\Delta}\\
E_1\times E_2\ar[r]_{\pi_1\times\pi_2}&B\times B
}
\]
with $\hat\Delta=(q_1\times q_2)\circ\Delta_P$, where $\Delta_P\colon P\to P\times P$
is the diagonal map.

Using \refP{pushPB}, we have
\begin{eqnarray*}
\pi_{1*}(\mu_1)\wedge\pi_{2*}(\mu_2)
&=&
\Delta^*(\pi_{1*}(\mu_1)\times\pi_{2*}(\mu_2))
\\
&=&
\Delta^*((\pi_{1}\times\pi_2)_*(\mu_1\times\mu_2))
\\
&=&\pi_*(\hat\Delta^*(\mu_1\times\mu_2))
\\
&=&\pi_*(\Delta_P^*((q_1\times q_2)^*(\mu_1\times\mu_2))
\\
&=&\pi_*(q_1^*(\mu_1)\wedge q_2^*(\mu_2)).
\end{eqnarray*}
\end{proof}

An interesting example of an oriented SA bundle which is not smooth is the
  projection $\pi\colon C[n+q]\to C[n]$ between
 Fulton-MacPherson compactifications of configurations spaces in $\BR^N$. See \cite{LaVo:for} for details.

\section{Differences between this paper and  \cite[Appendix A]{KoSo:def}}\label{S:diffKS}


The present paper is to a large extent a development of the ideas in \cite[Appendix 8]{KoSo:def}.
However, we have not succeeded in going as far as Kontsevich and Soibleman suggest might be possible.
We discuss this in detail in the present section.


\subsection{The scope of the equivalence $\ompa^*\simeq\Apl(-;\BR)$}


It is suggested in \cite[Appendix 8]{KoSo:def} that the weak equivalence $\ompa^*(X)\simeq\Apl(X;\BR)$ holds
for all semi-algebraic sets, and even for a larger class of spaces called \emph{PA spaces} (this is why our forms are called PA forms, following the notation in \cite[Appendix 8]{KoSo:def}).
Our proof, however, only holds for compact semi-algebraic sets. It seems reasonable though that a modification of the proof
of \refP{monCHSing} might produce the desired equivalence for noncompacts sets as well. The problem is that we require polyhedra to be collapsible in the Poincar\'e lemma, but such a notion does not seem to exist  for semi-open simplices.


\subsection{Strongly and weakly continuous families of chains}


The definition of a continuous family of chains in \cite[Definition 22]{KoSo:def} is different from ours (\refD{scc}). The definition in that paper is probably equivalent to the following definition of weakly continuous family of chains
 (we say ``probably" because the condition (d) in \cite[Definition 22]{KoSo:def} is somewhat informal and we have interpreted it as the condition (wc) below).
\begin{definition}
\label{D:wss}
Let $f\colon Y\to X$ be a semi-algebraic map.
A \emph{weakly continuous family of chains} or, shortly, a \emph{weakly continuous chain} of dimension $l$ over $X$ along $f$ is a map
\[\Phi\colon X\longrightarrow\Ch_l(Y)\]
such that there exist
\begin{enumerate}
\item a finite semi-algebraic stratification $\{S_\alpha\}_{\alpha\in I}$ of $X$, and, for each $\alpha\in I$,
\item an oriented compact  semi-algebraic manifold $F_\alpha$ of dimension $l$ and
\item a semi algebraic map $g_\alpha \colon{S_\alpha}\times F_\alpha\to Y$
satisfying
\begin{enumerate}
\item the diagram

\[\xymatrix{{S_\alpha}\times F_\alpha\ar[r]^{g_\alpha}\ar[d]_{\proj}&Y\ar[d]^f\\
{S_\alpha}\ar@{^(->}[r]&X}
\]
commutes, and
\item for each $\alpha\in I$ and $x\in{S_\alpha}$,
$\Phi(x)=g_{\alpha*}(\bbr{\{x\}\times F_\alpha})$, 

\end{enumerate}
\end{enumerate}
and such that
\[
\tag{wc} \textrm{if } x_n\to x \textrm{ in } X  \textrm{ then } \Phi(x_n)\rightharpoonup \Phi(x)\textrm{ in }C_k(Y).
\]

\end{definition}

Notice that if $\Phi$ is a strongly continuous chain then it is easy to show that for any
semi-algebraic curve $\theta\colon[0,1]\to X$ and for any sequence $\epsilon_n$ in $[0,1]$ converging to $0$,
we have $\Phi(\theta(\epsilon_n))\rightharpoonup \Phi(\theta(0))$.
So if we had required that the sequence $(x_n)$ live on a semi-algebraic curve in the (wc) condition,
then strongly continuous would imply weakly continuous.  But with the present definition, we do not know whether
such an implication is true, although this seems likely.

We now give two examples of  weakly continuous chains that are not strongly continuous.
\begin{example}
Set $X=\{(u,v)\in\BR^2:0\leq u\leq v\leq 1\}$, $Y=X\times [0,1]^2$, and let $f\colon Y\to X$ be the projection on
the first factor. Define $\Psi\colon X\to\Ch_0(Y)$ by $\Psi(0,0)=0$ and, for $(u,v)\in X$ with $v>0$,
\[\Psi(u,v)=\bbr{\{(u,v,u/v)\}}\times(\bbr{\{0\}}-\bbr{\{uv\}}).\]
This parametrized family is trivialized under the stratification $\{\{(0,0)\},X\setminus\{(0,0)\}\}$.
Using the flat norm we have that $\flatnorm(\Psi(u,v))\leq uv$ and this implies (wc).
Therefore $\Psi$ is a weakly continuous chain.  To see that it is not strongly continuous, consider the \emph{limit set} of the supports of $\Psi(u,v)$ as $(u,v)\to (0,0)$, by which we mean the set
\[\cap_{\epsilon>0}\left(\overline{\cup_{\max(|u|,|v|)<\epsilon}\spt\Psi(u,v)}\right).\]
Indeed this limit set is  $\{(0,0)\}\times[0,1]\times\{0\}$  which is one-dimensional, and this is impossible
in the case of a strongly continuous $0$-chain.
\end{example}

\begin{example}
Set $X$, $Y$, and $f\colon Y\to X$ to be the same as in the previous example. Define $\Phi\colon X\to\Ch_1(Y)$ by $\Phi(0,0)=0$ and, for $(u,v)\in X$ with $v>0$,
\[\Phi(u,v)=\bbr{\{(u,v)\}\times[0,u/v]}\times(\bbr{\{0\}}-\bbr{\{uv\}}).\]
The trivializing set $\{\{(0,0)\},X\setminus\{(0,0)\}\}$, same as before. Similarly we have
  $\flatnorm(\Phi(u,v))\leq 3uv$ which implies (wc). But $\Phi$ is not strongly continuous since its boundary  is not. Indeed, $\partial(\Phi)$ is a sum of a strongly continuous chain and the chain $\Psi$ from the previous example.
\end{example}

As will appear later in the discussion, strongly continuous chains have certain disadvantages and the reader might wonder why
we did not use weakly continuous chain instead. The problem with weakly continuous chains
was the proof of the Leibniz formula
\[\partial(\gamma\ltimes\Phi)=\partial(\gamma)\ltimes\Phi\pm\gamma\ltimes \partial(\Phi)\]
which is needed in \refL{deltaompa} (proving that $\ompa^*$ is a cochain subcomplex).
Another problem with the condition (wc) is that, as demonstrated at the beginning of \refS{conv}, weak convergence
in even the flat semi-norms is not well adapted in the study of semi-algebraic {\current}s.
In \refE{limPASA} of the proof of \refL{contcodimPA} (continuity principle for PA forms),
we used strong continuity property to deduce from $\gamma_n\toSA\gamma$ that
$\gamma_n\ltimes\Phi\toSA\gamma\ltimes\Phi$ (in the special case of a constant continuous chain $\Phi=\fib{F})$.
 If we had worked with
weakly contimuous chains, we would first have had to prove from (wc)
that  $\gamma_n\ltimes\Phi\toSA\gamma\ltimes\Phi$, which seems nontrivial.  Alternatively, we would have needed to show 
that  $\gamma_n\ltimes\Phi \rightharpoonup\gamma\ltimes\Phi$ and 
to deduce that $\langle\fint_\Phi\mu,\gamma_n\rangle\to\langle\fint_\Phi\mu,\gamma\rangle$
without SA convergence, and hence without using \refP{contSAmin}.  This also seems quite difficult
because of the example at the beginning of \refS{conv}.

Kontsevich suggested to us  another possible
definition of continuity which could perhaps be used instead of (wc).
First, it should be enough to check continuity by restricting the parametrized family
to any semi-algebraic curve in $X$ (although the proof of Leibniz formula could now be problematic), so we now restrict our attention to the case of semi-algebraic maps $f\colon Y\to [0,1]$.
A family of $k$-chains in $Y$ parametrized by $[0,1]$ gives a $k+1$-dimensional smooth
oriented submanifold $M_0\subset Y$, with multiplicities on each connected component, such that
$f|_{M_0}$ is a submersion over all but a finite number of points $[0,1]$ and such that
over a regular value $t\in[0,1]$, $\Phi(t)=\bbr{M_0\cap f^{-1}(t)}$.
There is then a $k+1$-current $T=\bbr{M_0}\in \Ch_{k+1}(Y)$.
A continuity condition could then be expressed as
\[
\forall t\in [0,1]:\,\dim(\spt(T)\cap f^{-1}(t))\leq k\textrm{ and }
\dim(\spt(\partial T)\cap f^{-1}(t))\leq k-1.
\]


\subsection{Complexity of the proof of the Poincar\'e lemma}\label{S:complexityPoincare}


The reader might be surprised by the complexity of our proof of the Poincar\'e lemma \ref{L:Poincare} and
in particular by our need for the Hauptvermutung for semi-algebraic spaces.
 Let us explain why following an obvious and more natural line of proof does not work.
 The simplest idea would be to convert a semi-algebraic homotopy $h\colon X\times I\to X$, $I=[0,1]$, into a homotopy operator
 \[\Theta_h\colon\ompa^*(X)\longrightarrow\ompa^{*-1}(X)\]
 such that $\delta\Theta_h\pm\Theta_h\delta=(h(-,1))^*-(h(-,0))^*$.
To try to construct such a $\Theta_h$, first notice that given a continuous chain $\Phi\in\Chs_l(Y\to X)$ and any
 cochain $\lambda\in\Ch^{k+l}(Y)$ we can define a cochain $\fint_\Phi\lambda\in\Ch^k(X)$ by 
 \[\langle\fint_\Phi\lambda,\gamma\rangle=\langle\lambda,\gamma\ltimes \Phi\rangle\]
  (even if $\lambda$ is not a minimal form as in \refE{Lambda}.) For a PA form $\alpha=\fint_\Phi\mu\in\ompa^k(X)$
  with $\mu\in\omin^{k+l}(X)$, we would like to define the cochain $\Theta_h(\alpha)$ through the equations
 \begin{eqnarray}
\label{E:Thetah}
 \langle\Theta_h(\fint_\Phi\mu)\,,\,\gamma\rangle&=&
 \langle\fint_\Phi\mu\,,\,h_*(\gamma\times\bbr{I})\rangle
 \\\notag&=&
  \langle h^*(\fint_\Phi\mu)\,,\,\gamma\times\bbr{I}\rangle
   \\\notag&=&
  \langle \fint_{h^*(\Phi)}h^*(\mu)\,,\,\gamma\times\bbr{I}\rangle
   \\\notag&=&
  \langle\fint_{\fib{I}}\left( {\fint_{h^*(\Phi)}}h^*(\mu)\right)\,,\,\gamma\rangle
 \end{eqnarray}
 where $\fib{I}$ is the constant chain in $\Chs_1(X\times I\stackrel{\proj}\to X)$.

 Suppose we knew that the operation $\ltimes$ defined in \refP{ltimes} could be extended to a natural operation
 \begin{equation}\label{E:ExtendProduct}
 \ltimes\colon\Chs_p(X\to T)\otimes\Chs_l(Y\to X)\longrightarrow\Chs_{p+l}(Y\to T)
 \end{equation}
by the formula
 \[(\Gamma\ltimes\Phi)(t)=(\Gamma(t))\ltimes\Phi,\]
 where $\Gamma\in   \Chs_p(X\to T)$, $\Phi\in\Chs_l(Y\to X)$, and $t\in T$.
 In this case we could set
 \[\Theta_h(\fint_\Phi\mu):=\fint_{\fib{I}\ltimes h^*\Phi}h^*(\mu)\in\ompa^{k-1}(X)\]
 and equations \refN{E:Thetah}  would show that $\Theta_h$ is an algebraic homotopy between $h(-,0))^*$ and $(h(-,1))^*$.
 This would of course imply the Poincar\'e lemma.

 The problem 
 is that the twisted product $\Gamma\ltimes\Phi$ is in general not a strongly continuous chain, as
 we can see in the following
 \begin{example}
Let $X$ be the convex hull of $\{(0,0,0),(1,0,0),(1,1,0),(0,0,2)\}$ in $\BR^3$ (a closed $3$-simplex)
and let $t_1,t_2,t_3$ be coordinates in $\BR^3$.
Consider the surface $$S=\{(t_1,t_2,t_3)\in X:t_2=t_1t_3\},$$ and the subspace below it and lying in $X$,
$$E=\{(t_1,t_2,t_3)\in X:t_2\geq t_1t_3\}.$$ Set $Y=E\times I$ and $f\colon Y\stackrel{\proj}\to E
\hookrightarrow X$. For $x\in X$ consider the distance
function\zz{typo corrected below} to $S$, $d\colon X\to I$, defined by
$d(x)=\textrm{dist}(x,S)$. Define a continuous $0$-chain $\Phi\in \Chs_0(Y\to X)$ by $\Phi(x)=\bbr{\{x\}}\times
(\bbr{\{0\}}-\bbr{\{d(x)\}})$ if
$x\in E$ and $\Phi(x)=0$ otherwise. This is a strongly continuous chain trivialized by the stratification
 $\{E,{X\setminus E}\}$.

Define $H\colon X\times I\to X$ such that $H(-,t)$ is a dilation of center $(0,0,2)$ and coefficient $t\in [0,1]$. Consider the pullback
\[
\xymatrix{
\hat Y\ar[r]\ar[d]_{\hat f}\ar@{}[rd]|{pullback}&Y\ar[d]^f\\X\times I\ar[r]_H&X.
}\]
We thus have the continuous chains $H^*(\Phi)\in \Chs_0(\hat Y\to X\times I)$
and  $\fib{I}\in\Chs_1(X\times I\stackrel{\proj_1}\to X)$.

However, $\fib{I}\ltimes H^*(\Phi)$ does not define a strongly continuous chain. To see this, 
first notice that the preimage of $(t_1,t_2,0)\in X$ with $0< t_2\leq t_1\leq 1$ under $H\circ\hat f$ is homeomorphic to
\begin{equation}
\left\{ (t_1,t_2,0,t_4,t_5)\, \left|\, 1-\frac{t_2}{2t_1}\leq t_4\leq 1\text{ and } 0\leq t_5\leq 1 \right.\right\},
\label{Eq:ex-support}
\end{equation}
where $t_4$ is the last coordinate of $X\times I$, and $t_5$ is the last coordinate of $Y=E\times I$. The chain  $\fib{I}\ltimes H^*(\Phi)$ at $(t_1,t_2,0)$ is a difference $I_0-I_1$ where $I_0$ and $I_1$ are given by
\begin{eqnarray}
I_0,I_1\colon [1-\frac{t_2}{2t_1},1]&\longrightarrow &(H\circ\hat f)^{-1}(t_1,t_2,0)\label{Eq:ex-chain}\\
\notag I_0\colon t_4& \longmapsto& (t_1,t_2,0,t_4,0)\\
\notag I_1\colon t_4& \longmapsto& (t_1,t_2,0,t_4,d(t_1t_4,t_2t_4,2-2t_4))
\end{eqnarray}
For a fixed $\lambda\in(0,1]$ it is easy to see that $\spt\left((\hat{I}\ltimes H^*(\Phi))(t_1,\lambda\cdot t_1,0)\right)$
tends to the limit set $\{(0,0,0)\}\times[1-\lambda/2,1]\times\{0\}$ as $t_1\to0_+$.

Now assume that $\fib{I}\ltimes H^*(\Phi)$ is strongly trivialized by  $\{S_\alpha,F_\alpha,g_\alpha\}_{\alpha\in J}$. We will say that a curve in $X$ passing through $(0,0,0)$ is \emph{locally in $S_\alpha$} if $\overline{S_\alpha}$ contains a neighborhood of $(0,0,0)$ in that curve. Obviously any such curve should be locally in at least one $S_\alpha$. Let us show that, for any $0<a<b<1$, the curves $\{(t,at,0)\, |\, 0\leq t\leq 1\}$ and $\{(t,bt,0)\, |\, 0\leq t\leq 1\}$ cannot be locally in the same $S_\alpha$.  This would then give a necessary contradiction (since there are infinitely many such curves, two of them must be locally in the same stratum). Applying~\eqref{Eq:ex-support} to the points from the first curve and taking the limit set as $t\to 0$, we get
$$
g_\alpha\left(\{(0,0,0)\}\times F_\alpha\right)\subset\{(0,0,0)\}\times[1-\frac a2,1]\times[0,1].
$$
But applying~\eqref{Eq:ex-chain} to the points from the second curve and again taking the limit set as $t\to 0$ we get
$$
g_\alpha\left(\{(0,0,0)\}\times F_\alpha\right)\supset\{(0,0,0)\}\times[1-\frac b2,1]\times\{0\}.
$$
Since $a<b$ these two conditions can not be satisfied simultaneously.
\fn{{\bf Victor:} The proof is given, but I am unhappy with this example, because the chain $\fib{I}\ltimes H^*(\Phi)$ is mapped to a strongly continuous chain under the natural inclusion $\hat Y\hookrightarrow X\times I^2$. (Which means that the integration a long $\fib{I}\ltimes H^*(\Phi)$  always produces a PA form.) Basically it is so because it can be globally trivialized. To see that notice that the image of $I$ under the inclusion $Y\hookrightarrow X\times I$ can be globally (with only one strata) trivialized by $\tilde\Phi(x)=\bbr{\{x\}}\times
(\bbr{\{0\}}-\bbr{\{\tilde d(x)\}})$, where $\tilde d(x)$ is the distance from $x$ to $X\setminus E$.
{\bf Pascal:} we keep this private comment for further thoughts but I do not think that we need to add anything about it in the current version.}
\end{example}

 We do not know whether the operation $\ltimes$ can be extended as in equation (\ref{E:ExtendProduct}) using weakly
 continuous chains.
 

 \subsection{Integration along the fiber of a PA form}
 
 
 At the end of \cite[Appendix 8]{KoSo:def}, Kontsevich and Soibelman claim that, given an oriented SA bundle $p\colon E\to B$
 with fiber of dimension $l$, there exists a pushforward operation\fn{pl 11 july 2008: corrected misprint $\ompa$ instead of $\omin$}
 \[p_*\colon\ompa^*(E)\longrightarrow\ompa^{*-l}(B)\]
 given by integration along the fiber. It is certainly true that there is a map
 \begin{align*}
 p_*\colon\omin^*(E) & \longrightarrow\ompa^{*-l}(B) \\
  \mu & \longmapsto\fint_\Phi\mu
  \end{align*}
 defined on minimal forms, where $\Phi\in\Chs_l(E\to B)$ is the continuous chain associated to this oriented SA bundle,
as we have seen in \refS{SAbundle}.\fn{\pl  11 july 2008 Has been rephrased}
However, we were unable to prove that there is an extension to all PA forms. A natural candidate for the pushforward
  $\int_\Psi\mu\in\ompa^k(X)$, where $\Psi\in\Chs_p(Y\to E)$ and $\mu$ a PA form in $\omin^{k+p}(Y)$,
 is
 \[p_*(\int_\Psi\mu)\,:=\,\int_\Phi\int_\Psi\mu=\int_{\Phi\ltimes\Psi}\mu\]
 but the problem as in the previous section is that it is unclear whether
 $\Phi\ltimes\Psi$ is a continuous chain and therefore whether $\int_\Phi\int_\Psi\mu$ is a PA form.

 One could try to avoid this problem by inductively defining $\ompa^*(X)$ by
 \begin{eqnarray*}
 {\ompa}^*_{(0)}(X)&:=&\omin^*(X);\\
 {\ompa}^*_{(p+1)}(X)&:=& \{\fint_\Psi\lambda:\Psi\in\Chs_*(Y\to X),\lambda\in{\ompa}^*_{(p)}(Y)\};\\
 {\ompa}^*_{(\infty)}(X)&:=&\cup_{p=0}^\infty {\ompa}^*_{(p)}(X).\\
 \end{eqnarray*}
In particular, $\ompa^*(X)= {\ompa}^*_{(1)}$. It is clear that one gets a well-defined pushforward operation on
${\ompa}^*_{(\infty)}$ this way and it can be proved exactly as for $\ompa$ that ${\ompa}^*_{(\infty)}$ is a cochain complex and satisfies the Poincar\'e lemma
and the other properties which makes it weakly equivalent to the cochain complex $\Apl(X;\BR)$.
The problem is to check that ${\ompa}^*_{(\infty)}$ is a commutative algebra. For this we would need a generalization of \refP{strKu-PA},
whose proof was based on the integral representation of \refL{repintPA}.  However, we were not able to adapt the proof of \refL{repintPA} to ${\ompa}^*_{(p)}$, even though it seems likely that elements of  ${\ompa}^*_{(p)}(X)$
are smooth on the complement of a codimension $1$ subset.


\subsection{Flat morphisms}


Also at the end of \cite{KoSo:def}, it is suggested that the pushforward operation exist for more general
maps than bundles. More precisely, 
given a proper semi-algebraic map
$f\colon Y\to X$
between oriented semi-algebraic manifolds such that, for each $x\in X$,
$$
\dim (f^{-1}\{x\})\ \le\ k-j,\ \ \text{where}\ j=\dim X\ \text{and}\ k=\dim Y\ ,
$$
one could construct a pushforward along $f$ as follows.
The notion of \emph{slicing} (\cite[4.3]{Fed:GMT},\cite[4.3]{Har:sli}, \cite[4.5]{Har:top}) then gives a continuous chain $\Phi\in \Chs_{k-j}(Y\stackrel{f}\to X)$ such that,
for each point $x$ in the open dense set of regular values, we have
$\Phi(x)=\bbr{f^{-1}\{x\}}$, i.e.~$\Phi(x)$ is the smooth fiber with the induced orientation. To
see this, let $\bbr{Y}$ denote the $k$ dimensional semi-algebraic {\current} given
by oriented integration over
$Y$. Then the correspondingly oriented graph $g_\#\bbr{Y}$, where $g(y)=(y,f(y))$,
is also a $k$ dimensional semi-algebraic {\current}.  Using the projections
$p(x,y)=y$ and $q(x,y)=x$ for $(y,x)\in Y\times X$, we have from \cite[4.3]{Har:sli}
and \cite[4.5]{Har:top} that the slice
$$
\langle\bbr{Y},f,x\rangle\ =\ p_\#\langle g_\#\bbr{Y},q,x\rangle
$$
is a $k-j$ dimensional semi-algebraic {\current}.  The map
$x\mapsto\langle\bbr{Y},f,x\rangle$ is also seen to be flat continuous, and hence weakly
continuous (here we are implicitly using the formula \cite[4.3.2 (6)]{Fed:GMT} which
shows that the slice is invariant under an oriented diffeomorphism of the range
and is thus well-defined as a map into the oriented manifold $X$).

While the slice $\langle\bbr{Y},f,x\rangle$ equals
the naturally oriented fiber $\bbr{f^{-1}\{x\}}$ for any regular value $x$ (by, for example,  \cite[4.3.8 (2)]{Fed:GMT}), one
only knows in general that $\spt\langle\bbr{Y},f,x\rangle\subset f^{-1}\{x\}$.  For
example, if $Y$ is the counter-clockwise oriented unit circle, $X=\BR$,
and $f(x,y)=x$, then
\[
\langle\bbr{Y},f,x\rangle\ =\
\begin{cases}
\bbr{(x,\sqrt{1-x^2})}-\bbr{(x,-\sqrt{1-x^2})},&\ \ \ -1<x<1;\\
0,&\ \ \ \textrm{otherwise,}
\end{cases}
\]
which is consistent with the weak continuity of the slice even at $x=\pm 1$.
Another elementary example is the formula
\[
\langle \bbr{\BC},h,0\rangle\ =\ \sum_{z\in h^{-1}\{0\} } n(h,z) \bbr{z} \ ,
\]
valid for any holomorphic $h\colon\BC\to\BC$, where $n(h,z)$ is the multiplicity as a zero of $h(z)$.

Applying the trivialization result of \cite{Har:loc} to the semi-algebraic map $f$,
one sees that $\Phi(x)=\langle\bbr{Y},f,x\rangle$ actually defines a strongly
continuous chain $\Phi\in \Chs_{k-j}(Y\stackrel{f}\to X)$.

Moreover, for any semi-algebraic chain $T\in \Ch_\ell(X)$ we have a {\it pullback
chain} $f^\#(T)\in \Ch_{\ell+k-j}(Y)$. This may be defined by noting that $T$
decomposes into a finite sum of chains of the form
$T_i=\langle\bbr{X_i},f_i,0\rangle$ for some open semi-algebraic subsets $X_i$ of
$X$ and semi-algebraic maps $f_i:X_i\to \BR^{j-\ell}$ with $\dim(X_i\cap
f_i^{-1}\{0\})\le\ell$ and $\dim(\fron(X_i)\cap f_i^{-1}\{0\})\le\ell-1$. Then the
pullback
$$
f^\#T\ =\ \sum_i\langle\bbr{Y\cap f^{-1}(X_i)},f_i\circ f,0\rangle\ \in\ \Ch_{\ell+k-j}(Y)
$$
is well-defined, independent of the decomposition.
Details and applications of this construction will be discussed in a later
paper.

\bibliographystyle{plain}

\def\cprime{$'$}

\end{document}